\documentclass[12pt,A4]{article}
\usepackage{pdfsync}
\usepackage{times}
\usepackage{amssymb,amsmath,amsfonts,amsthm,mathrsfs,graphicx, subfigure}
\usepackage{epsfig}
\usepackage{multimedia}
\usepackage{xspace}
\usepackage{wrapfig}
\usepackage{color}
\usepackage[french]{babel}
\usepackage[normalem ]{ulem}
\usepackage{soul}
\newtheorem{montheo}{Th\'eoreme}[section]
\newtheorem{moncor}[montheo]{Corollaire} 
\newtheorem{monlem}[montheo]{Lemme}
\newtheorem{mapropo}[montheo]{Proposition} 

\newtheorem{madef}[montheo]{D\'efinition} 
\newtheorem{maremarque}[montheo]{Remarque}

\author{ A. Benabdallah$^{1}$, M. Ben-Artzi$^2$\& Y. Dermenjian$^{3}$}
\begin{document}
\title{\textsf{\textbf{\Large 
    Concentration et confinement des fonctions propres dans un ouvert born\'e (version 2)
    %\footnote{fichier : concentration-confinement-reformatage-v2tertio-ter.tex}
}}}
\author{Assia Benabdallah$^{\dag}$, Matania Ben-Artzi$^{\ddag}$\,\& Yves Dermenjian$^{\dag}$,\\ 
\dag \,Aix Marseille Univ, CNRS, Centrale Marseille, I2M, Marseille, France\\
\ddag\, Hebrew University, Jerusalem, Isra\"{e}l\\
} 

\date{\today}
\maketitle
%%%%%%%%%%%%%%%%%%%%%%%%%%%%%%%%%%%%%%%%%%%%%%%%%%%%
%%%%%%%%%%%%%%%%%%%%%%%%%%%%%%%%%%%%%%%%%%%%%%%%%%%%
\begin{abstract}
Si $\omega$ et $\Omega, \omega\subset\Omega:=(0,L)\times(0,H),$ sont deux ouverts born\'es de $\mathbb{R}^{2}$, il est bien connu qu'il existe une constante $C_{\omega}$ telle que $0<C_{\omega}\leq R_{\omega}(u):=\frac{\Vert u\Vert^{2}_{L^{2}(\omega)}}{\Vert u\Vert^{2}_{L^{2}(\Omega)}}<\frac{\mbox{Vol}(\omega)}{\mbox{Vol}(\Omega)}$ pour toute fonction propre $u$ du Laplacien-Dirichlet $-\Delta$ sur $\Omega.$
 Ce r\'esultat n'\'etant plus exact pour un op\'erateur autoadjoint elliptique $A$ du deuxi\`{e}me ordre sans coefficients constants, plusieurs travaux l'ont consid\'er\'e dont r\'ecemment \cite{AFM:1}. On cr\'ee une partition $\mathfrak{F}_{NG} \cup\mathfrak{F}_{G}$ de l'ensemble des fonctions propres de l'op\'erateur $A$ : 
 \begin{itemize}
 \item les fonctions propres non guid\'ees i.e. $\forall \omega\not=\emptyset, \exists C_{\omega}>0$ tel que $R_{\omega}(u)>C_{\omega},$ si $u\in\mathfrak{F}_{NG},$
 \item les fonctions propres guid\'ees i.e. $\exists \omega\not=\emptyset,$ tel que $\inf\limits_{u\in\mathfrak{F}_G} R_{\omega}(u) = 0.$
 \end{itemize}
Entre autres choses, le papier caract\'erise spectralement ces deux ensembles pour des mod\`{e}les tr\`{e}s simples de milieux stratifi\'es, ce qui donne une condition suffisante, et parfois n\'ecessaire. La stratification permet le passage de la repr\'esentation habituelle du spectre de l'op\'erateur $A$, i.e. $\sigma(A)\subset\mathbb{R},$ \`{a} une repr\'esentation \`{a} deux indices d'o\`{u} une description g\'eom\'etrique des deux familles de fonctions propres. Dans la pr\'ec\'edente version, la section 4.1 \'etait incorrecte, elle est donc corrig\'ee, d'autres preuves sont simplifi\'ees et un r\'esultat plus g\'en\'eral ajout\'e.
 % ainsi que  d'autres r\'esultats reli\'es : interpr\'etation microlocale, mesure de d\'efaut de sous-suites, ...
 
 \centerline{\bf Abstract}
 Let $-\Delta$ be the Laplacian in $\Omega:=(0,L)\times(0,H),$ subject to Dirichlet boundary conditions and let $u$ be an eigenfunction of $-\Delta.$ For any open set $\omega\subset\Omega$ define $R_{\omega}(u) = \frac{\Vert u\Vert^{2}_{L^{2}(\omega)}}{\Vert u\Vert^{2}_{L^{2}(\Omega)}}.$ It is well known that  there exists a constant $C_{\omega}>0$ such that  $C_{\omega}\leq R_{\omega}(u) $ for all  eigenfunctions. This is no longer true for certain more general second-order elliptic  operators and many authors have considered this subject whose \cite{AFM:1} recently.  This work is concerned with such operators, occuring in ``layered media''. In this more general case the set of eigenfunctions is the disjoint union of two non-empty sets $\mathfrak{F}_{NG}\cup\mathfrak{F}_{G}$ as follows.
\begin{itemize} 
\item non-guided eigenfunctions: $\forall \omega\not=\emptyset$, any $u\in \mathfrak{F}_{NG}$ satisfies $R_{\omega}(u)>C_{\omega},$
\item guided eigenfunctions:  $\exists \omega, \omega\not=\emptyset$, such that $\inf\limits_{u\in\mathfrak{F}_G}  R_{\omega}(u) = 0.$
\end{itemize}
The paper deals with a spectral characterization of theses two sets among others things. The layered structure of the operator permits a representation of its spectrum as a subset of points  indexed by $(k,l)\in \mathbb{N}\times\mathbb{N}.$ This allows a geometric description of the guided and non-guided eigenfunction categories. Section 4.1 of the previous version was not correct, now it is corrected, many proofs are simplified and a new general result is added.
\end{abstract}
%\numberwithin{equation}{section}
%\numberwithin{montheo}{section}
%\numberwithin{moncor}{section}
%\numberwithin{monlem}{section}
%\numberwithin{mapropo}{section}
%\numberwithin{maremarque}{section}
%\numberwithin{madef}{section}
\section{Introduction : objectifs et le mod\`{e}le principal}
\label{section-introduction}
Pour un op\'erateur elliptique autoadjoint agissant dans un ouvert $\Omega$ born\'e ou non, on sait que les variations des coefficients de la partie principale, ainsi que certaines conditions au bord, peuvent cr\'eer un ph\'enom\`{e}ne de concentration de l'\'energie. La litt\'erature \'etant abondante, nous ne donnerons que quelques exemples, 
\begin{enumerate}
\item en 1930, Epstein \cite{EP:1} a mis en \'evidence, en milieu non born\'e, des ondes guid\'ees acoustiques pour une famille de vitesses, chacune \'etant une fonction analytique ne d\'ependant que de la coordonn\'ee verticale, ce  qui a donn\'e naissance \`{a} de nombreux travaux. Les applications vont de l'acoustique aux fibres optiques\footnote{Les fibres optiques sont de tr\`{e}s bonnes illustrations. La fibre monophase correspond au mod\`{e}le $N=1$ trait\'e ici : deux vitesses avec un c\oe ur d'indice plus grand (donc avec la vitesse la plus petite). Elle est un bon exemple de la concentration de l'\'energie dans le c\oe ur. Cette concentration est d'autant plus grande que le rayon transversal du c\oe ur (l'analogue du $h_{0}$ de ce travail) est petit. 
%La difficult\'e est de bien placer la source face au c\oe ur, ce qui am\`{e}ne souvent \`{a} lui pr\'ef\'erer la fibre multiphas\'ee (trait\'ee plus tard) qui a d'autres ennuis car il faut espacer les pulses dans le temps.
}, (cf. \cite{Ped-Wh:1} ainsi que \cite{Wil:1} et sa bibliographie).
\item le syst\`{e}me de l'\'elasticit\'e lin\'eaire dans le demi-espace $\Omega=\mathbb{R}^{n}\times(0,+\infty)$ avec la condition de surface libre  
cr\'ee un couplage au bord qui donne naissance \`{a} l'onde de Rayleigh. Elle est  particuli\`{e}rement destructrice lors d'un tremblement de terre (\cite{Schu:1}, \cite{DerGui:1}). Les noms des physiciens Lamb, Love et Stoneley ont aussi \'et\'e donn\'es \`{a} des ph\'enom\`{e}nes voisins (\cite{Cristofol:1}).
\end{enumerate}
Dans ces deux items, le milieu \'etait stratifi\'e et les fonctions propres ne l'\'etaient qu'en un sens g\'en\'eralis\'e car elles n'appartenaient pas au domaine de l'op\'erateur mais des sous-familles guid\'ees \'etaient quand m\^{e}me distingu\'ees. Maintenant notre ouvert $\Omega:=(0,L)\times(O,H)$ est born\'e et notre op\'erateur elliptique $A$ est $-\nabla\cdot(c\nabla)$ ou $-c\Delta$ avec la condition de Dirichlet au bord. Le coefficient de diffusion $c$ \'etant une fonction ne d\'ependant que la seconde coordonn\'ee $x_{2},$ le milieu est stratifi\'e et l'approche par s\'eparation des variables s'impose. Par suite, les valeurs propres sont naturellement index\'ees par deux indices, \`{a} savoir $(\lambda_{k,\ell}), k,\ell\geq 1.$ On construit ainsi une base orthonorm\'ee $\mathcal{B}$ de fonctions propres $(v_{k,\ell})_{k\geq 1,\ell\geq 1}$ associ\'ees aux valeurs propres $\lambda_{k,\ell}.$
 % qui v\'erifient donc $-\nabla\cdot(c\nabla v_{k,\ell})= \lambda_{k,\ell}v_{k,\ell}$. 
  Elles sont de la forme $v_{k,\ell}(x_{1},x_{2})= a_{k,\ell}\sin(\frac{k\pi}{L}x_{1})u_{k,\ell}(x_{2}) $ o\`{u} $u_{k,\ell}(x_{2})$ satisfait
\begin{equation}
\label{equation-modelegeneral:1}
(cu')' + (\lambda_{k,\ell}-c\frac{k^{2}\pi^{2}}{L^{2}})u = 0, u(0) = u(H) = 0.
\end{equation}
On trouvera deux exemples explicites en \eqref{equation-conc-conf-fonctionpropreguidee-1saut:1}  et \eqref{suite(stratification3valeurs)equation:6}. Le r\'esultat qui suit, la d\'emonstration \'etant dans l'annexe \ref{annexe-theoreme-conf-conc-nvelleappr:1}, peut \^{e}tre consid\'er\'e comme un point de d\'epart de notre r\'eflexion car il montre comment certaines fonctions propres peuvent se concentrer dans des r\'egions particuli\`{e}res de $\Omega.$
%donne d\'ej\`{a} une indication sur la localisation de la masse de certaines fonctions propres.
%nous modifions le cadre en supposant :
%%%%%%%%%%%%%%%%%%%%%%%%%%%%%%%%%%%%%%%%%%%%%%%%%%%%%%%%%%%
%%%%%%%%%%%%%%%%%%%%%%%%%%%%%%%%%%%%%%%%%%%%%%%%%%%%%%%%%%%
 \begin{montheo}
 \label{theoreme-conf-conc-nvelleappr:1}
 Soit $\Omega:=(0,L)\times(0,H)$ et l'op\'erateur autoadjoint positif $A = -c\Delta,$ op\'erant dans $\mathcal{H}:=L^{2}(\Omega, c^{-1}{\rm d}x)$ et de domaine $ D(A):= H_{0}^{1}(\Omega)\cap H^{2}(\Omega).$ On suppose le coefficient de diffusion $c\in L^{\infty}$ non constant, positif, born\'e inf\'erieurement mais constant pour $0<h_{0}< x_{2}<H$ :
 \begin{equation*}
\label{hypothese-conf-conc-nvelleappr:1}
\begin{array}{c}
c(x) =\left\lbrace 
\begin{array}{l}
 \gamma_{0}(x_{2}), 0<x_{2}<h_{0} \\
 c_{1}, h_{0}<x_{2}<H,
 \end{array}
 \right.\\
0<\underline{c}:= {\rm ess}\inf\, c< c_{1}.
 \end{array}
 \end{equation*}
  Soit $\omega:=(0,L)\times (a,b)$ o\`{u} l'intervalle  $(a,b)$ satisfait $h_{0}\leq a<b\leq H,$ et $\varepsilon$ assez petit. Il existe une constante $K_{\varepsilon, c}$ telle que, pour chaque valeur propre $\lambda_{k,\ell}$ verifiant $\underline{c}\frac{k^{2}\pi^{2}}{L^{2}}<\lambda_{k,\ell}<(c_{1}-\varepsilon)\frac{k^{2}\pi^{2}}{L^{2}},$ on ait
 \begin{equation*}
 \label{equation-conf-conc-nvelleappr:1}
 \frac{\Vert v_{k,\ell}\Vert^{2}_{L^{2}(\omega)}}{\Vert v_{k,\ell}\Vert^{2}_{L^{2}(\Omega)}}\leq  K_{\varepsilon,c}e^{-2\sqrt{\frac{k^{2}\pi^{2}}{L^{2}}-\frac{\lambda_{k,\ell}}{c_{1}}}(a-h_{0})}.
 \end{equation*}
% o\`{u} $\xi'_{1}:= \sqrt{\frac{k^{2}\pi^{2}}{L^{2}}-\frac{\lambda_{k,\ell}}{c_{1}}}.$
 \end{montheo}
 Notre principal mod\`{e}le sera moins g\'en\'eral mais, en contrepartie, nous obtiendrons des r\'esultats plus pr\'ecis. Nous consid\'erons donc
\vskip.3cm
\noindent
\centerline{\hskip-1cm\qquad\begin{tabular}[h]{|  c  |}
\hline
{\bf Hypoth\`{e}se g\'en\'erale (H0)}\\
$\left\lbrace
\begin{array}{l} 
\bullet\; \mbox{ Deux ouverts : } \omega\subset\Omega:=(0,L)\times(0,H), \omega:=\omega_{1}\times(a,b), 
%0\leq l_{1}<l_{2}\leq L,
0\leq a<b\leq H,\\
\bullet\; \mbox{ la fonction scalaire $ c,$ constante par morceau, ne d\'epend que de la seconde coordonn\'ee $x_{2},$} 
%c:x_{=(x_{1},x_{2})\to c(x) 
%\mbox{ et est }
\\
%ou est de classe } C^{1},\\
\bullet\; \mbox{ la fonction } c \mbox{ est monotone croissante avec }
%c(x) = c(x_{2}), 
0<c_{\min}\leq c\leq c_{\max}<+\infty,\\
\bullet\; \left\lbrace\begin{array}{l}
A=-c\Delta \mbox{ de domaine } D(A):= H_{0}^{1}(\Omega)\cap H^{2}(\Omega), \mbox{ op\'erant dans } \mathcal{H}:=L^{2}(\Omega, c^{-1}{\rm d}x)\\
\mbox{ ou }\\
A = -\nabla\cdot(c\nabla) \mbox{ de domaine } D(A):= \{u\in H_{0}^{1}(\Omega) ; Au\in L^{2}(\Omega)\}, \mbox{ op\'erant dans } L^{2}(\Omega,{\rm d}x).
\end{array}\right.
\end{array}
\right.
$\\
\hline
\end{tabular}}
\vskip.3cm
Nous compl\'etons l'hypoth\`{e}se {\bf (H0)} par la suivante\\
% afin de distinguer nos deux mod\`{e}les de base :\\
\centerline{\hskip-.7cm\qquad\begin{tabular}[h]{|  c  |}
\hline
{\bf Hypoth\`{e}se $\mathbf{(H1)}$ : Mod\`{e}le \`{a} N sauts }\\
$\bullet$\; l'ouvert $\Omega:= (0,L)\times(0,H)\subset\mathbb{R}^{2}$ est partag\'e en $N+1$ parties, $\Omega_{0}:=(0,L)\times(0,h_{0}),$\\ 
$\Omega_{i}:=(0,L)\times(h_{i-1},h_{i}), i = 1,\ldots,N,$ o\`{u}  $0=h_{-1}<h_{0}<\ldots< h_{i}<h_{i+1}<\ldots<h_{N-1}<h_{N}=H,$\\ qui sont s\'epar\'ees par $N$ interfaces  horizontales $S_{i}:=(0,L)\times\{h_{i}\}, i=0,\ldots,N-1;$\\
$\bullet$\; une fonction scalaire $c$ prenant $N+1$ valeurs, \`{a} savoir $c(x) =\left\lbrace \begin{array}{l}
c_{0} \mbox{ si } 0<x_{2}<h_{0},\\
c_{1}  \mbox{ si } h_{0}<x_{2}<h_{1},\\
\vdots\\
c_{N}  \mbox{ si } h_{N-1}<x_{2}<H,
\end{array}\right.$\\
avec $c_{0}<c_{1}<\ldots <c_{N-1}<c_{N}.$\\
\hline
\end{tabular}
}
%\vskip.3cm
%\centerline{\hskip-.7cm\qquad\begin{tabular}[h]{|  c  |}
%\hline
%{\bf Mod\`{e}le $C^{1}$ : Hypoth\`{e}se $\mathbf{(H2)}$}\\
%La fonction scalaire $c$ est strictement croissante et appartient \`{a} $C^{1}(\lbrack 0,H\rbrack).$\\
%\hline
%\end{tabular}}
\vskip.3cm
Pour l'op\'erateur elliptique $A$ autoadjoint  et positif, nous sommes int\'eress\'es par le rapport 
\begin{equation}
\label{notation-conc-conf-generalites:1}
R_{\omega}(v):= \frac{\int_{\omega} \vert v(x)\vert^{2}{\rm d}x}{\int_{\Omega} \vert v(x)\vert^{2}{\rm d}x},
\end{equation}
qui nous permettra de distinguer deux sous-familles de fonctions propres : les fonctions propres guid\'ees et les autres. L'id\'ee de s'int\'eresser au rapport \eqref{notation-conc-conf-generalites:1} n'est pas nouvelle.
Pour le Laplacien sur des vari\'et\'es compactes, analytiques r\'eelles avec une m\'etrique analytique r\'eelle, l'in\'egalit\'e
\begin{equation}
\label{Laurent-Leautaud:1}
C e^{-c\lambda_{j}^{\frac{1}{2}}}\Vert \varphi_{j}\Vert^{2}_{L^{2}(\Omega)}\leq \Vert \varphi_{j}\Vert^{2}_{L^{2}(\omega)}
\end{equation}
est connue, cf. \cite{DoFe:1}. Elle est devenue un cas particulier de l'in\'egalit\'e de Lebeau-Robbiano, cf. \cite{JL:1}, \cite{LR:1}, pour des op\'erateurs \`{a} coefficients variables r\'eguliers dans un ouvert born\'e. Dans un r\'ecent travail de Camille Laurent et Matthieu L\'eautaud, cf.\cite{LaLe:1}, les auteurs g\'en\'eralisent \eqref{Laurent-Leautaud:1} \`{a} des op\'erateurs hypoelliptiques lorsque $\Omega$ est une vari\'et\'e compacte sans bord, disons un tore, toujours sous la condition d'analyticit\'e.
 Nos encadrements sont plus pr\'ecis sur certains sous-espaces vectoriels de $L^{2}(\Omega)$ et ils montrent que l'exposant $\lambda_{j}^{\frac{1}{2}}$ de \eqref{Laurent-Leautaud:1} ne peut \^{e}tre am\'elior\'e (i.e. diminu\'e) que pour certains.\\
 %%%%%%%%%%%%%%%%%%%%%%%%%%%%%%%%%%%%%%
% \begin{figure}[h]
%\centerline{\includegraphics[scale=.5]{valeurspropres1saut-2.pdf}\includegraphics[scale=.5]{zonesconiquesspectrales2bis-1saut-2bis.pdf}}
%%\label{fig:figure1}
%\caption{\label{fig:figure1}
%{\bf 2 mod\`{e}les \`{a} un saut : r\'epartition des valeurs propres.}
%Dans le cas particulier de gauche, 
%$ \lambda_{3,3} = \lambda_{4,2}$
% et ces valeurs propres sont de part et d'autre de la ligne rouge.}
%\end{figure}
 \begin{figure}[h]
\centerline{\includegraphics[scale=.5]{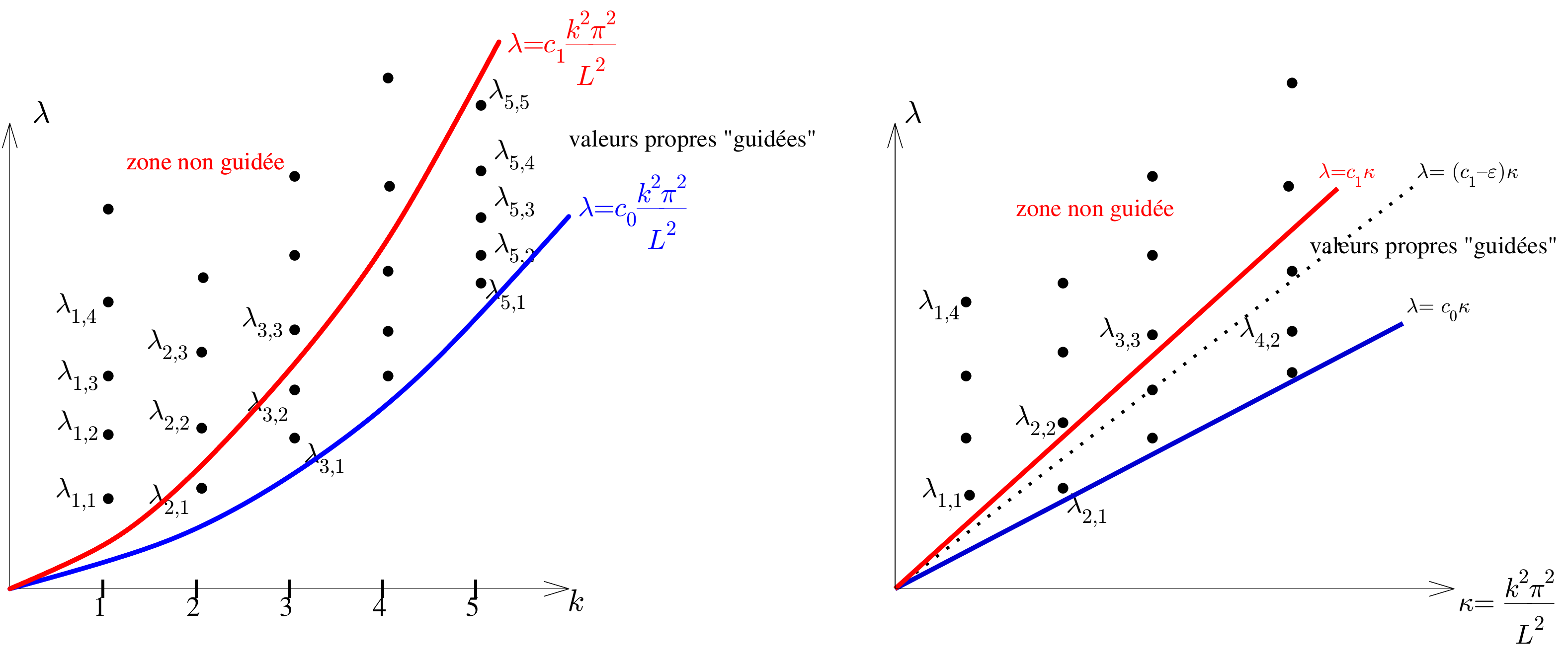}}
%\label{fig:figure1}
 \label{fig:figure1}
\caption{{\bf un mod\`{e}le \`{a} un saut : r\'epartition des valeurs propres.}
Dans ce cas particulier, 
$ \lambda_{3,3} = \lambda_{4,2}$
 et ces valeurs propres sont de part et d'autre de la ligne rouge.}
\end{figure}
 %%%%%%%%%%%%%%%%%%%%%%%%%%%%%%%%%%%%%%
 Prenons le mod\`{e}le \`{a} un saut, i.e. $N=1, 0<c_{0}<c_{1}.$ Pour chaque $k,$ il n'y a qu'un nombre fini de valeurs propres entre $ c_{0}\frac{k^{2}\pi^{2}}{L^{2}}$ et $ c_{1} \frac{k^{2}\pi^{2}}{L^{2}}$, les autres \'etant sup\'erieures \`{a} $\lambda = c_{1}\frac{k^{2}\pi^{2}}{L^{2}}.$ Dans notre partition, les fonctions dites guid\'ees, i.e. appartenant \`{a} $\mathfrak{F}_{G}$, sont associ\'ees aux valeurs propres comprises entre $ c_{0}\frac{k^{2}\pi^{2}}{L^{2}}$ et $ c_{1} \frac{k^{2}\pi^{2}}{L^{2}}$. La partie gauche de la Figure \ref{fig:figure1} les visualise dans cette zone en forme de corne parabolique.
% les valeurs propres associ\'ees \`{a} la famille des fonctions propres guid\'ees : chaque point $(k,\lambda_{k,\ell})$ de la zone comprise entre les deux paraboles $\lambda = c_{0}\frac{k^{2}\pi^{2}}{L^{2}}$ et $\lambda = c_{1} \frac{k^{2}\pi^{2}}{L^{2}}$ correspond \`{a} une fonction propre guid\'ee.
  La concentration en norme $L^{2}$ des fonctions propres associ\'ees aux grandes valeurs propres de cette corne a lieu du c\^{o}t\'e $\Omega_{0}$ de l'interface (th\'eor\`{e}me \ref{theoreme-conc-conf-fonctionguidee-1saut:1bis}, Cas 2) et cette concentration est d'autant plus forte que la valeur propre est grande.
  Noter que ces paraboles d\'eterminent deux zones qui deviennent coniques par un changement de variables (partie droite de la Figure \ref{fig:figure1}). Par exemple $(k,\lambda)\to (\kappa,\sqrt{\lambda})=(k\pi,\sqrt{\lambda})$ permet d'\'etablir une correspondance avec les notions microlocales habituelles en travaillant dans $\Omega\times((\mathbb{N}\frac{\pi}{L})\times\mathbb{R})$\footnote{L'application $(k,\lambda)\to (\frac{k^{2}\pi^{2}}{L^{2}},\lambda)$ convient aussi (cf. la partie droite de la Figure \ref{fig:figure1}) mais ne donne pas une correspondance imm\'ediate avec les coordonn\'ees microlocales.}, d'o\`{u} un autre point de vue pour mesurer la localisation si on utilise la
  \begin{madef}
  \label{definition-distance} On pose  $\tilde{\rho}_{k,\ell}:= (\frac{\sqrt{1+c_{1}^{2}}}{c_{1}}\rho_{k,\ell})^{\frac{1}{2}}$ o\`{u} $\rho_{k,\ell}$ est la distance du point $(\kappa=\frac{k^{2}\pi^{2}}{L^{2}},\lambda_{k,\ell})$ \`{a} la droite d'\'equation $\lambda=c_{1}\kappa$.
\end{madef}

%%%%%%%%%%%%%%%%%%%%%%%%%%%%%%%%%%%%%%%%%%%%%%%%%%%%%%
\begin{montheo}
\label{theoreme-preliminaire:1}
On suppose {\bf (H0)} et {\bf (H1)} avec $N=1, \omega\subset \Omega, \omega:=\omega_{1}\times  (a,b)$ o\`{u} $0\leq a<b\leq H.$\\ 
%\omega = (l_{1},l_{2})\times (a,b), h_{0}\leq a<b\leq H$ et une suite de valeurs propres $(\lambda_{k_{n},\ell_{n}})_{n}, k_{n}\to\infty,$ contenue dans un c\^{o}ne $c_{0}\kappa< \lambda<(c_{1}-\varepsilon)\kappa,$ alors
%\begin{enumerate}
{\bf 1. Cas non-guid\'e.} Il existe une constante $C_{\omega}>0$ telle que
\begin{equation}
\label{equation-preliminaire-1saut:1}
C_{\omega}\leq \frac{\Vert v_{k,\ell}\Vert^{2}_{L^{2}(\omega)}}{\Vert v_{k,\ell}\Vert^{2}_{L^{2}(\Omega)}}, \quad\forall \lambda_{k,\ell}>c_{1}\frac{k^{2}\pi^{2}}{L^{2}},
\end{equation}
sans restriction sur la localisation de l'ouvert $\omega.$\\
{\bf 1. Cas guid\'e.} Soit $0<\varepsilon <c_{1}-c_{0}$ et une suite de valeurs propres $(\lambda_{k_{n},\ell_{n}})_{n}, k_{n}\to\infty,$ contenue dans la corne parabolique $c_{0}\frac{k^{2}\pi^{2}}{L^{2}}< \lambda<(c_{1}-\varepsilon)\frac{k^{2}\pi^{2}}{L^{2}}.$ Alors, si $\omega$ est contenu dans $(0,L)\times(h_{0},H)$ (partie o\`{u} le coefficient de diffusion est le plus grand), on a
\begin{equation}
\label{equation-preliminaire-1saut:2}
\frac{\Vert v_{k_{n},\ell_{n}}\Vert^{2}_{L^{2}(\omega)}}{\Vert v_{k_{n},\ell_{n}}\Vert^{2}_{L^{2}(\Omega)}}= O\left(\frac{e^{-2(a-h_{0})\tilde{\rho}_{k_{n},\ell_{n}}}}{\tilde{\rho}_{k_{n},\ell_{n}}}\right)
\end{equation}
 Ainsi la masse de ces fonctions propres se concentre de plus en plus dans la \og vall\'ee\fg{} qui correspond, ici, \`{a} la partie inf\'erieure de $\Omega.$
%\begin{equation}
%\label{equation-preliminaire:1}
%R_{\omega}(v_{k_{n},\ell_{n}})= O\left(\frac{e^{-2(a-h_{0})\tilde{\rho}_{k,\ell}}}{\tilde{\rho}_{k,\ell}}\right)
%\end{equation}
%\end{enumerate}
\end{montheo}
\begin{maremarque}
\label{remarque-conc-conf-introduction:1}
Trois choses sont importantes \`{a} retenir :
\begin{enumerate}
\item  Si $\lambda$ est une valeur propre multiple (Figure \ref{fig:figure1}, \`{a} gauche), il y a plusieurs possibilit\'es : les fonctions propres associ\'ees peuvent n'\^{e}tre que des \'el\'ements de $\mathfrak{F}_{G}$ (fonctions propres guid\'ees) ou bien de $\mathfrak{F}_{NG}$ (fonctions propres non guid\'ees) ou certaines de $\mathfrak{F}_{G}$ et les autres de $\mathfrak{F}_{NG}$ (par exemple $\lambda_{3,3}$ et $\lambda_{4,2}$ de la figure \ref{fig:figure1}).
\item Pour mesurer le taux de d\'ecroissance, nous avons d'abord utilis\'e $\sqrt{\frac{k^{2}\pi^{2}}{L^{2}}-\frac{\lambda_{k,\ell}}{c_{1}}}$ (th\'eor\`{e}me \ref{theoreme-conf-conc-nvelleappr:1}), puis la distance $\tilde{\rho}_{k_{n},\ell_{n}}$ dans \eqref{equation-preliminaire-1saut:2} mais nous disposons aussi de la valeur propre $\lambda_{k_{n},\ell_{n}}$ comme il est fait en \eqref{Laurent-Leautaud:1} et, plus loin, dans la proposition \ref{proposition-conc-conf-fonctionguidee-1saut:1}.
\item On verra plus loin que selon la suite de fonctions propres guid\'ees et la localisation \og microlocale\fg{} de celle-ci, le taux de convergence de $R_{\omega}(v_{n})$ vers 0 varie de $e^{-b_{1}\lambda_{n}^{\frac{1}{2}}}$ \`{a} $c_{\omega}\lambda_{n}^{-3/2}$ en passant par $e^{-b_{2}\lambda_{n}^{\frac{1}{4}}}, \lambda_{n}$ \'etant la valeur propre associ\'ee \`{a} la fonction propre guid\'ee $v_{n}$.
\end{enumerate}
\end{maremarque}
%%%%%%%%%%%%%%%%%%%%%%%%%%%%%%%%%%%%%%%%%%%%%%%%%%%%%%%%%
 D\'efinir la partition $\mathfrak{F}_{NG}, \mathfrak{F}_{G}$ est une question d\'elicate. Pour le mod\`{e}le \`{a} 1 saut ($N=1, c_{0}< c_{1})$ nous l'avons fait de fa\c{c}on pr\'ecise. Une autre approche est la suivante :\\ 
$\bullet$ 
%{\textbf{\textit{
Pour la famille $\mathfrak{F}_{NG}$ (fonctions propres non guid\'ees) le rapport $R_{\omega}(v)$ est uniform\'ement minor\'e par une constante positive. Autrement dit, on se trouve dans le cas d'un op\'erateur elliptique du second ordre \`{a} coefficients constants (se reporter au th\'eor\`{e}me \ref{theoreme-conc-conf-fonctionnonguidee:1})
%}}}
\\
$\bullet$ 
%{\textbf{\textit{
Si $(v_{n})$ est une suite de fonctions propres guid\'ees, i.e. contenue dans $\mathfrak{F}_{G},$ alors le rapport $R_{\omega}(v_{n})\to 0$ si $\bar{\omega}\subset \Omega_{1},$ i.e. est contenu dans la zone o\`{u} le coefficient de diffusion a sa plus grande valeur.
% }}
\\
Avec cette d\'efinition, la partition $\mathfrak{F}_{G}, \mathfrak{F}_{NG}$ n'est pas unique : le transfert d'un nombre fini d'\'el\'ements de l'une des familles  \`{a} l'autre ne changera pas les propri\'et\'es ci-dessus. Ce n'est pas vraiment g\^{e}nant puisque ce sont les propri\'et\'es asymptotiques des \'el\'ements de la partition qui nous int\'eressent.\\
\\
%En outre, les r\'esultats obtenus pour $v_{n}$ sont beaucoup plus pr\'ecis si, au lieu d'utiliser les valeurs propres $\lambda_{n},$ on utilise la distance $\rho_{k,\ell}$ introduite dans la D\'efinition \ref{definition-distance}.
% \ref{remarque-distancebarriere:1}
%%%%%%%%%%%%%%%%%%%%%%%%%%%%%%%%%%%%%%%%%
Dans le mod\`{e}le \`{a} deux sauts, i.e. $N=2,$ nous conjecturons que les fonctions dites guid\'ees correspondent aux valeurs propres situ\'ees entre les deux paraboles $\lambda = c_{0}\frac{k^{2}\pi^{2}}{L^{2}}$ et $\lambda = c_{2} \frac{k^{2}\pi^{2}}{L^{2}}$ que nous divisons en la zone (0)$= \{c_{0}\frac{k^{2}\pi^{2}}{L^{2}}<\lambda_{k,\ell}<c_{1}\frac{k^{2}\pi^{2}}{L^{2}} \}$ et la zone (I)$=\{c_{1}\frac{k^{2}\pi^{2}}{L^{2}}<\lambda_{k,\ell}<c_{2}\frac{k^{2}\pi^{2}}{L^{2}}\} $.  On distinguera donc trois zones spectrales dont deux correspondant \`{a} une famille de fonctions propres guid\'ees (voir la Figure 2).  Chaque fonction propre associ\'ee \`{a} la zone (0) est une onde presque rasante en $S_{0}$ car elle arrive sur cette interface avec un angle sup\'erieur \`{a} l'angle limite de la r\'eflexion totale\footnote{Pour le voir il suffit d'\'ecrire \eqref{suite(stratification3valeurs)equation:5}-\eqref{suite(stratification3valeurs)equation:6} avec des exponentielles.} ce qui explique que son
%et, malgr\'e cela, ont une 
\'energie soit essentiellement localis\'ee dans $\Omega_{0}$ tandis que les fonctions propres associ\'ees \`{a} la zone (I) peuvent traverser l'interface $S_{0}$ sans difficult\'e mais subissent une r\'eflexion totale sur l'interface $S_{1}$ et sont donc essentiellement localis\'ees dans $\Omega_{0}\cup\Omega_{1}.$
%%%%%%%%%%%%%%%%%%%%%%%%%%%%%%%%%%%%%%%%%%%%%%%%%%%%
\begin{figure}[h]
\hspace{.8cm}\includegraphics[scale=.5]{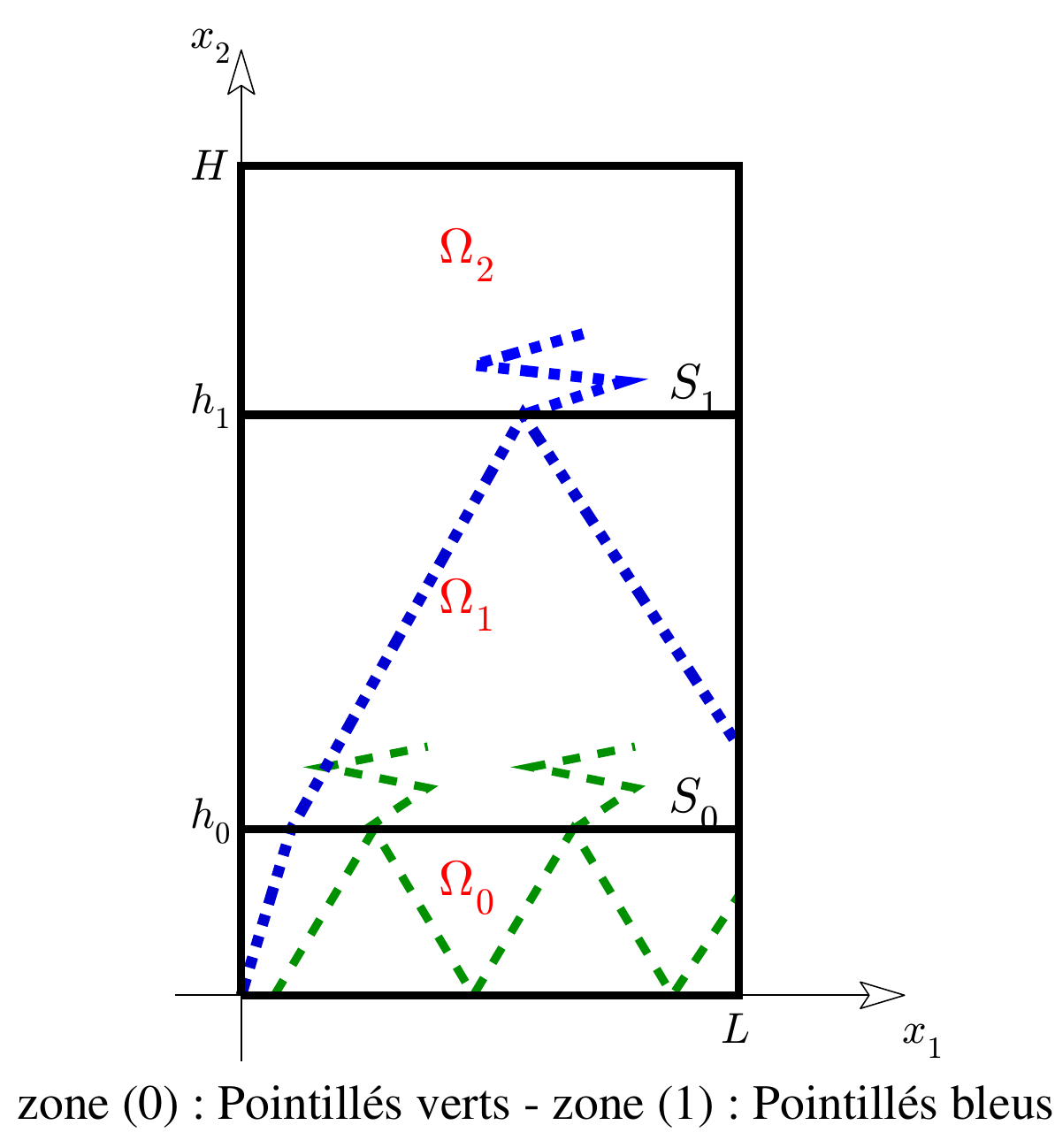}\hspace{.5cm}\includegraphics[scale=.45]{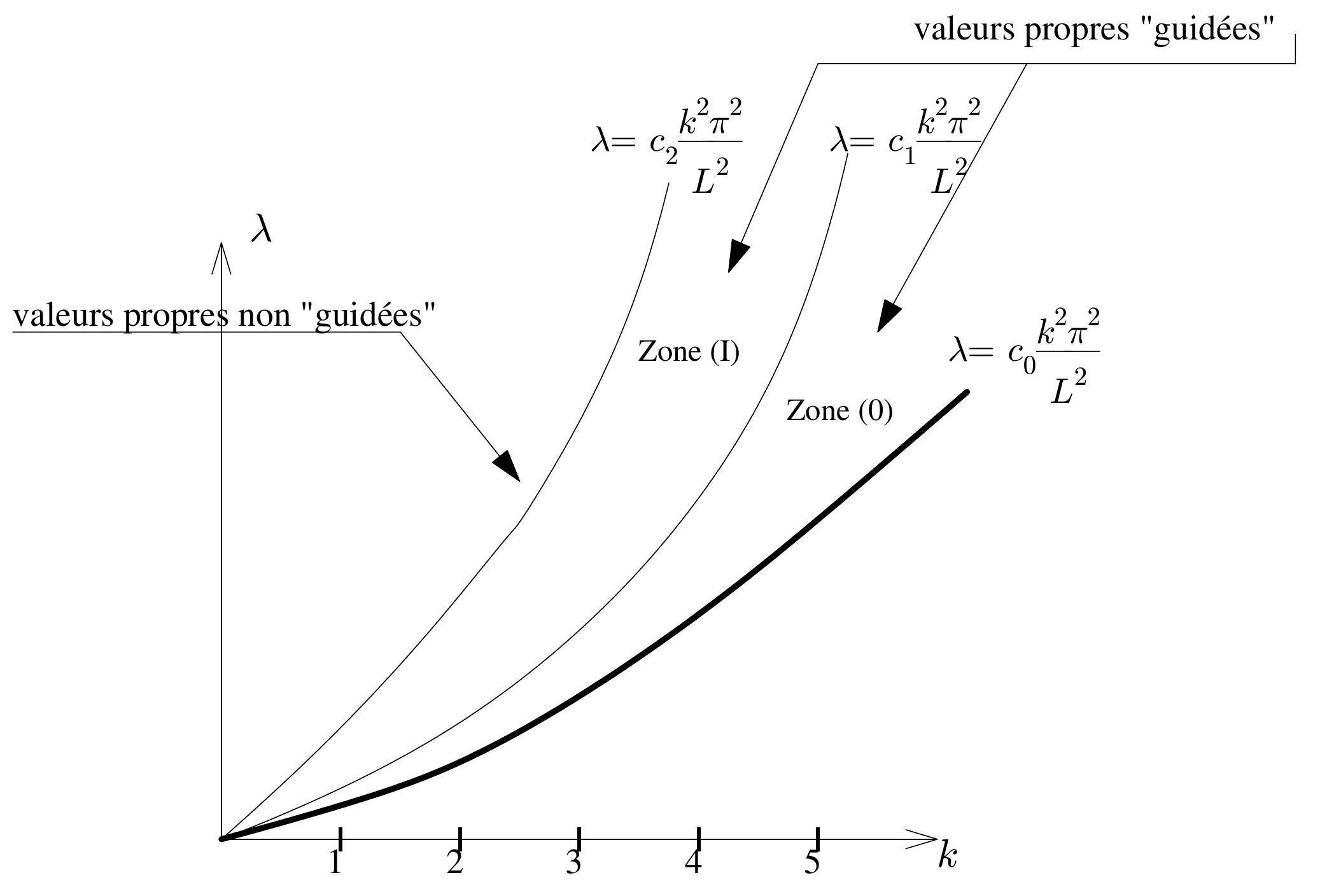}\hfill
\caption{\bf Mod\`{e}le \`{a} deux sauts.}
\label{fig:modele2sauts}
\end{figure}
Pour l'aspect non-guid\'e, on se reportera \`{a} l'annexe \ref{annexe-theoreme-conf-conc-fonctionnonguidee:1} o\`{u} se trouve la d\'emonstration
du
\begin{montheo}[cas non guid\'e]
\label{theoreme-conc-conf-fonctionnonguidee:1}
Soient $\mathbf{(H0)}, \mathbf{(H1)}, \varepsilon>0$ et un ouvert $\omega:=\omega_{1}\times(a,b)$ avec $0\leq a<b \leq H$ qui sont fix\'es. Si nous posons $c_{H}:= \sup_{0<x_{2}<H}c(x_{2})$, alors il existe une constante $C_{\omega}$ strictement positive, telle que l'on ait pour toute fonction propre $v_{k,\ell}$ associ\'ee \`{a} la valeur propre $\lambda_{k,\ell}$
\begin{equation}
\label{equation-conc-conf-fonctionnonguidee:1}
\lambda_{k,\ell}>(c_{H}+\varepsilon)\frac{k^{2}\pi^{2}}{L^{2}}\Longrightarrow 0<C_{\omega}\leq \frac{\Vert v_{k,\ell}\Vert_{L^{2}(\omega)}}{\Vert v_{k,\ell}\Vert_{L^{2}(\Omega)}}\leq 1.
\end{equation}
 La constante $C_{\omega}$ est invariante par toute translation de $\omega.$ Le r\'esultat est aussi valable si  on remplace les hypoth\`{e}ses $\mathbf{(H0)}- \mathbf{(H1)}$ par $\mathbf{(H0)}- \mathbf{(H2)}.$
\end{montheo}
\vskip-.2cm
\centerline{\hskip-.7cm\qquad\begin{tabular}[h]{|  c  |}
\hline
{\bf  Hypoth\`{e}se $\mathbf{(H2)}$ : Mod\`{e}le $C^{1}$}\\
La fonction scalaire $c$ est strictement croissante et appartient \`{a} $C^{1}(\lbrack 0,H\rbrack).$\\
\hline
\end{tabular}}
\vskip.1cm
Les r\'esultats des sections 3 et 4 sur la d\'ecroissance exponentielle des fonctions propres feront penser \`{a} ceux de S. Agmon (cf. \cite{AG:1}) et d'autres travaux o\`{u} les auteurs agissent principalement dans des ouverts non born\'es. L'analogie s'arr\^{e}te l\`{a} pour plusieurs raisons :
\begin{itemize}
\item[$\bullet$] C'est l'influence d'un potentiel $q$ dans l'op\'erateur $-\nabla\cdot(A\nabla) + q$ qui les motive, la distance d'Agmon permettant d'estimer la concentration de certaines fonctions propres en des zones d\'ependant de $q.$ Nous n'avons pas de potentiel car nous consid\'erons des op\'erateurs de la forme $-\nabla\cdot(c\nabla)$\footnote{Si $c\in C^{1}$, un changement de fonction inconnue  permettrait d'introduire un potentiel mais le probl\`{e}me changerait.} ou $-c\Delta$ et ce sont les variations de $c$ qui nous importent.
\item[$\bullet$] Pour les valeurs propres $\lambda$ inf\'erieures \`{a} $\Sigma$, l'infimum  du spectre essentiel, les fonctions propres associ\'ees d\'ecroissent exponentiellement \`{a} l'infini avec un taux li\'e \`{a} la diff\'erence $\sqrt{\Sigma-\lambda}$ (Theorem 4.1 de \cite{AG:1}). Dans notre travail, bien que $\Sigma = +\infty$\footnote{Cas que \cite{AG:1} n'exclue pas.}, on ne voit que pour certaines valeurs propres une d\'ecroissance d'allure exponentielle dans la partie sup\'erieure de $\Omega.$
\end{itemize}
Nous parlons donc, comme les physiciens, de solutions guid\'ees et le substitut \`{a} $\Sigma$ sera la barri\`{e}re spectrale mentionn\'ee dans la Remarque \ref{remarque-distancebarriere:1}. Les auteurs de \cite{AFM:1} \'elargissent dans un r\'ecent travail des id\'ees de S. Agmon mais l'exigence d'un potentiel $q\geq 0, q\not=0,$ ne permet qu'une comparaison partielle \`{a} ce stade avec nos r\'esultats : elle peut cependant \^{e}tre faite avec notre op\'erateur r\'eduit $A_{k}u:= -(cu')' + c\frac{k^{2}\pi^{2}}{L^{2}}u$ puisque le potentiel $c\frac{k^{2}\pi^{2}}{L^{2}}$ satisfait leur condition. 
\vskip.5cm
%%%%%%%%%%%%%%%%%%%%%%%%%%%%%%%%%%%%%%%%%%%%%%%%%%%%
Le plan de ce papier est le suivant.
%La section \ref{section-fonctionnonguidee:1} donne une condition suffisante pour qu'une fonction propre se comporte comme si l'op\'erateur $A$ avait un coefficient de diffusion constant. Cette condition est parfois n\'ecessaire pour le mod\`{e}le \`{a} un saut, il reste \`{a} prouver cette n\'ecessit\'e dans les autres cas. 
%Dans les sections suivantes on s'int\'eresse aux fonctions propres guid\'ees. 
La section \ref{section-modeleunsaut:1} rassemble les propri\'et\'es des suites de fonctions propres guid\'ees du mod\`{e}le \`{a} un saut. C'est le mod\`{e}le le plus simple \`{a} comprendre. La section \ref{section-modele2sauts:1} consid\`{e}re aussi les fonctions propres guid\'ees, cette fois-ci pour le mod\`{e}le \`{a} deux sauts, mais elle est incompl\`{e}te pour le moment. 
Des extensions sont en devenir : un disque ouvert $\Omega$ et un disque int\'erieur $\Omega_{0}$ de m\^{e}me centre, $\Omega_{1}$ \'etant la couronne $\Omega\setminus\Omega_{0},$ ou un tore plat avec un coefficient de diffusion discontinu. Nous pensons que les r\'esultats seront analogues. Il faudrait aussi envisager la g\'en\'eralisation \`{a} des situations moins simples. La plupart des calculs sont renvoy\'es aux annexes.\\
\noindent
{\bf Notations : } Si $(f_{n})$ et $(g_{n})$ sont deux suites de fonctions, la notation $f_{n}\approxeq g_{n}$ signifiera que ces quantit\'es sont comparables, i.e. qu'il existe des nombres r\'eels positifs $M_{1},M_{2}>0$ tels que $M_{1}g_{n}\leq f_{n}\leq M_{2}g_{n}$ d\`{e}s que $n$ est assez grand. Quand il n'y a aucune possibilit\'e d'ambigu\"{\i}t\'e, la notation $C$ d\'esigne une constante strictement positive qui peut prendre diff\'erentes valeurs. 
\section{Concentration des ondes guid\'ees du mod\`{e}le \`{a} un saut}
\label{section-modeleunsaut:1}
Maintenant nous consid\'erons les ondes (ou, suivant le langage adopt\'e, fonctions propres) guid\'ees pour un milieu avec un coefficient de diffusion ayant deux valeurs. De mani\`{e}re intuitive, ces fonctions sont celles dont l'essentiel de la masse est concentr\'ee dans une partie de l'ouvert $\Omega.$ Ici, elles correspondent exactement aux ondes ayant  une r\'eflexion dite totale sur l'interface.\\
On pose par commodit\'e, dans cette section,  $c_{1}>c_{0}=1, L=1$ et parfois $h_{0}= \frac{1}{2}$. Ces restrictions n'ont aucune cons\'equence sur la g\'en\'eralit\'e des r\'esultats car des transformations unitaires permettent de se ramener au cas g\'en\'eral (cf. la remarque \ref{remarque:4}). Avant de d\'etailler la formulation des $v_{k,\ell}$, pr\'ecisons quelques notations dans lesquelles les $\lambda_{k,\ell}$ sont les valeurs propres de l'op\'erateur
%$A$ que nous choisirons de la forme
 $A = -c\Delta$ (la forme divergentielle $-\nabla\cdot c\nabla$ donne des r\'esultats g\'en\'eraux analogues)
\begin{equation}
\label{notation:1}
%\begin{array}{c}
\xi_{0}=\xi_{0}(k,\ell):= (\lambda_{k,\ell}-k^{2}\pi^{2})^{1/2},\qquad
  \xi'_{1}= \xi'_{1}(k,\ell):= (k^{2}\pi^{2}-\frac{\lambda_{k,\ell}}{c_{1}})^{1/2}.
%  \end{array}
\end{equation}
%%%%%%%%%%%%%%%%%%%%%%%%%%%%%%%%%%%%%%%%%%%%%%%%%%%%
 Les fonctions propres dites guid\'ees sont de la forme
\begin{equation}
\label{equation-conc-conf-fonctionpropreguidee-1saut:1}
v_{k,\ell}(x)=a_{k,\ell}\sqrt{2}\sin(k\pi x_{1})\left\lbrace\begin{array}{l}
\sin(\xi_{0}x_{2})\mbox{ si } 0<x_{2}<h_{0},\\
\frac{\sin(\xi_{0}h_{0})}{\sinh(\xi'_{1}(H-h_{0}))}\sinh(\xi'_{1}(H-x_{2}))\mbox{ si }h_{0}<x_{2}<H,
\end{array}\right.
\end{equation}
o\`{u} les coefficients $a_{k,\ell}$ servent \`{a} la normalisation des fonctions propres et les quantit\'es $\lambda_{k,\ell}, \xi_{0}$ et $\xi'_{1}$ sont li\'ees par la relation de dispersion
\begin{equation}
\label{dispersionguidee-1saut:1}
\frac{\tanh(\xi'_{1}(H-h_{0}))}{\xi'_{1}} = -\frac{\tan(\xi_{0}h_{0})}{\xi_{0}}.
\end{equation}
Cette relation exprime la continuit\'e des traces des fonctions propres et de leurs d\'eriv\'ees normales en $x_{2} = h_{0}.$ Noter que les solutions ne v\'erifient ni $ \xi'_{1}=0$ ni $\xi_{0} = 0$. 
\begin{figure}[http!]
\hspace{-.8cm}\includegraphics[scale=.45]{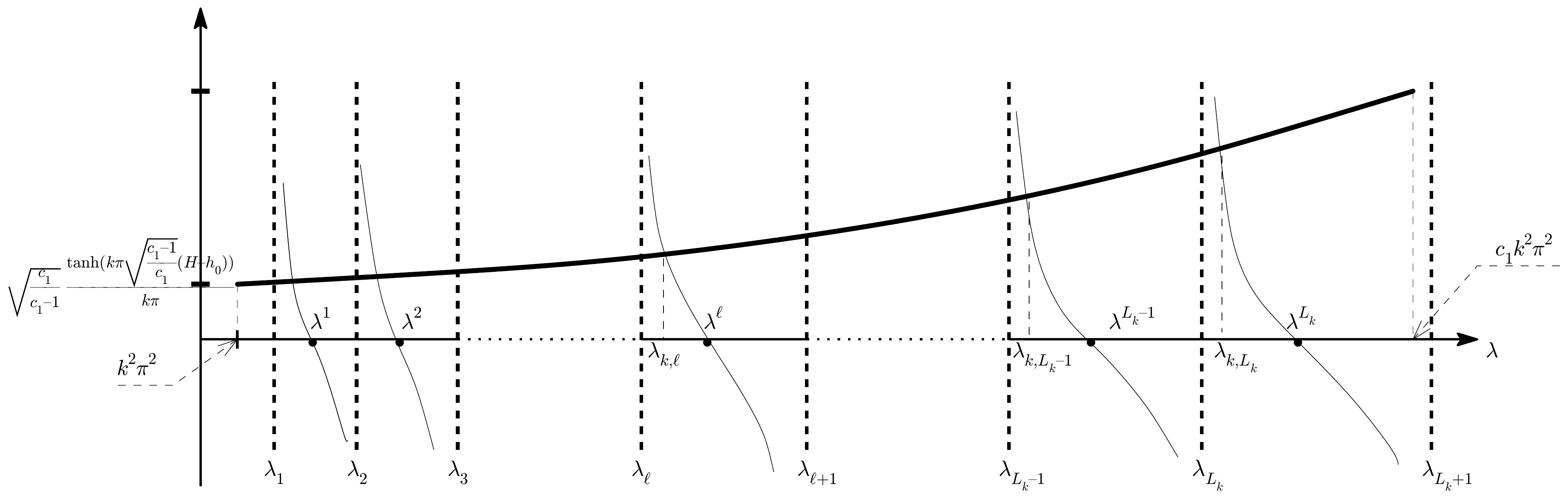}\hfill
\caption{r\'epartition des valeurs propres}
\label{fig:repvp}
\end{figure}
Le membre de gauche de \eqref{dispersionguidee-1saut:1} correspond \`{a} la courbe croissante en gras de la figure \ref{fig:repvp}. Le membre de droite de \eqref{dispersionguidee-1saut:1} est une fonction p\'eriodique en $\xi_{0}h_{0} $ ayant a priori une infinit\'e de branches. Pour chaque $k$ fix\'e, on a $\mathcal{L}_{k}$ valeurs propres guid\'ees $\lambda_{k,\ell}$ class\'ees dans l'ordre croissant $k^{2}\pi^{2}< \lambda_{k,1}< \lambda_{k,2}<\cdots<\lambda_{k,\mathcal{L}_{k}}<c_{1}k^{2}\pi^{2}$ et seules ces $\mathcal{L}_{k}$ branches nous int\'eressent ($k\to \mathcal{L}_{k}$ est une fonction croissante). Elles sont repr\'esent\'ees par une ligne continue fine et les valeurs propres guid\'ees sont les abscisses des points d'intersection des deux courbes de la figure \ref{fig:repvp}.
Dans la figure 
\ref{fig:repvp}
%\ref{fig:figurerepvp}
%, elles correspondent aux abscisses des points d'intersection de la courbe repr\'esentant la fonction croissante $\lambda\mapsto\frac{1}{\xi'_{1}}\tanh(\xi'_{1}(H-h_{0}))$ avec la courbe $\lambda \mapsto -\frac{1}{\xi_{0}}\tan(\frac{\xi_{0}}{2})$
 on a pos\'e 
%\begin{equation}
%\label{notation:2}
$\lambda_{\ell}= \lambda_{\ell}(k) = k^{2}\pi^{2} + (2\ell-1)^{2}\frac{\pi^{2}}{4h_{0}^{2}}$ et $\lambda^{\ell}=\lambda^{\ell}(k):= k^{2}\pi^{2} + \ell^{2}\frac{\pi^{2}}{h_{0}^{2}}, 1\leq \ell\leq \mathcal{L}_{k}+1.
%\end{equation}
$
% o\`{u}, ici, $\xi_{0}= (\lambda-k^{2}\pi^{2})^{1/2},  \xi'_{1}= (k^{2}\pi^{2}-\frac{\lambda}{c_{1}})^{1/2}.$ La figure 3 repr\'esente ceci o\`{u} on a pos\'e

%\begin{maremarque}
%\label{remarque:2} \st{Si $(u_{k_{n},\ell_{n}})_{n}$ est une suite de solutions guid\'ees, il n'est pas possible que $\xi_{0}(k_{n},\ell_{n}) \to 0$ puisque l'on aurait \`{a} la fois $\xi'_{1}\to\infty$ et $\frac{\tanh(\xi'_{1}(H-h_{0}))}{\xi'_{1}} \to -h_{0},$ ce qui est impossible puisque l'expression de gauche est positive.  Par cons\'equent, $\xi'_{1}(k_{n},\ell_{n}) \to +\infty$ et $\xi_{0}$ born\'e impliquent $\sin(\xi_{0}h)\to 0.$ De m\^{e}me, $\sin(\xi_{0}h_{0})\to 0$ implique $\xi'_{1}\to\infty.$}
%\end{maremarque}

%%%%%%%%%%%%%%%%%%%%%%%%%%%%%%%%%%%%%%%%%%%%%%%%%%%%

%%%%%%%%%%%%%%%%%%%%%%%%%%%%%%%%%%%%%%%%%%%%%%%%%%%%
\begin{maremarque}
\label{remarque-distancebarriere:1}
Les quantit\'es $\xi_{0}$ et $\xi'_{1}$ ont une interpr\'etation spectrale g\'eom\'etrique simple si on pose $\kappa:=k^{2}\pi^{2}$ : dans le nouveau rep\`{e}re $(\kappa,\lambda)$
%\footnote{Ce n'est pas celui qui correspond exactement aux coordonn\'ees microlocales usuelles.}
 les paraboles deviennent des droites et la quantit\'e $(\xi'_{1})^{2}$ est \'egale \`{a} $\frac{\sqrt{1+c_{1}^{2}}}{c_{1}} \rho_{k,\ell}$ o\`{u} $\rho_{k,\ell}$ est la distance du point $M_{\lambda_{k,\ell}}=(\kappa, \lambda_{k,\ell})$ \`{a} la barri\`{e}re spectrale d'\'equation $\lambda = c_{1}\kappa.$ De son c\^{o}t\'e, $\sqrt{2}\xi_{0}^{2}$ est la distance de $M_{\lambda_{k},\ell}$ \`{a} la fronti\`{e}re spectrale d'\'equation $\lambda = \kappa.$
\end{maremarque}
C'est donc le rapport
%%%%%%%%%%%%%%%%%%%%%%%%%%%%%%%%%%%%%%%%%%%%%%%%%%%%
%Pour $k$ fix\'e, on pose $\mathcal{L}= \mathcal{L}_{k}:= \max\{\ell;\ell\geq1, \lambda_{k,\ell}<c_{1}k^{2}\pi^{2}\}.$ C'est le plus grand entier $\ell$ tel que $\lambda_{k,\ell}$ est une valeur propre guid\'ee. On posera aussi dans l'article
\begin{eqnarray}
\label{notation-conc-conf-fonction guidee-1saut :3}
%C_{k;\omega} &:= & {\rm vol}(\omega)\left( 1-\frac{\sin(k\pi(l_{2}-l_{1}))}{k\pi(l_{2}-l_{1})}\cos(k\pi(l_{2}+l_{1}))\right),\\
R_{k,\ell;\omega}& := &\frac{\int_{\omega} \vert v_{k,\ell}(x)\vert^{2}{\rm d}x}{\int_{\Omega} \vert v_{k,\ell}(x)\vert^{2}{\rm d}x}
\end{eqnarray}
qui nous int\'eresse pour les fonctions propres $v_{k,\ell}, 1\leq \ell\leq \mathcal{L}_{k}, k\geq 1,$ avec $\omega:= \omega_{1}\times(a,b),  0\leq a<b\leq H$ et $v_{k,\ell}(x):= a_{k,\ell}\sqrt{\frac{2}{L}}\sin(k\pi x_{1})u_{k,\ell}(x_{2}), u_{k,\ell}$ \'etant d\'efinie implicitement par \eqref{equation-conc-conf-fonctionpropreguidee-1saut:1}. 
%Noter qu'il ne faut pas confondre les nombres r\'eels $l_{1}$ et $l_{2}$ avec les entiers $\ell_{n}$, que l'on utilise. 
Notons qu'\'etudier les $v_{k,\ell}$ revient \`{a} \'etudier les $u_{k,\ell}$ car $\frac{\int_{\omega} \vert v_{k,\ell}(x)\vert^{2}{\rm d}x}{\int_{\Omega} \vert v_{k,\ell}(x)\vert^{2}{\rm d}x} = C_{k;\omega}\frac{\int_{\omega} \vert u_{k,\ell}(x)\vert^{2}{\rm d}x}{\int_{\Omega} \vert u_{k,\ell}(x)\vert^{2}{\rm d}x}$ et, si $\omega_{1} = (\alpha,\beta),$ le coefficient $C_{k;\omega} :=  \frac{\beta -\alpha}{L}\left( 1-\frac{\sin(k\pi(\beta-\alpha))}{k\pi(\beta-\alpha)}\cos(k\pi(\alpha+\beta))\right)$ avec $C_{k;\omega}\to_{k\to\infty} \frac{\beta -\alpha}{L}.$
%%%%%%%%%%%%%%%%%%%%%%%%%%%%%%%%
\subsection{Les \'enonc\'es}
\label{soussection-conc-conf-fonctionguidee-1saut:1}
Les suites de fonctions propres associ\'ees \`{a} des valeurs propres v\'erifiant $(k_{n}\pi)^{2}<\lambda_{k_{n},\ell_{n}}<(c_{1}-\varepsilon)(k_{n}\pi)^{2}$ sont sans r\'ep\'etition et par cons\'equent $k_{n}\to\infty$. Un r\'esum\'e rapide pour la quantit\'e $\mathcal{R}_{n}:= R_{k_{n},\ell_{n};\omega}$ est
\begin{enumerate}
 \item si $\overline{\omega}\subset \Omega_{1}$ alors $\mathcal{R}_{n}$ tend exponentiellement vers 0 quand $n\to 0$ (cf.\eqref{equation-preliminaire-1saut:2}).
 \item si $\omega\subset \Omega_{0}$ alors $C_{1}{\rm vol}(\omega)<\mathcal{R}_{n}<C_{2}{\rm vol}(\omega), C_{1}, C_{2}>0.$
 \end{enumerate}
 Ces deux points sont d\'evelopp\'es dans le th\'eor\`{e}me \ref{theoreme-conc-conf-fonctionguidee-1saut:1bis}. Le corollaire \ref{corollaire-conc-conf-fonctionguidee-1saut:2} les pr\'ecise lorsque l'indice $\ell_{n}$ est fix\'e et la proposition \ref{proposition-conc-conf-fonctionguidee-1saut:1} \'etudie $\mathcal{R}_{n}$ pour la suite des valeurs propres la plus proche de $c_{1}(k\pi)^{2}.$ De plus, le lemme \ref{lemme-conc-conf-fonctionguidee-1saut:2} permet de convertir les coefficients de d\'ecroissance en terme de $\lambda_{k_{n},\ell_{n}}$ au lieu de $\xi'_{1}.$ Pour le cas o\`{u} $\omega$ est \`{a} cheval sur l'interface $S_{0}$ on se reportera \`{a} la Remarque \ref{remarque-conc-conf-fonctionguidee-1saut:0bis}.
 \begin{montheo}[cas guid\'e avec un saut]
\label{theoreme-conc-conf-fonctionguidee-1saut:1bis}
Soit l'ouvert $\omega:= \omega_{1}\times(a,b)\subset\Omega$ et une suite de fonctions propres guid\'ees $(v_{k_{n},\ell_{n}})_{n\geq 1}.$ On a\\
\begin{eqnarray}
\!\!\bullet\, \mbox{si }\omega\subset\Omega_{1},& \left\lbrace
\begin{array}{ll}
\label{equation-conc-conf-fonctionpropreguidee-1saut:3}
 (\xi'_{1}(k_{n},\ell_{n})\to\infty) \!\!\!&\Longrightarrow \frac{\Vert v_{k_{n},\ell_{n}}\Vert^{2}_{L^{2}(\omega)}}{\Vert v_{k_{n},\ell_{n}}\Vert^{2}_{L^{2}(\Omega)}}  \sim  \frac{{\rm vol}(\omega_{1})}{h_{0}}\frac{\sin^{2}(\xi_{0}h_{0})}{1-\frac{\sin(2\xi_{0}h_{0})}{2\xi_{0}h_{0}}}    \frac{e^{-2\xi'_{1}(a-h_{0})}}{\xi'_{1}},\\
 (\sup_{n}\xi'_{1}(k_{n},\ell_{n})<\infty)  \!\!\!&\Longrightarrow \inf_{n}\frac{\Vert v_{k_{n},\ell_{n}}\Vert^{2}_{L^{2}(\omega)}}{\Vert v_{k_{n},\ell_{n}}\Vert^{2}_{L^{2}(\Omega)}}>0.
 \end{array}\right.\\
\label{equation-conc-conf-fonctionpropreguidee-1saut:3''}
\!\!\!\bullet\, \mbox{ si }\omega\subset\Omega_{0},&
\inf_{n}\frac{\Vert v_{k_{n},\ell_{n}}\Vert^{2}_{L^{2}(\omega)}}{\Vert v_{k_{n},\ell_{n}}\Vert^{2}_{L^{2}(\Omega)}}>0.
\end{eqnarray}
\end{montheo}
\begin{maremarque}
\label{remarque-conc-conf-fonctionguidee-1saut:0}
%Nous n'avons consid\'er\'e que des suites de fonctions propres sans r\'ep\'etition.
On notera que $\xi'_{1}(k_{n},\ell_{n})\to\infty$ si, asymptotiquement, $\ell_{n}\leq \mathcal{L}_{k_{n}}-1.$ En effet, si $\ell_{n}\leq \mathcal{L}_{k_{n}}-1,$ on a $c_{1}(\xi'_{1}(k_{n},\ell_{n}))^{2} \geq c_{1}(\xi'_{1}(k_{n},\mathcal{L}_{k_{n}}-1))^{2} = c_{1}k_{n}^{2}\pi^{2} -\lambda_{k_{n},\mathcal{L}_{k_{n}}-1}>\lambda_{\mathcal{L}_{k_{n}}}-\lambda^{\mathcal{L}_{k_{n}}-1} =(\frac{\pi}{2h_{0}})^{2}(4\mathcal{L}_{k_{n}}-3)\to\infty$ puisque $\mathcal{L}_{k_{n}}\to \infty$. Par suite, si $\varepsilon$ est fix\'e tel que $0<\varepsilon<c_{1}-1$, l'examen de la figure \ref{fig:repvp} montre que la condition $\ell_{n}\leq \mathcal{L}_{k_{n}}-1$ est alors v\'erifi\'ee pour $n$ grand si $(k_{n}\pi)^{2}<\lambda_{k_{n},\ell_{n}}\leq (c_{1}-\varepsilon)(k_{n}\pi)^{2}.$ 
%\end{enumerate}
\end{maremarque}
%%%%%%%%%%%%%%%%%%%%%%%%%%%%%%%
\begin{moncor}
\label{corollaire-conc-conf-fonctionguidee-1saut:2}
Soit $\omega:= \omega_{1}\times(a,b)\subset\Omega$ et une suite de  fonctions propres guid\'ees $(v_{k,\ell})_{k}$, l'indice $\ell$ \'etant fix\'e. On a alors\\
$\bullet$ Cas 1 : $\omega\subset\Omega_{1}$, i.e. $h_{0}\leq a<b,$ \\
\begin{equation}
\label{equation-conc-conf-fonctionpropreguidee-1saut:2}
a_{1}\lambda_{k,\ell}^{-\frac{3}{2}}e^{-a_{2}\lambda_{k,\ell}^{\frac{1}{2}}}\Vert v_{k,\ell}\Vert^{2}_{L^{2}(\Omega)}\sim_{k\to\infty} \Vert v_{k,\ell}\Vert^{2}_{L^{2}(\omega)},
\end{equation}
o\`{u} $a_{1} = \frac{{\rm vol}(\omega_{1})(\ell\pi)^{2}}{h_{0}^{3}}\left(\frac{c_{1}}{c_{1}-1}\right)^{\frac{3}{2}}$ et $a_{2} = 2(a-h_{0})\sqrt{\frac{c_{1}-1}{c_{1}}}.$\\
$\bullet$ Cas 2 : $\omega\subset\Omega_{0},$  i.e. $0\leq a<b<h_{0},$
\begin{equation}
\label{equation-conc-conf-fonctionpropreguidee-1saut:5}
\frac{\Vert v_{k,\ell}\Vert^{2}_{L^{2}(\omega)}}{\Vert v_{k,\ell}\Vert^{2}_{L^{2}(\Omega)}}\sim_{k\to\infty} \frac{1}{h_{0}}{\rm vol}(\omega)\left(1 - \frac{\sin(\ell\frac{\pi}{h_{0}}(b-a))}{2\ell\frac{\pi}{h_{0}}(b-a)} \cos(2\ell\frac{\pi}{h_{0}}(a+b)) \right).
\end{equation}
\end{moncor}
\begin{maremarque}
\label{remarque-conc-conf-fonctionguidee-1saut:1}
Il est important de noter que pour une suite v\'erifiant $\xi'_{1}\to\infty$, \eqref{equation-conc-conf-fonctionpropreguidee-1saut:3} et \eqref{equation-conc-conf-fonctionpropreguidee-1saut:6} impliquent 
 que la norme $L^{2}$ se concentre dans $ \Omega_{0}$ puisque $\frac{\Vert v_{k_{n},\ell_{n}}\Vert_{L^{2}(\Omega_{0})}}{\Vert v_{k_{n},\ell_{n}}\Vert_{L^{2}(\Omega)}} = 1- O(\frac{1}{(\xi'_{1})^{\frac{1}{2}}}).$ 
La suite $(v_{k,\mathcal{L}_{k}})_{k}$ constitue une transition entre les fonctions propres guid\'ees et celles qui ne le sont pas. Ce qui suit donne une autre approche du comportement de $(v_{k,\mathcal{L}_{k}})_{k}$. et illustre cette transition.
\end{maremarque}
%%%%%%%%%%%%%%%%%%%%%%%%%%%%%%%%%%%%%%%%%%%%%%%%%%%%
\begin{mapropo}
\label{proposition-conc-conf-fonctionguidee-1saut:1} 
Supposons $h_{0}=\frac{1}{2}$\footnote{ Ce choix ne change pas la nature du r\'esultat puisque l'on ne fixe pas en m\^{e}me temps la valeur num\'erique de $c_{1}>1.$}, $\bar{\omega}\subset\Omega_{1}$
\begin{enumerate}
\item
\label{itemxi-1:6ter}
%Si $\sqrt{c_{1}-1}\in \mathbb{R}\setminus \mathbb{N},$ 
Il existe des constantes positives $C'_{1}, C'_{2}$ telles que 
%la suite de fonctions propres guid\'ees correspondant aux valeurs propres $\lambda_{k,\mathcal{L}_{k}}$ v\'erifie l'encadrement
\begin{equation}
\label{equation-conc-conf-fonctionguidee-1saut-itemxi-1:6ter}
C'_{1}e^{-C'_{2}\lambda_{k,\mathcal{L}_{k}}^{\frac{1}{4}}}\Vert v_{k,\mathcal{L}_{k}}\Vert^{2}_{L^{2}(\Omega)} \leq \Vert v_{k,\mathcal{L}_{k}}\Vert^{2}_{L^{2}(\omega)} 
%\leq  \Vert u_{k,\mathcal{L}_{k}}\Vert^{2}_{L^{2}(\Omega)}.
\end{equation}
\item
 \label{itemxi-1:0}
 Si $\sqrt{c_{1}-1}\in\mathbb{N},$ il existe des constantes positives $C_{1}, C'_{1},C_{2},C'_{2}$ telles que
%la suite de fonctions propres guid\'ees correspondant aux valeurs propres $\lambda_{k,\mathcal{L}_{k}}$ v\'erifie l'encadrement
\begin{equation}
\label{equation-conc-conf-fonctionguidee-1saut-xi-1:6''}
C'_{1}e^{-C'_{2}\lambda_{k,\mathcal{L}_{k}}^{\frac{1}{4}}}\Vert v_{k,\mathcal{L}_{k}}\Vert^{2}_{L^{2}(\Omega)} \leq \Vert v_{k,\mathcal{L}_{k}}\Vert^{2}_{L^{2}(\omega)} 
\leq  C_{1}e^{-C_{2}\lambda_{k,\mathcal{L}_{k}}^{\frac{1}{4}}}\Vert v_{k,\mathcal{L}_{k}}\Vert^{2}_{L^{2}(\Omega)}.
\end{equation}
\item
\label{itemxi-1:6quarto}
Si $\sqrt{c_{1}-1}\in \mathbb{R}\setminus \mathbb{N},$ il existe des constantes positives $C_{1}, C'_{1},C_{2},C'_{2}$ et une sous-suite de fonctions propres guid\'ees correspondant \`{a} des valeurs propres $\lambda_{k,\mathcal{L}_{k}}$ v\'erifiant l'encadrement
\begin{equation}
\label{xi-1:6quarto}
C'_{1}e^{-C'_{2}\lambda_{k,\mathcal{L}_{k}}^{\frac{1}{4}}}\Vert v_{k,\mathcal{L}_{k}}\Vert^{2}_{L^{2}(\Omega)} \leq \Vert v_{k,\mathcal{L}_{k}}\Vert^{2}_{L^{2}(\omega)} \leq C_{1}e^{-C_{2}\lambda_{k,\mathcal{L}_{k}}^{\frac{1}{4}}}\Vert v_{k,\mathcal{L}_{k}}\Vert^{2}_{L^{2}(\Omega)}.
\end{equation}
\end{enumerate}
\end{mapropo}
%%%%%%%%%%%%%%%%%%%%%%%%%%%%%%%%%%%%
%%%%%%%%%%%%%%%%%%%%%%%%%%%%%%%%%%%%
\begin{maremarque}
\label{remarque-conc-conf-fonctionguidee-1saut:1'}
Si $\sqrt{c_{1}-1}\in\mathbb{N},$ il est certain que la courbe $\lambda= c_{1}\frac{k^{2}\pi^{2}}{L^{2}}$ d\'efinit la fronti\`{e}re entre les valeurs propres guid\'ees et celles qui ne le sont pas. C'est moins clair si $\sqrt{c_{1}-1}\not\in\mathbb{N}$ car on ne peut exclure que $\sup_{n}\xi'_{1}(k_{n},\ell_{n})<\infty.$
\end{maremarque}
\begin{moncor}
\label{corolllaire-conc-conf-fonctiongiuidee-1saut:1}
Pour la base $(v_{k,\ell})$ de notre mod\`{e}le, l'in\'egalit\'e \eqref{Laurent-Leautaud:1} est valable et l'exposant $\lambda_{k,\ell}^{\frac{1}{2}}$ ne peut \^{e}tre modifi\'e sauf pour certaines sous-familles et certains ouverts $\omega.$ 
\end{moncor}
Ce corollaire est une cons\'equence imm\'ediate de la proposition \ref{proposition-conc-conf-fonctionguidee-1saut:1} : prendre une suite de fonctions propres satisfaisant cette proposition et $\omega\subset\Omega_{1}.$ Le r\'esultat qui suit est \`{a} la base du corollaire \ref{corollaire-conc-conf-fonctionguidee-1saut:2}.
%%%%%%%%%%%%%%%%%%%%%%%%%%%%%%%%%%%%
\begin{monlem}[Lemme pr\'eparatoire]
\label{lemme-conc-conf-fonctionguidee-1saut:2}
On consid\`{e}re une suite de fonctions propres guid\'ees $(v_{k_{n},\ell_{n}})_{n}$.
\begin{enumerate}
\item il n'est pas possible que $\xi_{0}(k_{n},\ell_{n}) \to 0.$
\item
\label{equationcomportementxi1item:3bis} 
 ($\xi'_{1}(k_{n},\ell_{n}) \to +\infty$ et $\xi_{0}$ born\'e) $\Longrightarrow(\sin(\xi_{0}h_{0})\to 0)\Longrightarrow (\xi'_{1}\to\infty).$ 

\item
\label{equationcomportementxi1item:3'} 
($\xi'_{1}$ born\'e) $\Longrightarrow \inf \vert\xi_{0}\cos(\xi_{0}h_{0})\vert>0,\sup \vert\xi_{0}\cos(\xi_{0}h_{0})\vert<\infty$ et $\vert\sin(\xi_{0}((k_{n},\ell_{n})h_{0})\vert\to 1.$
 \item
 \label{equationcomportementxi1item:2}
 Si $\ell_{n}=\ell$ fix\'e, on a $\lambda_{k_{n},\ell} - \lambda^{\ell} = o(1)$ o\`{u} $ \lambda^{\ell} =\lambda^{\ell}(k_{n}):= k_{n}^{2}\pi^{2} + \ell^{2}\frac{\pi^{2}}{h^{2}_{0}}, $ ce qui implique que $\xi_{0}(k_{n},\ell)\to \ell\frac{\pi}{h_{0}}$ et $\xi'_{1}(k_{n},\ell)\to\infty$ quand $n\to\infty.$\\
 On a plus pr\'ecis\'ement
\begin{eqnarray}
\label{equationcomportementxi1:1}
\sqrt{\frac{c_{1}-1}{c_{1}}}\lambda_{k_{n},\ell}^{\frac{1}{2}} - \xi'_{1}(k_{n},\ell) & = & \frac{\pi\ell^{2}}{2h_{0}^{2}}\sqrt{\frac{c_{1}}{c_{1}-1}}\frac{1}{k_{n}} + o(\frac{1}{k_{n}}),\\
\label{equationcomportementxi0:1}
\sin(\xi_{0}h_{0}) & = & -\frac{1}{\pi h_{0}}\sqrt{\frac{c_{1}}{c_{1}-1}}\frac{1}{k_{n}} + o(\frac{1}{k_{n}})
\end{eqnarray}
 \item Soit un entier $p\geq 2$ fix\'e. Si $\ell_{n}= \mathcal{L}_{k_{n}} -p$ alors $\xi'_{1}(k_{n},\ell_{n})\approxeq \lambda_{k_{n},\ell_{n}}^{\frac{1}{4}}.$
\end{enumerate}
\end{monlem}
\begin{proof} 
Les trois premiers r\'esultats sont des cons\'equences directes de la relation de dispersion \eqref{dispersionguidee-1saut:1}. D'abord, comme $\xi'_{1}(k_{n},\ell_{n}) \to \infty$ si $\xi_{0}(k_{n},\ell_{n}) \to 0,$ on aurait $0 = -h_{0},$ ce qui r\`{e}gle le premier point. La premi\`{e}re implication du point suivant est claire. Pour la seconde, supposant $\xi'_{1}$ born\'e,  il y a une contradiction connaissant le point pr\'ec\'edent. 
Dans le troisi\`{e}me alin\'ea, $\xi'_{1}$ born\'e ne peut arriver que si $\ell_{n}=\mathcal{L}_{k_{n}}$ pour les grandes valeurs de $n$. Ainsi $\xi_{0}\to\infty$ et, comme la $\liminf$ du c\^{o}t\'e droit de \eqref{dispersionguidee-1saut:1} doit \^{e}tre strictement positive il faut aussi que $\cos(\xi_{0}h_{0})\to 0$ et on conclut. Pour l'item \ref{equationcomportementxi1item:2}, voir l'annexe \ref{annexe-Lemme preparatoire:1}. Pour la derni\`{e}re affirmation on remarque que $c_{1}k^{2}_{n}\pi^{2} -\lambda^{\mathcal{L}_{k_{n}}-p}<c_{1}(\xi'_{1})^{2}<c_{1}k^{2}_{n}\pi^{2} -\lambda_{\mathcal{L}_{k_{n}}-p},$ i.e.
\begin{multline*} 
(\sqrt{c_{1}-1}k_{n} +2\mathcal{L}_{k_{n}} -2p)(\sqrt{c_{1}-1}k_{n} -2\mathcal{L}_{k_{n}} +2p)\pi^{2}<c_{1}(\xi'_{1})^{2}\\
 < (\sqrt{c_{1}-1}k_{n} +2\mathcal{L}_{k_{n}} -2p-1)(\sqrt{c_{1}-1}k_{n} -2\mathcal{L}_{k_{n}} +2p+1)\pi^{2}.
\end{multline*}
Pour \'evaluer $(\sqrt{c_{1}-1}k_{n} -2\mathcal{L}_{k} +2p)$ et $(\sqrt{c_{1}-1}k_{n} -2\mathcal{L}_{k_{n}} +2p+1)$ on utilise les trois cas de l'annexe D. Dans les cas 1 et 2 de cette annexe, on a $2p\leq \sqrt{c_{1}-1}k_{n} -2\mathcal{L}_{k_{n}} +2p<1+2p$ et, dans le dernier cas, on a $1 + 2p <\sqrt{c_{1}-1}k_{n} -2\mathcal{L}_{k_{n} }+2p<2 + 2p.$ Comme $k^{2}_{n}$ est comparable \`{a} $\lambda_{k_{n},\ell_{n}}$ (cf. la remarque C.1) et $\mathcal{L}_{k_{n}}$ \`{a} $k_{n}$ on peut conclure.
\end{proof}

%%%%%%%%%%%%%%%%%%%%%%%%%%%%%%%%%%%%%%%%%%%%%%%%%%%%
%%%%%%%%%%%%%%%%%%%%%%%%%%%%%%%%%%%%%%%%%%%%%%%%%%%% 
\subsection{Les preuves}
\label{soussection-conc-conf-fonctionguidee-1saut:2}
%%%%%%%%%%%%%%%%%%%%%%%%%%%%%%%
%%%%%%%%%%%%%%%%%%%%%%%%%%%%%%%
La preuve de la proposition \ref{proposition-conc-conf-fonctionguidee-1saut:1} suit de \eqref{equation-conc-conf-fonctionpropreguidee-1saut:3} et d'un 
partie technique d\'evelopp\'ee \`{a} l'annexe \ref{annexe-comportement-asymptotique-fonctionpropreguidee1indications}.
\subsubsection{D\'ecroissance dans $\omega\subset\Omega_{1}$ : \eqref{equation-conc-conf-fonctionpropreguidee-1saut:3} 
 et cas 1 du corollaire \ref{corollaire-conc-conf-fonctionguidee-1saut:2} }
Le terme $a_{k,\ell}\sqrt{2}\sin(k\pi x_{1})$ n'a qu'une influence marginale dans les r\'esultats qui nous int\'eressent. Rempla\c{c}ant par commodit\'e $x_{2}$ par $x$ puis, transformant  \eqref{equation-conc-conf-fonctionpropreguidee-1saut:1} \`{a} l'aide de \eqref{dispersionguidee-1saut:1}, on obtient
\begin{equation}
\label{equation- conc-conf-fonctionguidee-1saut-reformatage:1}
u_{k,\ell}(x) = -\frac{\xi_{0}}{\xi'_{1}}\cos(\xi_{0}h_{0})\frac{\sinh(\xi'_{1}(H-x))}{\cosh(\xi'_{1}(H-h_{0}))}, x\in (h_{0}, H),
\end{equation}
ce qui permet de r\'egler le cas o\`{u} $\xi'_{1}$ est born\'e et ne tend pas vers 0 gr\^{a}ce \`{a} l'alin\'ea \ref{equationcomportementxi1item:3'} du lemme \ref{lemme-conc-conf-fonctionguidee-1saut:2}. Si $\xi'_{1}$ s'approche vers 0,  un calcul direct montre que $R_{k_{n},\ell_{n};\omega}\approxeq {\rm vol}(\omega).$ Si $\xi'_{1}\to\infty$, on a
\begin{eqnarray}
\label{equation- conc-conf-fonctionguidee-1saut-reformatage:2}
\!\!\!\!\!\!\int_{h_{0}}^{H}\vert u(x)\vert^{2}{\rm d}x \!\!\!& = &\!\!\! \frac{\xi^{2}_{0}}{2(\xi'_{1})^{3}}\cos^{2}(\xi_{0}h_{0})\tanh(\xi'_{1}(H-h_{0}))( 1 +o(1) ),\\
\!\!\!\!\!\!\int_{a}^{b}\vert u(x)\vert^{2}{\rm d}x \!\!\!& = &\!\!\! \!\!\!\frac{\xi^{2}_{0}}{2(\xi'_{1})^{3}}\frac{\cos^{2}(\xi_{0}h_{0})}{\cosh^{2}(\xi'_{1}(H-h_{0}))}\sinh(\xi'_{1}(b-a))\!\cosh(2\xi'_{1}(H-\frac{a+b}{2}))(1\!+\!o(1)\!).
\end{eqnarray} 
Comme $\int_{0}^{h_{0}}\vert u(x)\vert^{2}{\rm d}x = \frac{h_{0}}{2}(1-\frac{\sin(2\xi_{0}h_{0})}{2\xi_{0}h_{0}})$, l'\'equivalence \eqref{equation-conc-conf-fonctionpropreguidee-1saut:3} s'en suit.
                                                       %%%%%%%%%%%%%%%%%%%%%%%%%%%
%%%%%%%%%%%%%%%%%%%%%%%%%%%%%%%
Avec l'hypoth\`{e}se $\ell_{n} = \ell$ fix\'e, l'\'equivalence \eqref{equation-conc-conf-fonctionpropreguidee-1saut:2} d\'ecoule du lemme \ref{lemme-conc-conf-fonctionguidee-1saut:2}, point 4, et du th\'eor\`{e}me \ref{theoreme-conc-conf-fonctionguidee-1saut:1bis}. De plus, \eqref{equation-conc-conf-fonctionpropreguidee-1saut:2} est plus pr\'ecise que l'estimation \eqref{equation-conc-conf-fonctionpropreguidee-1saut:3} puisque $\sin(\xi_{0}h_{0}) \to 0$ ( point 4 du lemme \ref{lemme-conc-conf-fonctionguidee-1saut:2}).\\
Noter que les quantit\'es $1-\frac{\sin(2\xi_{0}h_{0})}{2\xi_{0}h_{0}}$ ne tendent pas vers 0 dans \eqref{equation-conc-conf-fonctionpropreguidee-1saut:3} et que, si $\omega$ touche l'interface, l'exponentielle dispara\^{\i}t dans \eqref{equation-conc-conf-fonctionpropreguidee-1saut:2} et \eqref{equation-conc-conf-fonctionpropreguidee-1saut:3}.
%%%%%%%%%%%%%%%%%%%%%%%%%%%%%%%
\subsubsection{Concentration dans $\omega\subset\Omega_{0}$ : \eqref{equation-conc-conf-fonctionpropreguidee-1saut:3''}
 et cas 2  du corollaire \ref{corollaire-conc-conf-fonctionguidee-1saut:2}}
 Les r\'esultats annonc\'es suivront de la
 \begin{mapropo}
 Soit l'ouvert $\omega:= \omega_{1}\times(a,b)\subset\Omega_{0}$ et une suite de fonctions propres guid\'ees $(v_{k_{n},\ell_{n}})_{n\geq 1}.$ On a
%\begin{enumerate}
\begin{equation}
\label{equation-conc-conf-fonctionpropreguidee-1saut:6}
\!\mathbf{1}. \;(\xi'_{1}\to\infty)\Longrightarrow \!\!\left(\forall\varepsilon>0, \exists N, n\geq N, \frac{{\rm vol}(\omega)}{h_{0}}(1-\varepsilon)\leq \frac{\Vert v_{k_{n},\ell_{n}}\Vert^{2}_{L^{2}(\omega)}}{\Vert v_{k_{n},\ell_{n}}\Vert^{2}_{L^{2}(\Omega)}}\leq 
 \frac{5}{2} \frac{{\rm vol}(\omega)}{h_{0}}(1+\varepsilon)\!\right)\!.
\end{equation}
\!{\bf 2}. Pour la suite de fonctions propres guid\'ees associ\'ees aux $\lambda_{k,\mathcal{L}_{k}}$ , on a 
\begin{equation}
\label{equation-conc-conf-fonctionpropreguidee-1saut:7}
\forall\varepsilon>0, \exists N, \forall k>N\Longrightarrow  \frac{{\rm vol}(\omega)}{h_{0} + \frac{2}{3}(H-h_{0})}(1-\varepsilon) \leq\frac{\Vert v_{k,\mathcal{L}_{k}}\Vert^{2}_{L^{2}(\omega)}}{\Vert v_{k,\mathcal{L}_{k}}\Vert^{2}_{L^{2}(\Omega)}}\leq \frac{{\rm vol}(\omega)}{h_{0}}(1+\varepsilon).
\end{equation}
Si, de plus, $\xi'_{1}\to\infty$ pour une suite extraite, celle-ci v\'erifie
\begin{equation}
\label{equation-conc-conf-fonctionpropreguidee-1saut:8}
 \frac{\Vert v_{k,\mathcal{L}_{k}}\Vert^{2}_{L^{2}(\omega)}}{\Vert v_{k,\mathcal{L}_{k}}\Vert^{2}_{L^{2}(\Omega)}}\sim \frac{1}{h_{0}}{\rm vol}(\omega).
\end{equation}
%\end{enumerate}
\end{mapropo}
Comme $0\leq  a<b\leq h_{0}$, on part de
\begin{equation}
\label{equation-conc-conf-fonctionguidee-1saut:9}
R_{k_{n},\ell_{n},\omega}={\rm vol}(\omega)\frac{(1 + O(\frac{1}{k}))\left(1-\frac{\sin(\xi_{0}(b-a))}{\xi_{0}(b-a)}\cos(\xi_{0}(a+b)) \right)}{h_{0}(1-\frac{\sin(2\xi_{0}h_{0})}{2\xi_{0}h_{0}}) + \sin^{2}(\xi_{0}h_{0})\left( \frac{1}{\xi'_{1}\tanh(\xi'_{1}(H-h_{0}))} -\frac{H-h_{0}}{\sinh^{2}(\xi'_{1}(H-h_{0}))}\right)}.
\end{equation}
$\bullet$ L'encadrement \eqref{equation-conc-conf-fonctionpropreguidee-1saut:6} est une cons\'equence de  $\xi'_{1}\to\infty$ et $h_{0}(1-\frac{1}{2\pi})<h_{0}(1- \frac{\sin(2\xi_{0}h_{0}}{2\xi_{0}h_{0}})<h_{0}$ car $\liminf \xi_{0}h_{0}\geq 2\pi$ puisque $h_{0}\xi_{0}\geq h_{0} \xi_{0}(k_{n},1)\sim\pi$ d'apr\`{e}s l'item \ref{equationcomportementxi1item:2} du lemme \ref{lemme-conc-conf-fonctionguidee-1saut:2}.\\
$\bullet$ \eqref{equation-conc-conf-fonctionpropreguidee-1saut:8} vient de $ \frac{\Vert v_{k,\mathcal{L}_{k}}\Vert^{2}_{L^{2}(\omega)}}{\Vert v_{k,\mathcal{L}_{k}}\Vert^{2}_{L^{2}(\Omega)}}\sim \frac{1}{h_{0}}{\rm vol}(\omega)$ si $\xi'_{1}\to \infty,$ ce qui arrive lorsque $\sqrt{c_{1}-1}$ est un entier.  Pour d'autres conditions suffisantes voir aussi les lemmes \ref{lemme-comportement-asymptotique-fonctionpropreguidee:3} et \ref{lemme-comportement-asymptotique-fonctionpropreguidee:4}. Pour compl\'eter \eqref{equation-conc-conf-fonctionpropreguidee-1saut:7}, il faut \'etudier la possibilit\'e 
$\xi'_{1}\not\to\infty$ : utiliser l'in\'egalit\'e $0<\frac{1}{\sinh^{2}(\xi'_{1}(H-h_{0}))}(\frac{\sinh(2\xi'_{1}(H-h_{0}))}{2\xi'_{1}(H-h_{0})} -1)\leq \frac{2}{3}.$\\
$\bullet$ Pour \eqref{equation-conc-conf-fonctionpropreguidee-1saut:5}, on utilise $\xi_{0}\to\ell\frac{\pi}{h_{0}}$ (lemme \ref{lemme-conc-conf-fonctionguidee-1saut:2}, item \ref{equationcomportementxi1item:2} ) et $\xi'_{1}\to\infty$ qu'il suffit d'appliquer dans \eqref{equation-conc-conf-fonctionguidee-1saut:9}.
%%%%%%%%%%%%%%%%%%%%%%%%%%%%%%%
\begin{maremarque}
\label{remarque-conc-conf-fonctionguidee-1saut:0bis}
Si $\omega$ est une bande ouverte \`{a} cheval sur l'interface, i.e. $\omega:=(0,L)\times (h_{0}-\alpha_{-},h_{0}+\alpha_{+}), 0<\alpha_{-}<h_{0}, 0<\alpha_{+}<H-h_{0},$ et posant $\omega_{-}:= \omega\cap\Omega_{0}, \omega_{+}:= \omega\cap\Omega_{1},$ on peut \'evaluer le rapport $  \frac{\Vert v_{k_{n},\ell_{n}}^{2}\Vert_{L^{2}(\omega_{+})}^{2}}{\Vert v_{k_{n},\ell_{n}}\Vert^{2}_{L^{2}(\omega_{-})}}$ \`{a} partir du th\'eor\`{e}me \ref{theoreme-conc-conf-fonctionguidee-1saut:1bis}. On peut \^{e}tre plus pr\'ecis : 
 %et soit  $(v_{k_{n},\ell_{n}})_{n}$ une suite de fonctions propres guid\'ees.\\
si $\xi'_{1}\to\infty$
 il existe une constante positive $\bar{a} = \bar{a}(\alpha_{-} )<1$ telle que, pour tout $n,$ on a
\begin{equation}
\label{equation-concentration-interface-onde-guidee:2}
\frac{1}{\alpha_{-}(1 + \bar{a})}\frac{\sin^{2}(\xi_{0}h_{0})}{\xi'_{1}}(1 + o(1))\leq \frac{\Vert v_{k_{n},\ell_{n}}^{2}\Vert_{L^{2}(\omega_{+})}^{2}}{\Vert v_{k_{n},\ell_{n}}\Vert^{2}_{L^{2}(\omega_{-})}}\leq \frac{1}{\alpha_{-}(1-\bar{a})}\frac{\sin^{2}(\xi_{0}h_{0})}{\xi'_{1}}(1 + o(1))
\end{equation}
et, pour revenir \`{a} un $\omega = \omega_{1}\times(h_{0}-\alpha_{-},h_{0}+\alpha_{+})$ dans \eqref{equation-concentration-interface-onde-guidee:2}, il suffit de multiplier aux extr\'emit\'es gauche et  droite par le coefficient $C_{k;\omega}=2\int_{\omega_{1}} \sin^{2}(k\pi x_{1}){\rm d}x_{1}.$ 
\end{maremarque}
\noindent
Pour cette remarque, on part de $\frac{\int_{\omega_{+}}u_{k_{n},\ell_{n}}^{2}(x){\rm d}x}{\int_{\omega_{-}}u_{k_{n},\ell_{n}}^{2}(x){\rm d}x}   =\frac{\alpha_{+}}{\alpha_{-}}\frac{\sin^{2}(\xi_{0}h_{0})}{\sinh^{2}(\xi'_{1}(H-h_{0}))} \frac{\frac{\sinh(\xi'_{1}\alpha_{+})}{\xi'_{1}\alpha_{+}}\cosh(\xi'_{1}(2H-2h_{0}-\alpha_{+}))-1}{1-\frac{\sin(\xi_{0}\alpha_{-})}{\xi_{0}\alpha_{-}}\cos(\xi_{0}(2h_{0}-\alpha_{-}))}$ qui devient  
$ \frac{\sin^{2}(\xi_{0}h_{0})}{\xi'_{1}\alpha_{-}}\frac{1 + o(1)}{1-\frac{\sin(\xi_{0}\alpha_{-})}{\xi_{0}\alpha_{-}}\cos(\xi_{0}(2h_{0}-\alpha_{-}))}$ si $\xi'_{1}\to \infty.$ Comme $\xi_{0}(k_{n},\ell_{n})$ ne s'approche pas de 0 car $\xi_{0}(k_{n},\ell_{n})\geq \xi_{0}(k_{n},1)\to \frac{\pi}{h}$ d'apr\`{e}s le lemme \ref{lemme-conc-conf-fonctionguidee-1saut:2}, il existe une fonction d\'ecroissante $\alpha_{-}\to \bar{a}(\alpha_{-})\in (0,1)$
 telle que $\vert\frac{\sin(\xi_{0}\alpha_{-})}{\xi_{0}\alpha_{-}}\vert\leq \bar{a}(\alpha_{-})<1,$ d'o\`{u} l'encadrement \eqref{equation-concentration-interface-onde-guidee:2}. De plus, si $\ell_{n} = \ell$ fix\'e, on a $\sin^{2}(\xi_{0}h_{0})\sim \frac{c_{1}}{\pi^{2}h^{2}_{0}(c_{1}-1)}\frac{1}{k^{2}} $ venant de \eqref{equationcomportementxi0:1}.
L'\'epaisseur $\alpha_{+}$ est fix\'ee et n'intervient que dans $o(1)$. L'item \ref{equationcomportementxi1item:3bis} du lemme \ref{lemme-conc-conf-fonctionguidee-1saut:2} permet aussi de consid\'erer le cas o\`{u} la suite $(\ell_{n})_{n}$ est seulement born\'ee.
 %%%%%%%%%%%%%%%%%%%%%%%%%%%%%%%%%%%%%%%%%%%%%%%%%%%%
%%%%%%%%%%%%%%%%%%%%%%%%%%%%%%%%%%%%%%%%%%%%%%%%%%%%
\section{Concentration des ondes guid\'ees du mod\`{e}le \`{a} deux sauts}
\label{section-modele2sauts:1}
%%%%%%%%%%%%%%%%%%%%%%%%%%%%%%%%%%%%%%%%%%%%%%%%%%%%
%%%%%%%%%%%%%%%%%%%%%%%%%%%%%%%%%%%%%%%%%%%%%%%%%%%%
Nous continuons \`{a} remplacer $x_{2}$ par $x.$ L'\'equation \`{a} satisfaire dans $L^{2}(0,H)$ est
$
-(cu')' + (c\frac{k^{2}\pi^{2}}{L^{2}} -
\lambda)u,$ avec les conditions au bord $u(0) = u(H) =0,
$
et les notations suivantes :
\begin{equation}
\label{suite(stratification3valeurs)notations:2}
\begin{array}{c}
\!\!\!\!\!\!\!\!u_{0}(x)\!: =u(x) \mbox{ si } 0<x<h_{0}, u_{1}(x)\!: =u(x) \mbox{ si } h_{0}<x<h_{1}, u_{2}(x)\!: =u(x) \mbox{ si } h_{1}<x<H,\\
\xi_{0}^{2} := \frac{\lambda}{c_{0}} - \frac{k^{2}\pi^{2}}{L^{2}}, \xi_{1}^{2} := \frac{\lambda}{c_{1}} - \frac{k^{2}\pi^{2}}{L^{2}}, (\xi'_{1})^{2} :=  \frac{k^{2}\pi^{2}}{L^{2}} - \frac{\lambda}{c_{1}}, \xi_{2}^{2} := \frac{\lambda}{c_{2}} -\frac{k^{2}\pi^{2}}{L^{2}},  (\xi'_{2})^{2} :=  \frac{k^{2}\pi^{2}}{L^{2}} - \frac{\lambda}{c_{2}},
\end{array}
\end{equation}
\'etant entendu que les valeurs $\xi_{0}^{2}, \xi_{1}^{2}, \xi'_{1}$ et $\xi'_{2}$ seront toujours des nombres r\'eels positifs quand nous les utiliserons. Notant $u_{i}:=u_{\vert \Omega_{i}}$, on v\'erifie que
\begin{equation*}
\label{suite(stratification3valeurs)equation:2}
\begin{array}{l}
u_{0}(x) = a_{0}\sin(\xi_{0}x)\\
%u_{1} (x) = a_{1}\sin(\xi_{1}x) + b_{1}\cos(\xi_{1}x)\mbox{ si } \xi_{1}^{2}>0,\\
%u_{1} (x) = a'_{1}\sinh(\xi'_{1}x) + b'_{1}\cosh(\xi'_{1}x)\mbox{ si } \xi_{1}^{2}<0,\\
u_{1} (x) = \left\lbrace\begin{array}{l}
a_{1}\sin(\xi_{1}x) + b_{1}\cos(\xi_{1}x)\mbox{ si } \xi_{1}^{2}>0,\\
 a'_{1}\sinh(\xi'_{1}x) + b'_{1}\cosh(\xi'_{1}x)\mbox{ si } \xi_{1}^{2}<0,
\end{array} \right.\\
u_{2}(x) = \left\lbrace\begin{array}{l}
a_{2}\sinh(\xi'_{2}(H-x)) \mbox{ si } \xi_{2}^{2}<0,\\
 a_{2}\sin(\xi_{2}(H-x)) \mbox{ si } \xi_{2}^{2}>0,
\end{array}\right.
\end{array}
\end{equation*}
et on obtient trois relations de dispersion suivant la position de la valeur propre $\lambda$ (cf. figure \ref{fig:modele2sauts}) :
\begin{eqnarray}
\label{suite(stratification3valeurs)equation:4}
\mathbf{Zone(0)}\;&-\frac{\tan(\xi_{0}h_{0})}{c_{0}\xi_{0}} &=\frac{\frac{\tanh(\xi'_{1}(h_{1}-h_{0}))}{c_{1}\xi'_{1}}+\frac{\tanh(\xi'_{2}(H-h_{1}))}{c_{2}\xi'_{2}}}{1+ \frac{c_{1}\xi'_{1}}{c_{2}\xi'_{2}}\tanh(\xi'_{1}(h_{1}-h_{0}))\tanh(\xi'_{2}(H-h_{1}))}\mbox{ si } \xi_{1}^{2}<0,\\
\label{suite(stratification3valeurs)equation:3}
\mathbf{Zone(I)}\;&\frac{\tanh(\xi'_{2}(H-h_{1}))}{c_{2}\xi'_{2}} &= \frac{\frac{\tan(\xi_{0}h_{0})}{c_{0}\xi_{0}}+\frac{\tan(\xi_{1}(h_{1}-h_{0}))}{c_{1}\xi_{1}}}{\frac{c_{1}\xi_{1}}{c_{0}\xi_{0}}\tan(\xi_{0}h_{0})\tan(\xi_{1}(h_{1}-h_{0}))-1}\mbox{ si } \xi_{1}^{2}>0 \mbox{ et } (\xi'_{2})^{2}>0,\\
\label{suite(stratification3valeurs)equation:3bis}
\mathbf{Zone(II)}\;&\frac{\tan(\xi_{2}(H-h_{1}))}{c_{2}\xi_{2}} &= \frac{\frac{\tan(\xi_{0}h_{0})}{c_{0}\xi_{0}}+\frac{\tan(\xi_{1}(h_{1}-h_{0}))}{c_{1}\xi_{1}}}{\frac{c_{1}\xi_{1}}{c_{0}\xi_{0}}\tan(\xi_{0}h_{0})\tan(\xi_{1}(h_{1}-h_{0}))-1}\mbox{ si } \xi_{1}^{2}>0 \mbox{ et } \xi_{2}^{2}>0.\nonumber
\end{eqnarray}
%Si on consid\`{e}re, \`{a} $k$ fix\'e, la position des valeurs propres $\lambda_{k,\ell}$ par rapport aux valeurs limites $c_{0}\frac{k^{2}\pi^{2}}{L^{2}}, c_{1}\frac{k^{2}\pi^{2}}{L^{2}}$ et $c_{2}\frac{k^{2}\pi^{2}}{L^{2}}$, 
On constate que
\begin{itemize}
\item lorsque $c_{0}\frac{k^{2}\pi^{2}}{L^{2}}<\lambda_{k,\ell}<c_{1}\frac{k^{2}\pi^{2}}{L^{2}}$, la restriction \`{a} $\Omega_{0}$ de la solution $u_{k,\ell}$  arrive sur l'interface $S_{0}$ avec un angle sup\'erieur \`{a} l'angle limite de la r\'eflexion totale;
\item lorsque $c_{1}\frac{k^{2}\pi^{2}}{L^{2}}<\lambda_{k,\ell}<c_{2}\frac{k^{2}\pi^{2}}{L^{2}}$, la restriction \`{a} $\Omega_{0}\cup\Omega_{1}$ traverse $S_{0}$ et arrive sur l'interface $S_{1}$ avec un angle sup\'erieur \`{a} l'angle limite de la r\'eflexion totale;
\item lorsque $c_{2}\frac{k^{2}\pi^{2}}{L^{2}}<\lambda_{k,\ell}$, la solution $u_{k,\ell}$ est, dans chaque r\'egion $\Omega_{0},\Omega_{1}$ et $\Omega_{2},$ une combinaison lin\'eaire de sinus et cosinus sans r\'eflexion totale sur chacune des interfaces $S_{0}$ et $S_{1}.$ Cette situation est trait\'ee dans l'annexe \ref{annexe-theoreme-conf-conc-fonctionnonguidee:1}.
\end{itemize}
%%%%%%%%%%%%%}%%%%%%%%%%%%%%%%%%%%%%%%%%%%%%%%%%%%%%%
\subsection{Zone (0) : cas de la relation de dispersion \eqref{suite(stratification3valeurs)equation:4}}
\label{soussection-3couches-valeurspropresbasses}
Les valeurs propres $\lambda$ satisfont donc $c_{0}\frac{k^{2}\pi^{2}}{L^{2}} <\lambda< c_{1}\frac{k^{2}\pi^{2}}{L^{2}}.$ Cette infinit\'e forme une suite \`{a} deux indices  $\lambda_{k,\ell}, k=1,\cdots, 1\leq \ell\leq \mathcal{L}_{k}$, l'entier $\mathcal{L}_{k}$ \'etant une fonction monotone croissante de $k.$ Il existe $m>0$ minorant $\xi_{0}$, sinon on aurait $h_{0}/c_{0} = 0$ puisque $(\xi_{0}\to 0) \Longrightarrow ( \xi'_{1}\to \infty, \xi'_{2}\to \infty).$
\begin{mapropo}
\label{proposition-2saus-zone0-solutionprpre:1}
Posant $a_{0} = 1$, on a, \`{a} une constante de normalisation pr\`{e}s,
\begin{eqnarray}
\label{suite(stratification3valeurs)equation:5}
u_{0}(x)\!\!\!& = &\!\!\!\sin(\xi_{0}x)\\
\label{suite(stratification3valeurs)equation:6}
u_{1}(x)\!\!\!& = &\!\!\!\cosh(\xi'_{1}(x-h_{0}))\left\lbrack \sin(\xi_{0}h_{0}) + \frac{c_{0}\xi_{0}}{c_{1}\xi'_{1}}\cos(\xi_{0}h_{0})\tanh(\xi'_{1}(x-h_{0}))\right\rbrack\\
\label{suite(stratification3valeurs)equation:7}
u_{2}(x)\!\!\!& = &\!\!\!\frac{\cosh(\xi'_{1}(h_{1}-h_{0}))}{\sinh(\xi'_{2}(H-h_{1}))}\!\!\left\lbrack \sin(\xi_{0}h_{0}) + \frac{c_{0}\xi_{0}}{c_{1}\xi'_{1}}\cos(\xi_{0}h_{0})\tanh(\xi'_{1}(h_{1}-h_{0}))\!\right\rbrack\!\!\sinh(\xi'_{2}(\!H-x)\!)
\end{eqnarray}
\end{mapropo}
\begin{monlem}
\label{lemme-2sauts-zone(0)-estimations:1}
Si $c_{0}\frac{k^{2}\pi^{2}}{L^{2}} <\lambda< c_{1}\frac{k^{2}\pi^{2}}{L^{2}}$, la fonction propre associ\'ee v\'erifie
\begin{equation*}
\label{equation-2sauts-zone(0)-estimations:1}
\begin{array}{l}
\vert u_{1}(x)\vert \left\lbrace
\begin{array}{l}
 \leq 4\vert \cos(\xi_{0}h_{0})\vert \frac{c_{0}\xi_{0}}{c_{1}\xi'_{1}}e^{-\xi'_{1}(x-h_{0})}, \\
 \\
 \geq \left\lbrace\begin{array}{l} \frac{1}{8}\vert \cos(\xi_{0}h_{0})\vert \frac{c_{0}\xi_{0}}{c_{2}\xi'_{2}}\tanh(\xi'_{2}(H-h_{1}))e^{-\xi'_{1}(x-h_{0})},\\
 \frac{\vert \cos(\xi_{0}h_{0})\vert}{4(1+\sqrt{\frac{c_{1}(c_{1}-c_{0})}{{c_{2}(c_{2}-c_{0})}}})}\frac{c_{0}\xi_{0}}{c_{2}\xi'_{2}}e^{-\xi'_{1}(x-h_{0})}(1 - e^{-2\xi'_{1}(h_{1}-x)})
 \end{array}\right. \end{array}\right.h_{0}<x<h_{1},
 %\\
 \end{array}
 \end{equation*}
% \\
\begin{equation*}
 \!\!\!\!\!\!\!\!\!\!\!\!\!\!\!\!\vert u_{2}(x)\vert\left\lbrace
 \begin{array}{l}
 \leq 2 \vert\cos(\xi_{0}h_{0})\vert\frac{c_{0}\xi_{0}}{c_{2}\xi'_{2}}e^{-\xi'_{1}(h_{1}-h_{0})}\frac{\sinh(\xi'_{2}(H-x))}{\cosh(\xi'_{2}(H-h_{1}))},\\
 \\
 \geq \frac{1}{2}\vert\cos(\xi_{0}h_{0})\vert\frac{c_{0}\xi_{0}}{c_{2}\xi'_{2}}e^{-\xi'_{1}(h_{1}-h_{0})}\frac{\sinh(\xi'_{2}(H-x))}{\cosh(\xi'_{2}(H-h_{1}))},
 \end{array}\right. h_{1}<x<H.
% \end{array}
\end{equation*}
\end{monlem}
\begin{proof}
Pour $u_{1}$ on met \eqref{suite(stratification3valeurs)equation:6} sous la forme
\begin{equation}
\label{equation-2sauts-zone(0)-estimations:3}
u_{1}(x) = \cos(\xi_{0}h_{0})\frac{c_{0}\xi_{0}}{c_{1}\xi'_{1}}\cosh(\xi'_{1}(x-h_{0})) \left\lbrack\tanh(\xi'_{1}(x-h_{0})) + c_{1}\xi'_{1}\frac{\tan(\xi_{0}h_{0})}{c_{0}\xi_{0}}\right\rbrack.
\end{equation}
Appelant $M_{1}$ l'expression \`{a} l'int\'erieur du crochet et utilisant la relation de dispersion \eqref{suite(stratification3valeurs)equation:4}, on a
\begin{eqnarray}
\label{equation-2sauts-zone(0)-estimations:4}
M_{1}(x) & =& \tanh(\xi'_{1}(x-h_{0})) -\frac{\tanh(\xi'_{1}(h_{1}-h_{0}))+\frac{c_{1}\xi'_{1}}{c_{2}\xi'_{2}}\tanh(\xi'_{2}(H-h_{1}))}{1+ \frac{c_{1}\xi'_{1}}{c_{2}\xi'_{2}}\tanh(\xi'_{1}(h_{1}-h_{0}))\tanh(\xi'_{2}(H-h_{1}))}\nonumber\\
& = & \frac{\tanh(\xi'_{1}(x-h_{0}))-\tanh(\xi'_{1}(h_{1}-h_{0}))}{1+ \frac{c_{1}\xi'_{1}}{c_{2}\xi'_{2}}\tanh(\xi'_{1}(h_{1}-h_{0}))\tanh(\xi'_{2}(H-h_{1}))}\left(1 + \frac{c_{1}\xi'_{1}}{c_{2}\xi'_{2}}\frac{\tanh(\xi'_{2}(H-h_{1}))}{ \tanh(\xi'_{1}(h_{1}-x))}\right).
\end{eqnarray}
De $
\tanh(\xi'_{1}(x-h_{0}))-\tanh(\xi'_{1}(h_{1}-h_{0})) = -2 e^{-2\xi'_{1}(x-h_{0})}\frac{1 - e^{-2\xi'_{1}(h_{1}-x)}}{(1 + e^{-2\xi'_{1}(x-h_{0})})(1 + e^{-2\xi'_{1}(h_{1}-h_{0})})}$
on d\'eduit $\vert \tanh(\xi'_{1}(x-h_{0}))-\tanh(\xi'_{1}(h_{1}-h_{0})) \vert \leq 2 e^{-2\xi'_{1}(x-h_{0})}(1-e^{-2\xi'_{1}(h_{1}-x)}).$ Ce dernier facteur va servir \`{a} annihiler l'influence de $\tanh(\xi'_{1}(h_{1}-x))$ au second d\'enominateur de l'expression de $M_{1}$ puisque $(1-e^{-2\xi'_{1}(h_{1}-x)})(\tanh(\xi'_{1}(h_{1}-x)))^{-1}= 1 +  e^{-2\xi'_{1}(h_{1}-x)}.$ Pour finir la majoration, tenant compte de $(c_{1}\xi'_{1})/(c_{2}\xi'_{2}) <1,$ on obtient $\vert M_{1}(x)\vert \leq 4 e^{-2\xi'_{1}(x-h_{0})}.$ Ne prenant en compte que $\frac{c_{1}\xi'_{1}}{c_{2}\xi'_{2}}\frac{\tanh(\xi'_{2}(H-h_{1}))}{ \tanh(\xi'_{1}(h_{1}-x))}$ dans la parenth\`{e}se de \eqref{equation-2sauts-zone(0)-estimations:4}, on a la premi\`{e}re minoration : $\vert M_{1}(x)\vert\geq \frac{1}{4}\frac{c_{1}\xi'_{1}}{c_{2}\xi'_{2}}\tanh(\xi'_{2}(H-h_{1}))e^{-2\xi'_{1}(x-h_{0})}$ puisque $1-e^{-2y}\geq \tanh y.$ Avec le premier terme de cette parenth\`{e}se on a $\vert M_{1}(x) \vert > \frac{1}{m}e^{-2\xi'_{1}(x-h_{0})}(1 - e^{-2\xi'_{1}(h_{1}-x)})$ o\`{u} $m^{-1} = 2(1+\sqrt{\frac{c_{1}(c_{1}-c_{0})}{{c_{2}(c_{2}-c_{0})}}}).$
Pour $u_{2},$ utilisant encore \eqref{suite(stratification3valeurs)equation:4}, on met \eqref{suite(stratification3valeurs)equation:7} sous la forme
\begin{equation*}
\label{equation-2sauts-zone(0)-estimations:6}
u_{2}(x) = \cos(\xi_{0}h_{0})\frac{c_{0}\xi_{0}}{c_{1}\xi'_{1}}\frac{\cosh(\xi'_{1}(h_{1}-h_{0}))}{\sinh(\xi'_{2}(H-h_{1}))}M_{2}\sinh(\xi'_{2}(H-x))
\end{equation*}
avec $M_{2}:= -\frac{c_{1}\xi'_{1}}{c_{2}\xi'_{2}}\frac{\tanh(\xi'_{2}(H-h_{1}))}{\cosh^{2}(\xi'_{1}(h_{1}-h_{0}))}\frac{1}{1 + \frac{c_{1}\xi'_{1}}{c_{2}\xi'_{2}}\tanh(\xi'_{2}(H-h_{1}))\tanh(\xi'_{1}(h_{1}-h_{0}))}.$ Comme $\frac{c_{1}\xi'_{1}}{c_{2}\xi'_{2}}<1,$ il vient que 
\begin{equation*}
\label{equation-2sauts-zone(0)-estimations:7}
\frac{1}{2}\frac{c_{1}\xi'_{1}}{c_{2}\xi'_{2}} \frac{\tanh(\xi'_{2}(H-h_{1}))}{\cosh(\xi'_{1}(h_{1}-h_{0}))}e^{-\xi'_{1}(h_{1}-h_{0})}\leq \vert M_{2}\vert\leq  2\frac{c_{1}\xi'_{1}}{c_{2}\xi'_{2}}\frac{\tanh(\xi'_{2}(H-h_{1}))}{\cosh(\xi'_{1}(h_{1}-h_{0}))}e^{-\xi'_{1}(h_{1}-h_{0})}
\end{equation*}
d'o\`{u}
\begin{eqnarray*}
\label{equation-2sauts-zone(0)-estimations:8}
\vert u_{2}(x) \vert&\leq &  2\vert\cos(\xi_{0}h_{0})\vert\frac{c_{0}\xi_{0}}{c_{2}\xi'_{2}}e^{-\xi'_{1}(h_{1}-h_{0})}\frac{\sinh(\xi'_{2}(H-x))}{\cosh(\xi'_{2}(H-h_{1}))}\\
\vert u_{2}(x) \vert&\geq & \frac{1}{2}\vert\cos(\xi_{0}h_{0})\vert\frac{c_{0}\xi_{0}}{c_{2}\xi'_{2}}e^{-\xi'_{1}(h_{1}-h_{0})}\frac{\sinh(\xi'_{2}(H-x))}{\cosh(\xi'_{2}(H-h_{1}))}
\end{eqnarray*}
\end{proof}
Maintenant nous pouvons mesurer la masse des fonctions propres d'une suite $(v_{k_{n},\ell_{n}}),n\to\infty$, dans un ouvert $\omega$ et nous avons vu qu'il suffit de consid\'erer un ouvert de la forme $\omega:=(0,L)\times (a,b)\subset \Omega,$ ce qui revient \`{a} n\'egliger la largeur de l'ouvert $\omega.$

\begin{montheo}
\label{theoreme-2sauts-zone0-repartitionenergie:1}
Soit une suite $(v_{k_{n},\ell_{n}})_{n}$ de fonctions propres telle que $c_{0}\frac{k_{n}^{2}\pi^{2}}{L^{2}} <\lambda_{k_{n},\ell_{n}}< c_{1}\frac{k_{n}^{2}\pi^{2}}{L^{2}}$. On a forc\'ement $\xi'_{2}\to \infty$ et la suite v\'erifie \`{a} un coefficient de normalisation pr\`{e}s\\
$\bullet$\; Sur $\Omega_{0}$
\begin{eqnarray}
\label{suite(stratification3valeurs)equation:12}
\int_{\Omega_{0}} v_{k_{n},\ell_{n}}^{2}(x){\rm d}x \!\!\!\!& = \!\!\!\!& \frac{h_{0}}{2}\left(1 -\frac{\sin(2\xi_{0}h_{0})}{2\xi_{0}h_{0}}\right)
%\to  \frac{h_{0}}{2}\
\end{eqnarray}
%Noter que $\xi_{0}$ a une borne inf\'erieure strictement positive.\\
$\bullet$\; Sur $\Omega_{1}$
%\begin{equation}
%\label{equation-2sauts-zone(0)-estimations:8}
%(\xi'_{1}\to\infty)\Longrightarrow\int_{h_{0}}^{h_{1}}u_{k_{n},\ell_{n}}^{2}(x){\rm d}x \left\lbrace\begin{array}{l}\leq C_{2} \cos^{2}(\xi_{0}h_{0})\left(\frac{c_{0}\xi_{0}}{c_{1}\xi'_{1}}\right)^{2}\frac{1}{\xi'_{1}},\\
%\geq C_{1} \cos^{2}(\xi_{0}h_{0})\left(\frac{c_{0}\xi_{0}}{c_{2}\xi'_{2}}\right)^{2}\frac{1}{\xi'_{1}}.
%\end{array}
%\right.
%\end{equation}
%
\begin{eqnarray}
\label{equation-2sauts-zone(0)-estimations:8}
(\xi'_{1}\to\infty)&\Longrightarrow&\int_{\Omega_{1}}v_{k_{n},\ell_{n}}^{2}(x){\rm d}x \left\lbrace\begin{array}{l}\leq C_{2} \cos^{2}(\xi_{0}h_{0})\left(\frac{c_{0}\xi_{0}}{c_{1}\xi'_{1}}\right)^{2}\frac{1}{\xi'_{1}},\\
\geq C_{1} \cos^{2}(\xi_{0}h_{0})\left(\frac{c_{0}\xi_{0}}{c_{2}\xi'_{2}}\right)^{2}\frac{1}{\xi'_{1}}.
\end{array}
\right.\\
\label{equation-2sauts-zone(0)-estimations:8'}
(\sup \xi'_{1} < \infty)&\Longrightarrow&\left(\exists C_{1}, C_{2}>0 \mbox{ telles que } C_{1}<\int_{\Omega_{1}}v_{k_{n},\ell_{n}}^{2}(x){\rm d}x <C_{2}\right)
\end{eqnarray}
$\bullet$\; Sur $\Omega_{2}$
\begin{equation}
\label{equation-2sauts-zone(0)-estimations:9}
\int_{\Omega_{2}}v_{k_{n},\ell_{n}}^{2}(x){\rm d}x = \cos^{2}(\xi_{0}h_{0})\left(\frac{c_{0}\xi_{0}}{c_{2}\xi'_{2}}\right)^{2}e^{-2\xi'_{1}(h_{1}-h_{0})}\,O\!\left(\frac{1}{\xi'_{2}}\right)
\end{equation}
Dans le cas particulier o\`{u}, de plus, $\sup \xi'_{1}<\infty$, on a
\begin{equation}
\label{equation-2sauts-zone(0)-estimations:9'}
(\sup \xi'_{1} < \infty)\Longrightarrow \int_{\Omega_{2}}v_{k_{n},\ell_{n}}^{2}(x){\rm d}x = O\left( \frac{1}{(\xi'_{2})^{3}}\right)
\end{equation}
\end{montheo}
\begin{maremarque} Il y a une forte analogie avec la concentration d\'ecrite dans la section \ref{section-modeleunsaut:1}
%qui a lieu lorsque le coefficient de diffusion $c$  a un saut unique 
:
\begin{enumerate}
\item Les r\'esultats pr\'ec\'edents signifient que pour des valeurs propres de la zone {\bf{(0)}}, se rapprochant de la ligne $\lambda = c_{1}\frac{k^{2}\pi^{2}}{L^{2}},$ la norme des $v_{k_{n},\ell_{n}}$ se concentre de plus en plus dans $\Omega_{0}\cup\Omega_{1}.$
\item On peut prouver que $\vert\cos(\xi_{0}h_{0})\vert\to 1$ si $1\leq\ell_{n}\leq N$ o\`{u} $N$ est un entier fix\'e. La repr\'esentation de \eqref{suite(stratification3valeurs)equation:4} \'etant similaire \`{a} celle de la figure \ref{fig:repvp} on peut raisonner similairement.
\item Si $\frac{c_{0}\xi_{0}}{c_{1}\xi'_{1}}$ est born\'e dans \eqref{equation-2sauts-zone(0)-estimations:8}, il vient que la couche  $\Omega_{1}$ contribue de fa\c{c}on n\'egligeable dans $\Vert u_{k_{n},\ell_{n}}\Vert_{L^{2}(\Omega)}$ par rapport \`{a} la couche $\Omega_{0}$ lorsque $n\to\infty.$ C'est le cas si $c_{0}\frac{k_{n}^{2}\pi^{2}}{L^{2}} <\lambda_{k_{n},\ell_{n}}< (c_{1}-\varepsilon)\frac{k_{n}^{2}\pi^{2}}{L^{2}},\varepsilon>0.$ En effet,
 les rapports $\frac{c_{0}\xi_{0}}{c_{2}\xi'_{2}}$ et $\frac{c_{0}\xi_{0}}{c_{1}\xi'_{1}}$ sont alors comparables puisque cette condition implique $\sqrt{\frac{c_{1}\varepsilon}{c_{2}(c_{2}-c_{1} + \varepsilon)}}\leq\frac{c_{1}\xi'_{1}}{c_{2}\xi'_{2}} \leq \sqrt{\frac{c_{1}(c_{1}-c_{0})}{c_{2}(c_{2}-c_{0})}}$. Ce sera aussi le cas si $\ell_{n}<\mathcal{L}_{k}-1.$
Il en est de m\^{e}me pour la couche $\Omega_{2}$ par rapport \`{a} la couche $\Omega_{1}.$
\item Les rapports de d\'ecroissance pr\'ec\'edents seraient beaucoup plus importants, i.e. de type exponentiel, si on consid\'erait des ouverts $\omega$ telles que $\bar{\omega}\subset\Omega_{1}$ ou $\bar{\omega}\subset\Omega_{2}$.
\end{enumerate}
\end{maremarque}
\begin{proof} Examinons seulement les affirmations ne d\'ecoulant pas du lemme \ref{lemme-2sauts-zone(0)-estimations:1}. Si $\sup \xi'_{1} <M<\infty$, on a $\xi_{0}, \xi'_{2}\to \infty$ et $\frac{\tanh(M(h_{1}-h_{0}))}{c_{1} M}\leq \frac{\tanh(\xi'_{1}(h_{1}-h_{0}))}{c_{1} \xi'_{1}}\sim -\frac{\tan(\xi_{0}h_{0})}{c_{0}\xi_{0}}\leq \infty.$ Comme $\xi_{0} \to \infty,$ il faut que $\tan(\xi_{0}h_{0})\to\infty$, i.e. $\cos (\xi_{0}h_{0})\to 0$ d'o\`{u} $\vert\sin(\xi_{0}h_{0})\vert\to 1$ et il existe $C>0$ telle que $C< \vert\xi_{0}\cos(\xi_{0}h_{0})\vert <\frac{1}{C}$ que l'on utilisera dans une nouvelle formulation de \eqref{equation-2sauts-zone(0)-estimations:3} :
\begin{equation*}
\vert u_{1}(x) \vert= c_{0}\xi_{0}\vert \cos(\xi_{0}h_{0})\vert\cosh(\xi_{1}(x-h_{0}))\left\lbrack \frac{x-h_{0}}{c_{1}}\frac{\tanh(\xi'_{1}(x-h_{0}))}{\xi'_{1}(x-h_{0})} + \frac{h_{0}}{c_{0}}\frac{\tanh(\xi_{0}h_{0})}{\xi_{0}h_{0}}\right\rbrack.
\end{equation*}
Ainsi $c_{0}\xi_{0}\vert \cos(\xi_{0}h_{0})\vert \frac{x-h_{0}}{c_{1}}\frac{\tanh(\xi'_{1}(h_{1}-h_{0}))}{\xi'_{1}(h_{1}-h_{0})}\vert \leq \vert u_{1}(x)\vert \leq c_{0}\xi_{0}\vert \cos(\xi_{0}h_{0})\vert\cosh(\xi_{1}(h_{1}-h_{0}))\lbrack  \frac{h_{1}-h_{0}}{c_{1}} +  \frac{h_{0}}{c_{0}}\rbrack$, ce qui implique que la norme $L^{2}$ sur $\Omega_{1}$ est encadr\'ee par deux constantes non nulles. Pour \eqref{equation-2sauts-zone(0)-estimations:9'}, il suffit de noter que 
\begin{equation*}
\label{equation-2sauts-zone(0)-estimations:9''}
\int_{h_{1}}^{H}\sinh^{2}(\xi'_{2}(H-x)){\rm d}x = \frac{\sinh(2\xi'_{2}(H-h_{1}))}{4\xi'_{2}}\left\lbrack 1 -\frac{2\xi'_{2}(H-h_{1})}{\sinh(2\xi'_{2}(H-h_{1}))}\right\rbrack 
\end{equation*}
\end{proof}
%\noindent
%{\bf\color{red}$\clubsuit\clubsuit$ Faut-il r\'e\'ecrire ces r\'esultats en fonction des valeurs propres ? Je n'ai pas regard\'e ce que cela donnerait.}
%%%%%%%%%%%%%%%%%%%%%%%%%%%%%%%%%%%%%%%%%%%%%%%%%%%%
%%%%%%%%%%%%%%%%%%%%%%%%%%%%%%%%%%%%%%%%%%%%%%%%%%%%
\subsection{Zone (I) : cas de la relation de dispersion \eqref{suite(stratification3valeurs)equation:3}}
\label{soussection-2sauts-zone(I):1}
Puisque $ \xi_{1}^{2}>0$ on a
\begin{equation*}
\label{equation-2sauts-zone(I):1}
\begin{array}{l}
u_{0}(x) = a_{0}\sin(\xi_{0}x)\\
u_{1} (x) = a_{1}\sin(\xi_{1}x) + b_{1}\cos(\xi_{1}x),\\
u_{2}(x) = a_{2}\sinh(\xi'_{2}(H-x))
\end{array}
\end{equation*}
%{\bf Notation : } Si $(f_{n})$ et $(g_{n})$ sont deux suites de fonctions, la notation $f_{n}\approxeq g_{n}$ signifiera que ces quantit\'es sont comparables, i.e. qu'il existe des nombres r\'eels positifs $M_{1},M_{2}>0$ tels $M_{1}\leq \frac{f_{n}}{g_{n}}\leq M_{2}$ d\`{e}s que $n$ est assez grand.\\
Le th\'eor\`{e}me \ref{annexe-theoreme-2sauts-zone(I):1} n'est qu'une r\'eponse partielle \`{a} nos questions sur la concentration des fonctions propres. Avec une r\'esolution compl\`{e}te de la relation de dispersion \eqref{suite(stratification3valeurs)equation:3} nous pourrions utiliser toutes les propri\'et\'es des valeurs propres et agir comme dans la section \ref{soussection-3couches-valeurspropresbasses}.
\begin{montheo} 
\label{annexe-theoreme-2sauts-zone(I):1}
Soit une suite $(v_{k_{n},\ell_{n}})_{n}$ de fonctions propres telle que $c_{1}\frac{k_{n}^{2}\pi^{2}}{L^{2}} <\lambda_{k_{n},\ell_{n}}< c_{2}\frac{k_{n}^{2}\pi^{2}}{L^{2}}$. La suite v\'erifie \`{a} un coefficient de normalisation pr\`{e}s\\
%La suite $(\lambda_{k_{n},\ell_{n}})_{n}$ appartient \`{a} la zone (I) , i.e. $c_{1}\frac{k^{2}\pi^{2}}{L^{2}} < \lambda_{k_{n},\ell_{n}} < c_{2}\frac{k^{2}\pi^{2}}{L^{2}},$ et $(v_{k_{n},\ell_{n}})_{n}$ est une suite de fonctions propres associ\'ees.
\begin{enumerate}
\item Si  $(c_{1}+\varepsilon)\frac{k_{n}^{2}\pi^{2}}{L^{2}} < \lambda_{k_{n},\ell_{n}} < (c_{2}-\varepsilon)\frac{k_{n}^{2}\pi^{2}}{L^{2}}$ avec $0<\varepsilon<\frac{c_{2}-c_{1}}{2}$, les fonctions propres $(v_{k_{n},\ell_{n}})_{n}$ concentrent alors leurs normes $L^{2}$ dans $\Omega_{0}\cup\Omega_{1}$ : 
\begin{eqnarray}
\label{equation-2sauts-zone(I):2}
h_{1}\leq a<b<H & \Longrightarrow & \int_{(0,L)\times(a,b)}\vert v_{k_{n},\ell_{n}}(x)\vert^{2}{\rm d}x\approxeq\frac{e^{-2\xi'_{2}(a-h_{1})}}{\xi'_{2}}\int_{\Omega_{0}}\vert v_{k_{n},\ell_{n}}(x)\vert^{2}{\rm d}x\\
\label{equation-2sauts-zone(I):3}
\int_{\Omega_{0}}\vert v_{k_{n},\ell_{n}}(x)\vert^{2}{\rm d}x &\approxeq & \int_{\Omega_{1}}\vert v_{k_{n},\ell_{n}}(x)\vert^{2}{\rm d}x 
\end{eqnarray}
\item Soit $(\xi'_{2})_{n}$ la suite des quantit\'es  correspondant \`{a} la suite $(\lambda_{k_{n},\ell_{n}})_{n},$ alors 
\begin{equation}
\label{equation-2sauts-zone(I):4}
(\xi'_{2}\to 0)\Longrightarrow\int_{\Omega_{0}}\vert v_{k_{n},\ell_{n}}(x)\vert^{2}{\rm d}x \approxeq  \int_{\Omega_{1}}\vert v_{k_{n},\ell_{n}}(x)\vert^{2}{\rm d}x \approxeq \int_{\Omega_{2}}\vert v_{k_{n},\ell_{n}}(x)\vert^{2}{\rm d}x.
\end{equation}
Autrement dit, la norme $L^{2}$ des fonctions propres $(v_{k_{n},\ell_{n}})_{n}$ est r\'epartie de mani\`{e}re comparable sur les $\Omega_{i}, i=0,1,2.$
\item  Soit $(\xi_{1})_{n}$ la suite des quantit\'es  correspondant \`{a} la suite $(\lambda_{k_{n},\ell_{n}})_{n},$ alors 
\begin{eqnarray}
\label{equation-2sauts-zone(I):5}
(\xi_{1}\to 0)&\Longrightarrow& \left( h_{1}\leq a<b\leq H \Longrightarrow \int_{(0,L)\times(a,b)}\vert v_{k_{n},\ell_{n}}(x)\vert^{2}{\rm d}x \approxeq  \frac{e^{-2\xi'_{2}(a-h_{1})}}{\xi'_{2}}\right),\\
(\xi_{1}\to 0)&\Longrightarrow&\left(\frac{h_{0}}{2}\sim\int_{\Omega_{0}}\vert v_{k_{n},\ell_{n}}(x)\vert^{2}{\rm d}x  \gg  \int_{\Omega_{2}}\vert v_{k_{n},\ell_{n}}(x)\vert^{2}{\rm d}x \approxeq \frac{1}{\xi'_{2}}\right),
\end{eqnarray}
et la norme $L^{2}$ des fonctions propres $(v_{k_{n},\ell_{n}})_{n}$ se concentre dans $\Omega_{1}.$
\end{enumerate}
\begin{proof} Utiliser les trois cas particuliers de l'annexe \ref{annexe-detailcalculs-zone(1):1}.
\end{proof} 
\end{montheo}
Toutes les possibilit\'es n'ont pas \'et\'e examin\'ees. Par exemple, il se pourrait que la suite des valeurs propres v\'erifie $\alpha_{1}>\xi_{1}>\alpha_{2}>0$ sans que l'on soit dans la situation 1 du th\'eor\`{e}me  \ref{annexe-theoreme-2sauts-zone(I):1}.
%%%%%%%%%%%%%%%%%%%%%%%%%%%%%%%%%%%%%%%%%%%%%%%%%%%%
%%%%%%%%%%%%%%%%%%%%%%%%%%%%%%%%%%%%%%%%%%%%%%%%%%%%
%%%%%%%%%%%%%%%%%%%%%%%%%%%%%%%%%%%%%%%%%%%%%%%%%%%%
%%%%%%%%%%%%%%%%%%%%%%%%%%%%%%%%%%%%%%%%%%%%%%%%%%%%
\vskip.3cm
\noindent
{\bf Conclusion} Il s'agit d'une amorce et d'autres cas ponctuels ont \'et\'e obtenus mais beaucoup reste \`{a} faire dans cette voie, en particulier la g\'en\'eralisation \`{a} un coefficient de diffusion seulement continu. Cependant, les r\'esultats obtenus disent d\'ej\`{a} ce qu'il est raisonnable d'esp\'erer. Notons que les \'etudes tr\`{e}s pr\'ecises du cas $N=1$ (Section \ref{section-modeleunsaut:1}) et de la Section \ref{soussection-3couches-valeurspropresbasses} r\'esultent d'une utilisation presque explicite de la relation de dispersion, ce ne fut pas possible dans la Section \ref{soussection-2sauts-zone(I):1}. D'autre part, nous aurions pu aussi consid\'erer l'\'energie $\vert u_{k,\ell}\vert^{2} + \vert \nabla u_{k,\ell}\vert^{2}$ au lieu de $\vert u_{k,\ell}\vert^{2}$, un rapide coup d'oeil montre que dans la plupart des cas la valeur du rapport $R_{\omega}(v)$ dans \eqref{notation-conc-conf-generalites:1} change mais pas sa nature.
\vskip.5cm
%%%%%%%%%%%%%%%%%%%%%%%%%%%%%%%%%%%%%%%%%%%%%%%%%%%%
%%%%%%%%%%%%%%%%%%%%%%%%%%%%%%%%%%%%%%%%%%%%%%%%%%%%
\centerline{\bf\Large Annexes}
\appendix
\renewcommand{\theequation}{\thesection.\arabic{equation}}
%%%%%%%%%%%%%%%%%%%%%%%%%%%%%%%%%%%%%%%%%%%%%%%%%%%%
%%%%%%%%%%%%%%%%%%%%%%%%%%%%%%%%%%%%%%%%%%%%%%%%%%%%

%%%%%%%%%%%%%%%%%%%%%%%%%%%%%%%%%%%%%%%%%%%%%%%%%%%%
%%%%%%%%%%%%%%%%%%%%%%%%%%%%%%%%%%%%%%%%%%%%%%%%%%%%
\section{Preuve du th\'eor\`{e}me \ref{theoreme-conf-conc-nvelleappr:1}}
\setcounter{equation}{0}
\label{annexe-theoreme-conf-conc-nvelleappr:1}
%%%%%%%%%%%%%%%%%%%%%%%%%%%%%%%%%%%%%%%%%%%%%%%%%%%%
%%%%%%%%%%%%%%%%%%%%%%%%%%%%%%%%%%%%%%%%%%%%%%%%%%%%
 La preuve sera faite en plusieurs \'etapes.\\
 {\bf 1/ Mise en place des \'equations}\\
 %\begin{enumerate}
Soit une valeur propre $\lambda = \lambda_{k,\ell}$ que l'on fixe pour l'instant et la fonction propre $v_{k,\ell}$ associ\'ee (normalis\'ee avec un coefficient positif). Apr\`{e}s la transformation de Fourier en $x_{1}$, on travaille avec la fonction $u_{k,\ell}$ ne d\'ependant plus que de $x_{2}$ et qui v\'erifie, posant $\mu_{k}= \frac{k\pi}{L},$
\begin{equation*}
\label{equation-conf-conc-nvelleappr:2}
\left\lbrace 
\begin{array}{l}
u_{k,\ell}^{''}- \mu_{k}^{2}u_{k,\ell} = -\frac{\lambda}{c}u_{k,\ell}, (0,H)\setminus\{h_{0}\},\\
u_{k,\ell}(0) = u_{k,\ell}(H) = 0,\\
u_{k,\ell}(h_{0}^{+}) = u_{k,\ell}(h_{0}^{-}), u'_{k,\ell}(h_{0}^{+}) = u'_{k,\ell}(h_{0}^{-}) .
\end{array}
 \right.
 \end{equation*}
 %On suppose que $\lambda<c_{1}\mu_{k}^{2}$
%\end{enumerate}
On supprime les indices $(k,\ell)$ dans la suite et, comme $c = c_{1}$ sur $(h_{0}, H),$ on peut \'ecrire
\begin{equation*}
\left\lbrace \begin{array}{l}
u'' -(\mu^{2}-\frac{\lambda}{c_{1}})u = 0, x_{2}\in (h_{0},H),\\
u'' - (\mu^{2}-\frac{\lambda}{c_{1}})u =\lambda(\frac{1}{c_{1}} - \frac{1}{c})u, x_{2}\in (0, h_{0}),\\
u(0) = u(H) =0,\\
u(h_{0}^{+}) = u(h_{0}^{-}), u'(h_{0}^{+}) = u'(h_{0}^{-}),
\end{array}
 \right.
 \end{equation*}
 puis, notant $
% \begin{equation}
 \label{notation-conf-conc-nvelleappr:1}
f(x) =\left\lbrace 
\begin{array}{l}
 0, x_{2}\in (h_{0},H),\\
-\lambda(\frac{1}{c_{1}} - \frac{1}{c})u, x_{2}\in\rbrack 0,h_{0}\lbrack
\end{array}
 \right.
%\end{equation}
$,
 on doit r\'esoudre
 \begin{equation}
 \label{notation-conf-conc-nvelleappr:3}
\left\lbrace \begin{array}{l}
u'' -\xi^{'2}_{1}u = -f,\\
u(0) = u(H) =0,\\
u(h_{0}^{+}) = u(h_{0}^{-}), u'(h_{0}^{+}) = u'(h_{0}^{-}) ,
\end{array}
 \right.
 \end{equation}
 avec $\xi'_{1}:= \sqrt{\mu^{2}-\frac{\lambda}{c_{1}}}>0$. Le calcul de $u(h_{0})$ sera une fonction de $u_{\vert (0,h_{0})}$ puisque le second membre de l'\'equation sur l'intervalle $(h_{0},H)$ est nul.
% On \'ecrit
%  \begin{equation}
%  \label{equation-conf-conc-nvelleappr:4}
%\left\lbrace \begin{array}{l}
%u'' -\xi^{'2}u = 0, x_{2}\in (h_{0},H),\\
%u(H) =0,\\
%%u(h_{0}) = L (u^{-}).
%\end{array}
% \right.
% \end{equation}
%\vskip.2cm\noindent
L'id\'ee est  d'estimer $ \Vert u\Vert_{L^{2}(h_{0},H)}$ et $\Vert u\Vert_{L^{2}(0,h_{0})}$ en fonction de $\vert u(h_{0})\vert$ d'o\`{u} l'estimation de
%L'introduction d'une bonne fonction poids sur $(h_{0},H)$ (la fonction $\beta$ aurait son maximum sur $\omega= (\bar{h}, H)\subset(h_{0}, H)$) permettrait d'estimer sur $(\bar{h}, H)$
 $\frac{\Vert u\Vert_{L^{2}(h_{0},H)}}{\Vert u\Vert_{L^{2}(0,H)}}$.\\
 % en fonction de la distance \`{a} la courbe $\lambda = c_{1}\frac{k^{2}\pi^{2}}{L^{2}}$ qui correspond au changement de comportement du coefficient de discontinuit\'e.\\
% \vskip.2cm\noindent
 {\bf 2/ Un lemme pr\'eparatoire}
\begin{monlem}
\label{lemme-conf-conc-nvelleappr:1}
La solution du syst\`{e}me 
%eqref{notation-conf-conc-nvelleappr:1}-
\eqref{notation-conf-conc-nvelleappr:3} v\'erifie
\begin{equation*}
\label{equation-conf-conc-nvelleappr:5}
u(h_{0})= \frac{\lambda}{\xi'_{1}}\frac{\tanh(\xi'_{1}(H-h_{0}))}{\tanh(\xi'_{1}(H-h_{0}))+ \tanh(\xi'_{1}h_{0})}\int_{0}^{h_{0}} \frac{\sinh(\xi'_{1}\sigma)}{\cosh(\xi'_{1}h_{0})}\frac{c_{1} - c(\sigma)}{c_{1}c(\sigma)}u(\sigma){\rm d}\sigma.
\end{equation*}
\end{monlem}
\begin{proof} Puisque l'\'equation diff\'erentielle est \`{a} coefficient constant la solution s'\'ecrit
\begin{equation*}
u(x) = u(h_{0})\cosh(\xi'_{1}(x-h_{0}) + \frac{u'(h_{0})}{\xi'_{1}}\sinh(\xi'_{1}(x-h_{0}) -\frac{1}{\xi'_{1}}\int_{h_{0}}^{x}\sinh(\xi'_{1}(x-\sigma))f(\sigma){\rm d}\sigma,
\end{equation*}
%o\`{u}
% %l'indice 0 correspond \`{a} $x\in (0,h_{0})$ et l'indice 1 \`{a} $x\in (h_{0},H)$ avec $A_{0} = u(h_{0}^{-})=A_{1} = 
%$u(h_{0})=A$ et $B\xi'_{1} = u'(h_{0})$ 
%d'o\`{u}
%\begin{equation*}
%u(x)\!\!\!= u(h_{0})\!\cosh(\xi'_{1}(x-h_{0}))  -\frac{1}{\xi'_{1}}\int_{h_{0}}^{x}\!\!\!\sinh(\xi'_{1}(x-\sigma))f(\sigma){\rm d}\sigma+ \frac{u'(h_{0}^{+})}{\xi'_{1}}\sinh(\xi'_{1}(x-h_{0})) \!\left\lbrace\begin{array}{l}
% \!\frac{c_{1}}{c_{0}^{-}}, 0<x<h_{0},\\
%\! 1, h_{0}<x<H.
%\end{array}
%\right.
%\end{equation*}
Les conditions aux extr\'emit\'es donnant
\begin{eqnarray*}
0 & =&  u(h_{0})\!\cosh(\xi'_{1}h_{0})  -\frac{u'(h_{0})}{\xi'_{1}}\sinh(\xi'_{1}h_{0}) - \frac{1}{\xi'_{1}}\int^{h_{0}}_{0}\sinh(\xi'_{1}\sigma)f(\sigma){\rm d}\sigma\\
0 & =&  u(h_{0}) \cosh(\xi'_{1} (H-h_{0}))  +\frac{u'(h_{0})}{\xi'_{1}} \sinh(\xi'_{1} (H-h_{0})) - \frac{1}{\xi'_{1}}\int_{h_{0}}^{H}\sinh(\xi'_{1} (H-\sigma))f(\sigma){\rm d}\sigma,
\end{eqnarray*}
on arrive au syst\`{e}me matriciel
%\begin{eqnarray*}
%u(h_{0}) -\frac{\tanh(\xi'_{1}h_{0})}{\xi'_{1}}u'(h_{0}) & = &\frac{1}{\xi'_{1}}\int^{h_{0}}_{0}\frac{\sinh(\xi'_{1}\sigma)}{\cosh(\xi'_{1} h_{0})}f(\sigma){\rm d}\sigma\\
%u(h_{0}) +\frac{\tanh(\xi'_{1}(H-h_{0}))}{\xi'_{1}}u'(h_{0}) &=& \frac{1}{\xi'_{1}}\int_{h_{0}}^{H} \frac{\sinh(\xi'_{1}(H-\sigma))}{\cosh(\xi'_{1} (H-h_{0}))}f(\sigma){\rm d}\sigma
%\end{eqnarray*}
%qui s'\'ecrit
\begin{equation*}
A\left(\begin{array}{c}
u(h_{0})\\
u'(h_{0})
\end{array}\right) = \left(\begin{array}{c}
\frac{1}{\xi'_{1}}\int^{h_{0}}_{0}\frac{\sinh(\xi'_{1}\sigma)}{\cosh(\xi'_{1} h_{0})}f(\sigma){\rm d}\sigma\\
\frac{1}{\xi'_{1}}\int_{h_{0}}^{H} \frac{\sinh(\xi'_{1}(H-\sigma))}{\cosh(\xi'_{1} (H-h_{0}))}f(\sigma){\rm d}\sigma
\end{array}\right)
\mbox{ avec } A = \left(\begin{array}{cc}
1 & -\frac{\tanh(\xi'_{1}h_{0})}{\xi'_{1}}\\
1 & \frac{\tanh(\xi'_{1}(H-h_{0}))}{\xi'_{1}}
\end{array}\right)
\end{equation*}
et $\det(A) = \frac{\tanh(\xi'_{1}(H-h_{0}))+ \tanh(\xi'_{1}h_{0})}{\xi'_{1}}.$ Finalement, on obtient
\begin{equation*}
\left(\begin{array}{c}
u(h_{0})\\
u'(h_{0})
\end{array}\right)  = \frac{1}{\det A}\left( 
\begin{array}{cc}
\frac{\tanh(\xi'_{1}(H-h_{0}))}{\xi'_{1}} & \frac{\tanh(\xi'_{1}h_{0})}{\xi'_{1}}\\
-1 & 1
\end{array}\right)\left(\begin{array}{c}
\frac{1}{\xi'_{1}}\int^{h_{0}}_{0}\frac{\sinh(\xi'_{1}\sigma)}{\cosh(\xi'_{1} h_{0})}f(\sigma){\rm d}\sigma\\
\frac{1}{\xi'_{1}}\int_{h_{0}}^{H} \frac{\sinh(\xi'_{1}(H-\sigma))}{\cosh(\xi'_{1} (H-h_{0}))}f(\sigma){\rm d}\sigma
\end{array}\right),
\end{equation*}
ce qui donne pour notre fonction $f$ particuli\`{e}re, puisque $f = 0$ sur l'intervalle $(h_{0}, H),$ 
\begin{eqnarray}
\label{equation-conf-conc-nvelleappr:5'}
\!\!\!u(h_{0}) & = &\!\!\frac{1}{\tanh(\xi'_{1}(H-h_{0}))+\tanh(\xi'_{1}h_{0})}\frac{\tanh(\xi'_{1}(H-h_{0}))}{\xi'_{1}}\int^{h_{0}}_{0}\frac{\sinh(\xi'_{1}\sigma)}{\cosh(\xi'_{1} h_{0})}f(\sigma){\rm d}\sigma\\
\!\!\!u'(h_{0}) & = &\!\!-\frac{1}{\tanh(\xi'_{1}(H-h_{0}))+ \tanh(\xi'_{1}h_{0})}\int^{h_{0}}_{0}\frac{\sinh(\xi'_{1}\sigma)}{\cosh(\xi'_{1} h_{0})}f(\sigma){\rm d}\sigma.\nonumber
\end{eqnarray}
%et, maintenant, appliquant ce qui pr\'ec\`{e}de \`{a} notre fonction $f$ particuli\`{e}re puisque $f = 0$ sur l'intervalle $(h_{0}, H)$ on a
%\begin{eqnarray*}
%\label{equation-conf-conc-nvelleappr:6}
%u(h_{0}) & = &\frac{\lambda}{\xi'_{1}}\frac{c_{0}^{-}\tanh(\xi'_{1}(H-h_{0}))}{c_{0}^{-}\tanh(\xi'_{1}(H-h_{0}))+ c_{1}\tanh(\xi'_{1}h_{0})}\int^{h_{0}}_{0}\frac{\sinh(\xi'_{1}\sigma)}{\cosh(\xi'_{1} h_{0})}\frac{c_{1}- c(\sigma)}{c(\sigma)c_{1}}u(\sigma){\rm d}\sigma\nonumber
%% & = & -\frac{c_{1}-c_{0}}{c_{1}}\frac{\lambda}{\xi'_{1}}\frac{\tanh(\xi'_{1}(H-h_{0}))}{c_{0}\tanh(\xi'_{1}(H-h_{0}))+ c_{1}\tanh(\xi'_{1}h_{0})}\int^{h_{0}}_{0}\frac{\sinh(\xi'_{1}(h_{0}-\sigma))}{\cosh(\xi'_{1} h_{0})}u(\sigma){\rm d}\sigma
%\end{eqnarray*}
\end{proof}
\noindent
{\bf 3/ Majorer $\vert u(h_{0})\vert$ par $\Vert u\Vert_{L^{2}(0,h_{0})}$}\\
On d\'eduit de $\frac{\sinh(\xi'_{1}\sigma)}{\cosh(\xi'_{1} h_{0})} = e^{-\xi'_{1}(h_{0}-\sigma)}\frac{1-e^{-2\xi'_{1}\sigma}}{1+ e^{-2\xi'_{1}h_{0}}}\leq e^{-\xi'_{1}(h_{0}-\sigma)}$ les estimations suivantes
\begin{eqnarray*}
\left\vert\int^{h_{0}}_{0}\frac{\sinh(\xi'_{1}\sigma)}{\cosh(\xi'_{1} h_{0})}\frac{c_{1}- c(\sigma)}{c(\sigma)c_{1}}u(\sigma){\rm d}\sigma\right\vert^{2} & \leq & \max(\frac{c-c_{1}}{c c_{1}})^{2}\int_{0}^{h_{0}}e^{-2\xi'_{1}(h_{0}-\sigma)}{\rm d}\sigma\int_{0}^{h_{0}}\vert u(\sigma)\vert^{2}{\rm d}\sigma\\
& \leq & \max(\frac{c-c_{1}}{c c_{1}})^{2}\frac{1}{2\xi'_{1}}\Vert u\Vert^{2}_{L^{2}(0,h_{0})},
\end{eqnarray*}
ce qui donne, en injectant cette in\'egalit\'e dans \eqref{equation-conf-conc-nvelleappr:5},
\begin{equation*}
\label{equation-conf-conc-nvelleappr:7}
\vert u(h_{0})\vert^{2}\leq \frac{1}{2} \max(\frac{c-c_{1}}{c c_{1}})^{2}\frac{\lambda^{2}}{(\xi'_{1})^{3}}\Vert u\Vert^{2}_{L^{2}(0,h_{0})},
\end{equation*}
que l'on peut aussi \'ecrire
\begin{equation*}
\label{equation-conf-conc-nvelleappr:10}
\Vert u\Vert^{2}_{L^{2}(0,h_{0})}
 %& \geq & \vert u(h_{0}\vert^{2} \left(\max\frac{c_{1}-c}{c_{1}c}\right)^{-2}2\frac{(\xi'_{1})^{3}}{\lambda^{2}}\nonumber\\
%& 
\geq  K \frac{(\xi'_{1})^{4}}{\lambda^{2}}\frac{1}{\xi'_{1}}\vert u(h_{0}\vert^{2} \mbox{ avec } K = 2(\max\frac{c_{1}-c}{c_{1}c})^{-2}.
\end{equation*}
{\bf 4/ \'Evaluer $\Vert u\Vert_{L^{2}(h_{0},H)}$ et $\Vert u\Vert_{L^{2}(a,b)}$ en fonction de $\vert u(h_{0})\vert$ pour $h_{0}\leq a<b\leq H.$}\\
Comme $c$ est une fonction constante sur l'intervalle $(h_{0}, H),$ on voit d'une part que 
\begin{eqnarray*}
\label{equation-conf-conc-nvelleappr:8}
\int_{h_{0}}^{H}\vert u\vert^{2}{\rm d}\sigma & = &\frac{H-h_{0}}{2}\left( \frac{u(h_{0})}{\sinh(\xi'_{1}(H-h_{0}))}\right)^{2}\left( \frac{\sinh(\xi'_{1}(H-h_{0}))}{\xi'_{1}(H-h_{0})}\cosh(\xi'_{1}(H-h_{0})) -1\right)\nonumber\\
& = &\frac{\vert u(h_{0})\vert^{2}}{4\xi'_{1}}\frac{\sinh(2\xi'_{1}(H-h_{0}))}{\sinh^{2}(\xi'_{1}(H-h_{0}))}\left( 1-\frac{2\xi'_{1}(H-h_{0})}{\sinh(2\xi'_{1}(H-h_{0}))}\right)\nonumber\\
& \sim &\frac{\vert u(h_{0})\vert^{2}}{2\xi'_{1}} \mbox{ si } \xi'_{1}\to\infty,
\end{eqnarray*}
et, d'autre part, que
\begin{eqnarray*}
\label{equation-conf-conc-nvelleappr:9}
\int_{a}^{b}\vert u\vert^{2}{\rm d}\sigma \!\!\!\!\!\!& = \!\!\!&\frac{b-a}{2}\left( \frac{u(h_{0})}{\sinh(\xi'_{1}(H-h_{0}))}\right)^{2}\left( \frac{\sinh(\xi'_{1}(b-a))}{\xi'_{1}(b-a)}\cosh(2\xi'_{1}(H-\frac{a+b}{2})) -1\right)\nonumber\\
\!\!\!\!\!\!& = &\!\!\!\!\!\frac{\vert u(h_{0})\vert^{2}}{\sinh^{2}(\xi'_{1}(H-h_{0}))}\frac{\sinh(\xi'_{1}(b-a))}{2\xi'_{1}}\cosh(2\xi'_{1}(H-\frac{a+b}{2}))\nonumber\\
&&\qquad\qquad\qquad\qquad\qquad\qquad\times\left( 1- \frac{\xi'_{1}(b-a)}{\sinh(\xi'_{1}(b-a))\cosh(2\xi'_{1}(H-\frac{a+b}{2}))} \right)\nonumber\\
&\sim & \frac{\vert u(h_{0})\vert^{2}}{\xi'_{1}}e^{-2\xi'_{1}(a-h_{0})}\mbox{ si } \xi'_{1}\to\infty.
\end{eqnarray*}
%{\bf 6/ \'Evaluer $\Vert u\Vert_{L^{2}(0,h_{0})}$ en fonction de $\vert u(h_{0})\vert$}\\  
%Il vient de \eqref{equation-conf-conc-nvelleappr:5} que
%\begin{eqnarray}
%\label{equation-conf-conc-nvelleappr:10}
%\Vert u\Vert^{2}_{L^{2}(0,h_{0})} & \geq & \vert u(h_{0}\vert^{2} \left(\max\frac{c_{1}-c}{c_{1}c}\right)^{-2}2\frac{(\xi'_{1})^{3}}{\lambda^{2}}\nonumber\\
%& \geq & K \frac{(\xi'_{1})^{4}}{\lambda^{2}}\frac{1}{\xi'_{1}}\vert u(h_{0}\vert^{2} \mbox{ avec } K = 2(\max\frac{c_{1}-c}{c_{1}c})^{-2}.
%\end{eqnarray}
{\bf Conclusion}\\
Si $\underline{c}\frac{k^{2}\pi^{2}}{L^{2}}<\lambda< (c_{1}-\varepsilon)\frac{k^{2}\pi^{2}}{L^{2}},$ alors on est certain que $
%\begin{equation}
%\label{equation-conf-conc-nvelleappr:11}
\frac{\varepsilon}{c_{1}(c_{1}-\varepsilon)}<\frac{(\xi'_{1})^{2}}{\lambda}<\frac{c_{1}-\underline{c}}{c_{1}\underline{c}}
%\end{equation}
$
et, de plus, $\xi'_{1}\to\infty,$ ce qui permet d'en d\'eduire le th\'eor\`{e}me \ref{theoreme-conf-conc-nvelleappr:1} puisque
\begin{equation*}
\frac{\Vert v_{k_{n},\ell_{n}}\Vert^{2}_{L^{2}(\omega)}}{\Vert v_{k_{n},\ell_{n}}\Vert^{2}_{L^{2}(\Omega)}} \leq \frac{\Vert v_{k_{n},\ell_{n}}\Vert^{2}_{L^{2}(\omega)}}{\Vert v_{k_{n},\ell_{n}}\Vert^{2}_{L^{2}(x_{2}<h_{0})} + \Vert v_{k_{n},\ell_{n}}\Vert^{2}_{L^{2}(x_{2}>h_{0})}}\leq \frac{\frac{ \vert v_{k_{n},\ell_{n}}(h_{0})\vert^{2}}{\xi'_{1}}e^{-2\xi'_{1}(a-h_{0})}}{ K \frac{(\xi'_{1})^{4}}{\lambda^{2}}\frac{1}{\xi'_{1}}\vert v_{k_{n},\ell_{n}}(h_{0})\vert^{2}+ \frac{ \vert v_{k_{n},\ell_{n}}(h_{0})\vert^{2}}{2\xi'_{1}}}(1 + o(1))
\end{equation*}
Si $\alpha >0$ et si on impose que $H-a>\alpha$ alors la constante $K_{\varepsilon, c}$ de l'\'enonc\'e est ind\'ependante de $a.$
%%%%%%%%%%%%%%%%%%%%%%%%%%%%%%%%%%%%%%%%%%%%%%%%%%%%
%%%%%%%%%%%%%%%%%%%%%%%%%%%%%%%%%%%%%%%%%%%%%%%%%%%%
%%%%%%%%%%%%%%%%%%%%%%%%%%%%%%%%%%%%%%%%%%%%%%%%%%%%
%%%%%%%%%%%%%%%%%%%%%%%%%%%%%%%%%%%%%%%%%%%%%%%%%%%%
\section{Preuve du th\'eor\`{e}me \ref{theoreme-conc-conf-fonctionnonguidee:1}}
\setcounter{equation}{0}
\label{annexe-theoreme-conf-conc-fonctionnonguidee:1}
%%%%%%%%%%%%%%%%%%%%%%%%%%%%%%%%%%%%%%%%%%%%%%%%%%%%
%%%%%%%%%%%%%%%%%%%%%%%%%%%%%%%%%%%%%%%%%%%%%%%%%%%%
\begin{maremarque}
\label{remarque-conc-conf-fonctionnonguidee:1}
\begin{enumerate}
\item Pour le mod\`{e}le \`{a} $N$ sauts, le r\'esultat appara\^{\i}t plausible \`{a} premi\`{e}re vue car les solutions sont des combinaisons de sinus et cosinus. La difficult\'e vient des coefficients de ces combinaisons qu'il convient de comparer entre eux.
\item Lorsque $c\in C^{1}(0,H)$ mais non \`{a} $C^{1}(\lbrack 0,H)\rbrack)$, les translations doivent respecter $0<\alpha<a<b<H-\alpha<H$ avec $\alpha$ fix\'e. Pour le mod\`{e}le \`{a} un saut, on peut poser $\varepsilon =0$ mais une \'etude fine de la relation de dispersion est n\'ecessaire.
\item Lorsque $c\in C^{1}(\lbrack 0,H\rbrack)$, la transformation de Liouville nous ram\`{e}ne \`{a} la r\'esolution de $F'' + QF = 0 $ avec $ Q\geq 0.$ Ensuite, la transformation de Pr\"{u}fer modifi\'ee donne un syst\`{e}me \'equivalent form\'e de deux \'equations diff\'erentielles du premier ordre.
\end{enumerate}
\end{maremarque}
\noindent
Lorsque $\varepsilon>0,$ la preuve  ne n\'ecessite pas la connaissance de la relation de dispersion, celle dont les racines sont les valeurs propres de l'op\'erateur $A.$  
Cherchant les fonctions propres $v$ de l'op\'erateur $A = -\nabla\cdot(c\nabla)$ ayant la forme $v(x) = \sin(\frac{k\pi}{L}x_{1})u(x_{2})$ et associ\'ees \`{a} la valeur propre $\lambda,$ nous introduisons dans $L^{2}(0,H)$ la famille d'op\'erateurs r\'eduits autoadjoints $(A_{k})$ d\'efinis formellement par $A_{k}u= -(cu')' +c\frac{k^{2}\pi^{2}}{L^{2}}u, u(0) = u(H) = 0,$ et nous sommes amen\'es \`{a} r\'esoudre l'\'equation
\begin{equation}
\label{equation-con-conf-fonctionnonguidee-C1:1}
(cu')' + (\lambda-c\frac{k^{2}\pi^{2}}{L^{2}})u = 0, u(0) = u(H) = 0.
\end{equation}
Pour $u$ dans le domaine de l'op\'erateur r\'eduit  $A_{k}$, on a $\int_{0}^{H}(cu'^{2} + (c\frac{k^{2}\pi^{2}}{L^{2}}-\lambda)u^{2}){\rm d}x = 0,$ ce qui implique que le spectre $\sigma(A_{k})$ v\'erifie $\sigma(A_{k})\subset ((\inf c) \frac{k^{2}\pi^{2}}{L^{2}}, +\infty).$ 
\subsection{Mod\`{e}le \`{a} $N$ sauts}
\label{soussection-con-conf-fonctionnonguidee-1saut:1}
%%%%%%%%%%%%%%%%%%%%%%%%%%%%%%%%%%%%%%%%%%%%%%%%%%%%
Les hypoth\`{e}ses $\mathbf{(H0)}$ et $\mathbf{(H1)}$ sont  satisfaites dans cette sous-section. Nous nous int\'eressons aux solutions $u$ de \eqref{equation-con-conf-fonctionnonguidee-C1:1} pour les valeurs propres $\lambda>c_{H}\frac{k^{2}\pi^{2}}{L^{2}}$ (ici $c_{H}=c_{N}$).
%\begin{equation}
%\label{equation-conc-conf-fonctionnonguidee-Nsauts:1}
%-(cu')' + (c\frac{k^{2}\pi^{2}}{L^{2}} -
%\lambda)u = 0,\quad u(0) = u(H) =0.
%\end{equation}
Si on pose $\xi_{i}:= (\frac{\lambda}{c_{i}}-\frac{k^{2}\pi^{2}}{L^{2}})^{\frac{1}{2}}$, la fonction propre $u$ associ\'ee \`{a} la valeur propre $\lambda = \lambda_{k,\ell}$ (on remplace $x_{2}$ par $x$) s'\'ecrit
\begin{eqnarray*}
\label{equation-conc-conf-fonctionnonguidee-Nsauts:2}
u_{0}(x) & = & a_{0}\sin(\xi_{0}x), 0<x<h_{0},\nonumber\\
u_{1}(x) & = & a_{1}\sin(\xi_{1}x) + b_{1}\cos(\xi_{1}x) , h_{0}<x<h_{1}, \nonumber\\
\vdots& = &\vdots\nonumber\\
u_{i}(x) & = & a_{i}\sin(\xi_{i}x) + b_{i}\cos(\xi_{i}x) , h_{i-1}<x<h_{i}, \nonumber\\
\vdots& = &\vdots\nonumber\\
u_{N-1}(x) & = & a_{N-1}\sin(\xi_{N-1}x) + b_{N-1}\cos(\xi_{N-1}x) , h_{N-2}<x<h_{N-1}, \nonumber\\
u_{N}(x) & = & a_{N}\sin(\xi_{N}x) + b_{N}\cos(\xi_{N}x) , h_{N-1}<x<h_{N}=H, \nonumber\\
& = & \alpha \sin(\xi_{N}(H-x)), h_{N-1}<x<h_{N}=H.
\end{eqnarray*}
Quelque soit la fonction propre consid\'er\'ee, nous choisissons le m\^{e}me $a_{0}$, par exemple 1, mais, par contre, les $(a_{i}, b_{i}), i = 1,\ldots,N,$ d\'ependent de $(k,\ell).$
\`{A} l'interface $S_{i}, i = 0,\cdots, N-1$, les conditions de transmission s'\'ecrivent
\begin{eqnarray*}
\label{equation-conc-conf-fonctionnonguidee-Nsauts:3}
a_{i}\sin(\xi_{i}h_{i}) + b_{i}\cos(\xi_{i}h_{i}) & = &a_{i+1}\sin(\xi_{i+1}h_{i}) + b_{i+1}\cos(\xi_{i+1}h_{i})\nonumber\\
c_{i}a_{i}\xi_{i}\cos(\xi_{i}h_{i}) - c_{i}b_{i}\xi_{i}\sin(\xi_{i}h_{i}) & = &c_{i+1} a_{i+1}\xi_{i+1}\cos(\xi_{i+1}h_{i}) - c_{i+1}b_{i+1}\xi_{i+1}\sin(\xi_{i+1}h_{i})
\end{eqnarray*}
d'o\`{u} l'\'ecriture matricielle $S_{i}\left(\begin{array}{c}a_{i}\\ b_{i}\end{array}\right) = T_{i}\left(\begin{array}{c}a_{i+1}\\ b_{i+1}\end{array}\right) $ avec $\det S_{i} = -c_{i}\xi_{i}, \det T_{i} = -c_{i+1}\xi_{i+1}$ et 
\begin{equation*}
\label{equation-conc-conf-fonctionnonguidee-Nsauts:4}
\!\!\!\!\begin{array}{ll}
T_{i}= \!\!\left(
\begin{array}{cc}
\sin(\xi_{i+1}h_{i}) & \cos(\xi_{i+1}h_{i})\\
c_{i+1}\xi_{i+1}\cos (\xi_{i+1}h_{i}) & -c_{i+1}\xi_{i+1}\sin(\xi_{i+1}h_{i})
\end{array}
\!\!\right)\!, &
\!\!\!T_{i}^{-1}= \!\!\left(
\begin{array}{cc}
\sin(\xi_{i+1}h_{i}) & \frac{\cos(\xi_{i+1}h_{i})}{c_{i+1}\xi_{i+1}}\\
\cos (\xi_{i+1}h_{i}) & -\frac{\sin(\xi_{i+1}h_{i})}{c_{i+1}\xi_{i+1}}
\end{array}
\!\!\right)\!, 
\\
\\
S_{i}= \!\!\left(
\begin{array}{cc}
\sin(\xi_{i}h_{i}) & \cos(\xi_{i}h_{i})\\
c_{i}\xi_{i}\cos (\xi_{i}h_{i}) & -c_{i}\xi_{i}\sin(\xi_{i}h_{i})
\end{array}
\!\!\!\right)\,,&
\!\!\!S_{i}^{-1}= \left(
\begin{array}{cc}
\sin(\xi_{i}h_{i}) & \frac{\cos(\xi_{i}h_{i})}{c_{i}\xi_{i}}\\
\cos (\xi_{i}h_{i}) & - \frac{\sin(\xi_{i}h_{i})}{c_{i}\xi_{i}}
\end{array}
\right). 
\end{array}
\end{equation*}
Pour $i=0,1,2,\ldots,N-1,$ on v\'erifie que
\!\!\!\begin{eqnarray*}
%%%%%%%%%%%%%%%%%%%%%%%%%%%%%%%%
\label{equation-conc-conf-fonctionnonguidee-Nsauts:6}
\left(\!\!\begin{array}{c}
a_{i+1}\\
-b_{i+1}
\end{array}\!\!\right) \!\!& = &\!\!\! \left(
\begin{array}{cc}
\sin(\xi_{i+1}h_{i}) &  \frac{c_{i}\xi_{i}}{c_{i+1}\xi_{i+1}}\cos(\xi_{i+1}h_{i})\\
-\cos(\xi_{i+1}h_{i}) &  \frac{c_{i}\xi_{i}}{c_{i+1}\xi_{i+1}}\sin(\xi_{i+1}h_{i})
\end{array}\right)
\!\left(
\begin{array}{cc}
\sin(\xi_{i}h_{i}) & \cos (\xi_{i}h_{i})\\
\cos(\xi_{i}h_{i}) & -\sin(\xi_{i}h_{i})
\end{array}\right)
\!\!\left(\begin{array}{c}
a_{i}\\
b_{i}
\end{array}\!\!\!\right)\\
%%%%%%%%%%%%%%%%%%%%%%%%%%%%%%%%
\label{equation-conc-conf-fonctionnonguidee-Nsauts:7}
\left(\begin{array}{c}
a_{i}\\
-b_{i}
\end{array}\right) \!\!& = &\!\!\!  \left(
\begin{array}{cc}
\sin(\xi_{i}h_{i}) &  \frac{c_{i+1}\xi_{i+1}}{c_{i}\xi_{i}}\cos(\xi_{i}h_{i})\\
-\cos(\xi_{i}h_{i}) &  \frac{c_{i+1}\xi_{i+1}}{c_{i}\xi_{i}}\sin(\xi_{i}h_{i})
\end{array}\right)
\!\!\left(
\begin{array}{cc}
\sin(\xi_{i+1}h_{i}) & \cos (\xi_{i+1}h_{i})\\
\cos(\xi_{i+1}h_{i}) & -\sin(\xi_{i+1}h_{i})
\end{array}\!\right)
\!\!\left(\!\!\begin{array}{c}
a_{i+1}\\
b_{i+1}
\end{array}\!\!\!\!\right)
\end{eqnarray*}
%%%%%%%%%%%%%%%%%%%%%%%%%%%%%%%%
\begin{monlem}
\label{lemme-conc-conf-fonctionnonguidee-Nsauts:1}
Les normes euclidiennes dans $\mathbb{R}^{2}$ des op\'erateurs lin\'eaires associ\'es aux matrices
\begin{equation*}
\label{equation-conc-conf-fonctionnonguidee-Nsauts:8}
\!\!\!\!\!B_{i+1,i}:=  \!\left(\!\!\!
\begin{array}{cc}
\sin(\xi_{i}h_{i}) &  \frac{c_{i+1}\xi_{i+1}}{c_{i}\xi_{i}}\cos(\xi_{i}h_{i})\\
-\cos(\xi_{i}h_{i}) &  \frac{c_{i+1}\xi_{i+1}}{c_{i}\xi_{i}}\sin(\xi_{i}h_{i})
\end{array}\!\!\right), 
B_{i,i+1}:=\! \left(\!\!\!
\begin{array}{cc}
\sin(\xi_{i+1}h_{i}) &  \frac{c_{i}\xi_{i}}{c_{i+1}\xi_{i+1}}\cos(\xi_{i+1}h_{i})\\
-\cos(\xi_{i+1}h_{i}) &  \frac{c_{i}\xi_{i}}{c_{i+1}\xi_{i+1}}\sin(\xi_{i+1}h_{i})
\end{array}\!\!\right)
\end{equation*}
pour $ i=0,\ldots, N-1,$ v\'erifient
\begin{equation}
\label{equation-conc-conf-fonctionnonguidee-Nsauts:9}
\Vert B_{i+1,i}\Vert= \max\left(1, \frac{c_{i+1}\xi_{i+1}}{c_{i}\xi_{i}}\right), \Vert B_{i,i+1}\Vert = \max\left(1,\frac{c_{i}\xi_{i}}{c_{i+1}\xi_{i+1}}\right).
\end{equation}
\end{monlem}
\begin{proof} Avec  $v=(v_{1}, v_{2})$ on a $\Vert B_{i+1,i}v\Vert^{2} = v_{1}^{2} + (\frac{c_{i+1}\xi_{i+1}}{c_{i+1}\xi_{i}})^{2}v_{2}^{2}\leq\max\left(1, (\frac{c_{i+1}\xi_{i+1}}{c_{i}\xi_{i}})^{2}\right)(v_{1}^{2} + v_{2}^{2}).$
% si on pose d'o\`{u} $\Vert B_{i+1,i}v\Vert\leq \max\left(1, \frac{c_{i+1}\xi_{i+1}}{c_{i}\xi_{i}}\right).$ 
 Pour l'\'egalit\'e, on prend $v=(1,0)$ si $\frac{c_{i+1}\xi_{i+1}}{c_{i}\xi_{i}}\leq 1$ ou $v=(0,1)$ si $\frac{c_{i+1}\xi_{i+1}}{c_{i}\xi_{i}}\geq 1.$
\end{proof}
%%%%%%%%%%%%%%%%%%%%%%%%%%%%%%%%
\begin{monlem}
\label{lemme-conc-conf-fonctionnonguidee-Nsauts:2}
Si $\lambda_{k,\ell}>c_{N}\frac{k^{2}\pi^{2}}{L^{2}}$ ($c_{H}=c_{N}$), on a
\begin{eqnarray}
\label{equation-conc-conf-fonctionnonguidee-Nsauts:10}
\sqrt{\frac{c_{i}}{c_{i+1}}} & < & \frac{c_{i}\xi_{i}}{c_{i+1}\xi_{i+1}}, i = 0,\ldots,N-1,\nonumber\\
\label{equation-conc-conf-fonctionnonguidee-Nsauts:10bis}
\frac{c_{i}\xi_{i}}{c_{i+1}\xi_{i+1}}& < & \sqrt{\frac{c_{i}(c_{N}-c_{i})}{c_{i+1}(c_{N}-c_{i+1})}}, i = 0,\ldots,N-2.\\
\label{equation-conc-conf-fonctionnonguidee-Nsauts:10ter}
\Vert B_{i+1,i}\Vert & \leq & \sqrt{\frac{c_{i+1}}{c_{i}}}, i = 0, N-1.
\end{eqnarray}
\end{monlem}
\begin{proof} Pour la premi\`{e}re in\'egalit\'e, il suffit de remarquer que $(\frac{c_{i}\xi_{i}}{c_{i+1}\xi_{i+1}})^{2} = \frac{c_{i}}{c_{i+1}}\frac{\lambda-c_{i}\frac{k^{2}\pi^{2}}{L^{2}}}{\lambda-c_{i+1}\frac{k^{2}\pi^{2}}{L^{2}}}$ o\`{u} le deuxi\`{e}me facteur est toujours plus grand que 1. La seconde in\'egalit\'e se v\'erifie par un calcul simple.
%Comme l'hyperbole d\'ecrite par l'application $\lambda\to\frac{\lambda-c_{i}\frac{k^{2}\pi^{2}}{L^{2}}}{\lambda-c_{i+1}\frac{k^{2}\pi^{2}}{L^{2}}}$ est d\'ecroissante sur l'intervalle consid\'er\'e, on en d\'eduit la seconde in\'egalit\'e. 
La derni\`{e}re in\'egalit\'e est une cons\'equence de la premi\`{e}re in\'egalit\'e et de \eqref{equation-conc-conf-fonctionnonguidee-Nsauts:9}.
\end{proof} 
%%%%%%%%%%%%%%%%%%%%%%%%%%%%%%%%%%%%%
\begin{monlem}
\label{lemme-conc-conf-fonctionnonguidee-Nsauts:3} Si $\lambda_{k,\ell}>c_{N}\frac{k^{2}\pi^{2}}{L^{2}},$ on a\\
\begin{equation*}
\vert a_{0}\vert\leq \sqrt{\frac{c_{i}}{c_{0}}}\;\left\Vert\left(\begin{array}{c}
a_{i}\\
b_{i}
\end{array}\right)\right\Vert, i = 0,\ldots,N,
\end{equation*}
et il existe une constante $M>0$ telle que
\begin{equation}
\label{equation-conc-conf-fonctionnonguidee-Nsauts:11}
\left\Vert \left(\begin{array}{c}
a_{i}\\
b_{i}
\end{array}\right)\right\Vert\leq M \vert a_{0}\vert 
, i = 0,\ldots, N-1.
\end{equation}
\end{monlem}
\begin{proof}  $\!\!\left(\!\begin{array}{cc}
\sin(\xi_{i+1}h_{i}) & \cos (\xi_{i+1}h_{i})\\
\cos(\xi_{i+1}h_{i}) & -\sin(\xi_{i+1}h_{i})
\end{array}\!\right)$ \'etant unitaire, on a $\vert a_{0}\vert\leq \Vert B_{1,0}\Vert\cdots \Vert B_{i,i-1}\Vert\; \left\Vert \left(\begin{array}{c}
a_{i}\\
b_{i}
\end{array}\right)\right\Vert$ d'o\`{u} $\vert a_{0}\vert\leq \sqrt{\frac{c_{i}}{c_{0}}}\;\left\Vert\left(\begin{array}{c}
a_{i}\\
b_{i}
\end{array}\right)\right\Vert, i = 0,\ldots,N$ en utilisant \eqref{equation-conc-conf-fonctionnonguidee-Nsauts:10ter}. La matrice $\left(
\begin{array}{cc}
\sin(\xi_{i}h_{i}) & \cos (\xi_{i}h_{i})\\
\cos(\xi_{i}h_{i}) & -\sin(\xi_{i}h_{i})
\end{array}\right)$ \'etant aussi unitaire, on a $\Vert \left(\begin{array}{c}
a_{i}\\
b_{i}
\end{array}\right)\Vert\leq \Vert B_{i-1,i}\Vert\cdots \Vert B_{0,1}\Vert \;\vert a_{0}\vert$ ce qui permet de dire, utilisant \eqref{equation-conc-conf-fonctionnonguidee-Nsauts:10bis},  que les vecteurs $\left(\begin{array}{c}
a_{i}\\
b_{i}
\end{array}\right)$ sont born\'es, pour $i=0,\ldots, N-1,$ par $M\vert a_{0}\vert$ o\`{u} $M$ est une constante 
ne d\'ependant pas de la valeur propre $\lambda_{k,\ell}>c_{N}\frac{k^{2}\pi^{2}}{L^{2}}.$ 
\end{proof}
On voudrait que \eqref{equation-conc-conf-fonctionnonguidee-Nsauts:11} reste vraie pour $i =N$ mais $\Vert B_{N-1,N}\Vert= \max(1,\frac{c_{N-1}\xi_{N-1}}{c_{N}\xi_{N}})$ tendra vers l'infini s'il existe une suite infinie de valeurs propres $(\lambda_{k_{n},\ell_{n}})_{n},$ telle que $ \lambda_{k_{n},\ell_{n}}-c_{N}\frac{k_{n}^{2}\pi^{2}}{L^{2}}\to 0^{+}.$
Pour cette raison, nous consid\'ererons les valeurs propres satisfaisant $\lambda_{k,\ell}>(c_{N}+\varepsilon)\frac{k^{2}\pi^{2}}{L^{2}}.$ En effet, avec ce choix\footnote{Le choix $\lambda_{k,\ell}>c_{N}\frac{k^{2}\pi^{2}}{L^{2}}+\varepsilon$ ne conviendrait pas.}, on a $\frac{c_{N-1}\xi_{N-1}}{c_{N}\xi_{N}}\leq \sqrt{\frac{c_{N-1}(c_{N} - c_{N-1} +\varepsilon)}{c_{N}\varepsilon}}$ ce qui permet d'affirmer que les couples $(a_{i}, b_{i}), i = 0,\ldots, N,$ sont comparables les uns avec les autres ($b_{0}=0$) et nous obtenons la
%%%%%%%%%%%%%%%%%%%%%%%%%%%%%%%%%%%%%
%%
\begin{mapropo}
\label{proposition-conc-conf-fonctionnonguidee-Nsauts-mod:1}
Si $\varepsilon>0,$ il existe une constante $M_{\varepsilon}>0$ telle que
\begin{equation*}
\label{equation-conc-conf-fonctionnonguidee-Nsauts-mod:3}
\left\Vert \left(\!\begin{array}{c}
a_{i}\\
b_{i}
\end{array}\!\right)\right\Vert\leq M_{\varepsilon} \vert a_{0}\vert \leq M_{\varepsilon}\sqrt{\frac{c_{i}}{c_{0}}}\left\Vert
\left(\begin{array}{c}
a_{i}\\
b_{i}
\end{array}\!\right)\right\Vert, i = 0,\ldots, N, \forall \lambda_{k,\ell}>(c_{N}+\varepsilon)\frac{k^{2}\pi^{2}}{L^{2}}.
\end{equation*}
\end{mapropo}
\noindent
qui a pour cons\'equence \eqref{equation-conc-conf-fonctionnonguidee:1}.
\noindent
En effet, on commence par noter que
\begin{enumerate}
\item nous pouvons \'ecrire $a_{i}\sin(\xi_{i}x) + b_{i}\cos(\xi_{i}x)  = \sqrt{a_{i}^{2} + b_{i}^{2}} \cos(\xi_{i}x-\beta_{i})$ o\`{u} $0\leq \beta_{i}<2\pi$ lorsque $h_{i-1}<x<h_{i};$
\item il existe au moins un indice $j, 0\leq j\leq N,$ tel que $\omega_{j}:=\omega\cap ((l_{1},l_{2})\times(h_{j},h_{j+1}))= (l_{1},l_{2})\times (d,d')$ avec $d'-d\geq \frac{b-a}{N+1}.$ 
\item $\int_{d}^{d'}u^{2}(x){\rm d}x = \frac{a_{j}^{2} + b_{j}^{2}}{2}(d'-d)\lbrack 1 + \frac{\sin(\xi_{j}(d'-d))}{\xi_{j}(d'-d)}\cos(\xi_{i}(d'+d)-2\beta_{i})\rbrack.$
\end{enumerate}
On pose $\bar{M}= \max_{\xi_{i}} \vert\frac{\sin(\xi_{i}(d'-d))}{\xi_{i}(d'-d)}\vert$ et comme $\xi_{j}>\sqrt{\frac{\varepsilon}{c_{j}}}\frac{k\pi}{L}\geq \sqrt{\frac{\varepsilon}{c_{N}}}\frac{k\pi}{L}$ on a $\bar{M}<1$ ce qui permet de conclure puisque $\int_{d}^{d'}u^{2}(x){\rm d}x \geq (1-\bar{M})\frac{a_{j}^{2} + b_{j}^{2}}{2}(d'-d).$
%%%%%%%%%%%%%%%%%%%%%%%%%%%%%%%%%%%%%%%%%%%%%%%%%%%%
\subsection{Mod\`{e}le $C^{1}$}
\label{soussection-con-conf-fonctionnonguidee-C1:1}
%%%%%%%%%%%%%%%%%%%%%%%%%%%%%%%%%%%%%%%%%%%%%%%%%%%%
Lorsque le coefficient $c$ a un peu de r\'egularit\'e on peut se ramener \`{a} la forme $F'' + Q F = 0.$ Par exemple, si le coefficient $c$ est deux fois d\'erivable, le changement de fonction $F= \frac{u}{\sqrt{c}}$ donne
\begin{equation*}
\label{equation-con-conf-fonctionnonguidee-C1:2}
F'' + \frac{c'^{2} - 2cc'' + 4 (\lambda -c\frac{k^{2}\pi^{2}}{L^{2}})c}{4 c^{2}} F= 0,
\end{equation*}
ce qui donne \`{a} penser que la solution $u$ aura un comportement sinuso\"{\i}dal pour $\lambda >\max_{x}(c\frac{k^{2}\pi^{2}}{L^{2}} + \frac{cc''}{2})$. Intuitivement, on serait dans la situation non guid\'ee de la Figure \ref{fig:figure1}.\\
Nous d\'ecomposons la preuve en quatre \'etapes.\\
$\bullet$ {\bf \'Etape 1 : Se ramener \`{a} $F'' + QF = 0.$}\\
La transformation de Liouville demande le changement de variable $s = \int_{0}^{x}\frac{1}{c(t)}{\rm d}t.$ 
Cette application, $g:x\to s,$ est positive, croissante et de classe $C^{1}$. Sa r\'eciproque $g^{-1}$ est aussi positive, croissante et de classe $C^{1}.$ Posant $F(s)= u(g^{-1}(s)),$ on v\'erifie que
\begin{equation}
\label{equation-con-conf-fonctionnonguidee-C1:3}
\begin{array}{l}
F'(s) = u'(g^{-1}(s))c(g^{-1}(s))\nonumber\\
F''(s) = u''(g^{-1}(s))(c(g^{-1}(s)))^2 + u'(g^{-1}(s))c'(g^{-1}(s))c(g^{-1}(s))
\end{array}
\end{equation}
ce qui implique $ u''(g^{-1}(s))c(g^{-1}(s)) = \frac{F''(s)}{c(g^{-1}(s))} - u'(g^{-1}(s))c'(g^{-1}(s)).$ Comme $cu'' + c'u' +(\lambda- c \frac{k^{2}\pi^{2}}{L^{2}}) u = 0$, on a
\begin{equation}
\label{equation-con-conf-fonctionnonguidee-C1:4}
\begin{array}{l}
F'' + Q F = 0.\nonumber\\
Q (s) = (\lambda -\frac{k^{2}\pi^{2}}{L^{2}}c\circ g^{-1}(s))c\circ g^{-1} (s)
\end{array}
\end{equation}
Pour ces op\'erations on a suppos\'e que $c$ est d\'erivable. La fonction $Q$ est positive d\`{e}s que $\lambda$ est assez grand. \\
%%%%%%%%%%%%%%%%%%%%%%%%%%%%%%%%%%%%%%%%%%%%%%%%%%%%%
 %%%%%%%%%%%%%%%%%%%%%%%%%%%%%%%%%%%%%%%%%%%%%%%%%%%%
{$\bullet$ \bf \'Etape 2 : Utiliser la transformation de Pr\"{u}fer modifi\'ee}\\
La r\'esolution de $F'' + Q F = 0$ se ram\`{e}ne \`{a} la r\'esolution des deux \'equations diff\'erentielles du premier ordre suivantes (cf. \cite{Z:1})
\begin{equation}
\label{equation-con-conf-fonctionnonguidee-C1:5}
\begin{array}{lll}
\left\lbrace\begin{array}{l}
\phi'(s)  =  -Q^{1/2} -\frac{1}{4}\frac{Q'}{Q}\sin(2\phi)\\
\frac{1}{R}R'(s))  =  \frac{1}{4}\frac{Q'}{Q}\cos(2\phi).
\end{array}\right.
\Longrightarrow
\left\lbrace\begin{array}{l}
F(s)  =  R(s)Q^{-1/4}\cos(\phi(s)),\\
F'(s)  =  R(s)Q^{1/4}\sin(\phi(s)).
\end{array}
\right.
\end{array}
\end{equation}
 %%%%%%%%%%%%%%%%%%%%%%%%%%%%%%%%%%%%%%%%%%%%%%%%%%%%
{$\bullet$ \bf \'Etape 3 : Propri\'et\'es de l'amplitude $R$ et de la phase $\phi$}\\
Posant
$\bar{c}(s):= c(g^{-1}(s)), \bar{c}_{M} := \int_{0}^{H}\frac{1}{c(t)}{\rm d}t,$
nous avons $F'' + Q F =  0$ sur l'intervalle $(0, \bar{c}_{M}).$ On note $c_{0}:= c(0)$ et $c_{H}:=c(H).$
\begin{monlem}
\label{lemme-conc-conf-fonctionnonguidee-C1:1}
Soit $\varepsilon >0$ et supposons que $c\in C^{1}(\lbrack 0,H\rbrack).$ On a
\begin{enumerate}
\item 
\begin{equation}
\label{equation-con-conf-fonctionnonguidee-C1:6}
\begin{array}{l}
(\lambda\geq c_{H}\frac{k^{2}\pi^{2}}{L^{2}})\Longrightarrow c^{2}_{0}(\frac{\lambda}{c(H)}-\frac{k^{2}\pi^{2}}{L^{2}})\leq Q \leq c^{2}_{H}(\frac{\lambda}{c_{0}}-\frac{k^{2}\pi^{2}}{L^{2}})\\
(\lambda\geq (c_{H}+\varepsilon)\frac{k^{2}\pi^{2}}{L^{2}})\Longrightarrow\left\vert\frac{Q'}{Q}(t)\right\vert \leq \frac{\bar{c}'}{\bar{c}}(1 + \frac{\bar{c}}{\varepsilon})
\end{array}
\end{equation}
\item Si $\lambda\geq (c_{H}+\varepsilon)\frac{k^{2}\pi^{2}}{L^{2}},$ les solutions $\phi$ et $R$ des \'equations diff\'erentielles \eqref{equation-con-conf-fonctionnonguidee-C1:5} v\'erifient 
\begin{eqnarray}
\label{equation-con-conf-fonctionnonguidee-C1:7}
\phi(s) & \sim& -\int_{0}^{s}Q^{1/2}(t){\rm d}t \mbox{ quand }(\frac{\lambda}{\bar{c}}-\frac{k^{2}\pi^{2}}{L^{2}})\to\infty\\
R(s) & = & C e^{\frac{1}{4}\int_{0}^{s}\frac{Q'}{Q}\cos(2\phi){\rm d}r}, C\not=0.\nonumber
\end{eqnarray}
\end{enumerate}
\end{monlem}
\begin{proof}
De $Q'(s) = \bar{c}'(\lambda - 2\bar{c}\frac{k^{2}\pi^{2}}{L^{2}})$ on voit que $ \frac{Q'}{Q} = \frac{\bar{c}'}{\bar{c}} (\frac	{\lambda -2\bar{c}\frac{k^{2}\pi^{2}}{L^{2}}}{\lambda -\bar{c}\frac{k^{2}\pi^{2}}{L^{2}}}) = \frac{\bar{c}'}{\bar{c}}(1 - \frac{\bar{c}\frac{k^{2}\pi^{2}}{L^{2}}}{\lambda - \bar{c}\frac{k^{2}\pi^{2}}{L^{2}}})$ et comme, avec l'hypoth\`{e}se faite, on a $\lambda -  \bar{c}\frac{k^{2}\pi^{2}}{L^{2}}\geq \varepsilon\frac{k^{2}\pi^{2}}{L^{2}}$ on arrive \`{a} $\frac{\bar{c}\frac{k^{2}\pi^{2}}{L^{2}}}{\lambda - \bar{c}\frac{k^{2}\pi^{2}}{L^{2}}}\leq \frac{\bar{c}}{\varepsilon}.$ Si on avait suppos\'e $\lambda\geq (2c_{H}+\varepsilon)\frac{k^{2}\pi^{2}}{L^{2}}$ on aurait $0< \frac{Q'}{Q}(s)< \frac{\bar{c}'}{\bar{c}}.$ Noter que pour \eqref{equation-con-conf-fonctionnonguidee-C1:7} la condition $(\frac{\lambda}{\bar{c}}-\frac{k^{2}\pi^{2}}{L^{2}})\to\infty$ est remplie si $k\to\infty$ mais aussi quand $\lambda = \lambda_{k,\ell}$ avec $\ell\to\infty.$
\end{proof}
L'amplitude $R$ n'est jamais nulle sinon elle le serait identiquement. En imposant $\phi(0) = 0$ ou $\phi(0) = \pi/2$ on obtient deux solutions lin\'eairement ind\'ependantes
 d'o\`{u} la solution cherch\'ee correspond \`{a} $\phi(0) = \frac{\pi}{2}$ (\`{a} un multiple pr\`{e}s)\footnote{Chaque valeur propre de cet op\'erateur est simple : si $u_{1}$ et $u_{2}$ sont deux fonctions propres associ\'ees \`{a} $\lambda$, le wronskien $W = cu'_{1}u_{2}-cu'_{2}u_{1}$ a une d\'eriv\'ee nulle sur $(0,H)$ ce qui montre que les vecteurs $(u_{1},cu'_{1})$ et $(u_{2},cu'_{2})$ sont colin\'eaires.}. On a donc trouv\'e la forme \'ecrite de la fonction propre correspondant \`{a} la valeur propre $\lambda\geq (c_{H}+\varepsilon)\frac{k^{2}\pi^{2}}{L^{2}})$ :
\begin{equation*}
\label{equation-con-conf-fonctionnonguidee-C1:8}
u(x) = R(g(x))(Q(g(x))^{-1/4} \cos(\phi(g(x)))\mbox{ avec } g(x) = \int_{0}^{x}\frac{1}{c(t)}{\rm d}t.
\end{equation*}
 %%%%%%%%%%%%%%%%%%%%%%%%%%%%%%%%%%%%%%%%%%%%%%%%%%%%
{$\bullet$ \bf \'Etape 4 : Estimation de $\Vert u\Vert^{2}_{L^{2}}$}
\begin{moncor}
\label{corollaire-conc-conf-fonctionnonguidee-C1:1}
Soient $\varepsilon>0$ et un ouvert $\omega\subset\Omega.$ Il existe une constante $a_{\varepsilon;\omega}>0,$ ind\'ependante par translation de $\omega$ dans $\Omega,$ telle
\begin{equation*}
\label{equation-con-conf-fonctionnonguidee-C1:9}
a_{\varepsilon;\omega}\leq \frac{\int_{\omega}\vert u(x)\vert^{2}{\rm d}x}{\int_{\Omega}\vert u(x)\vert^{2}{\rm d}x}\leq 1
\end{equation*}
pour la famille des fonctions propres associ\'ees aux valeurs propres $\lambda >(c_{H} + \varepsilon)\frac{k^{2}\pi^{2}}{L^{2}},$ hormis un ensemble fini.
\end{moncor}
\begin{proof}
Il vient du lemme \ref{lemme-conc-conf-fonctionnonguidee-C1:1} que le module de la fonction $R$ est encadr\'e
%born\'e inf\'erieurement et sup\'erieurement
 par deux constantes strictement positives : $r_{1}<\vert R\vert<r_{2}$ puisque $0<s<\bar{c}_{M}.$ Joignant ce r\'esultat \`{a} l'encadrement de $Q$  de \eqref{equation-con-conf-fonctionnonguidee-C1:6} on peut affirmer que pour $0<a<b<H$
\begin{multline*}
\label{equation-con-conf-fonctionnonguidee-C1:10}
r_{1}^{2}c_{H}^{-1}(\frac{\lambda}{c_{0}}-\frac{k^{2}\pi^{2}}{L^{2}})^{-1/2}\int_{a}^{b}\cos^{2}(\phi(g(x))){\rm d}x\leq\int_{a}^{b}(R(g{\large {\large }}(x)))^{2} (Q(g(x)))^{-1/2}\cos^{2}(\phi(g(x))){\rm d}x\\
\leq r_{2}^{2}c_{0}^{-1}(\frac{\lambda}{c_{H}}-\frac{k^{2}\pi^{2}}{L^{2}})^{-1/2}\int_{a}^{b}\cos^{2}(\phi(g(x))){\rm d}x. 
\end{multline*}
Gr\^{a}ce \`{a} l'encadrement similaire pour $\int_{0}^{H}(R(g(x)))^{2} (Q(g(x)))^{-1/2}\cos^{2}(\phi(g(x))){\rm d}x$
 %ce qui veut dire qu'
 il suffit de regarder le rapport\footnote{D\`{e}s que $2\varepsilon<c_{H}-c_{0}$, il existe $0<d<\frac{c_{0}}{c_{H}}$  tel que $d<(\frac{\lambda}{c_{H}}-\frac{k^{2}\pi^{2}}{L^{2}})({\frac{\lambda}{c_{0}}-\frac{k^{2}\pi^{2}}{L^{2}}})^{-1}<1$}$\int_{a}^{b}\cos^{2}(\phi(g(x))){\rm d}x/\int_{0}^{H}\cos^{2}(\phi(g(x))){\rm d}x.$ Comme $\phi'$ est n\'egative pour les valeurs propres consid\'er\'ees, \`{a} un nombre fini pr\`{e}s, la fonction $x\to \phi(g(x))$ est strictement d\'ecroissante. Avec le nouveau changement de variable $t = \phi(g(x))$, on a ${\rm d}x = \frac{c(g^{-1}(\phi^{-1}(t)))}{\phi'(\phi^{-1}(t))}{\rm d}t$ ce qui donne
\begin{equation*}
\label{equation-con-conf-fonctionnonguidee-C1:11}
\int_{a}^{b}\cos^{2}(\phi(g(x))){\rm d}x = \int_{\phi(g(b))}^{\phi(g(a))} \frac{c(g^{-1}(\phi^{-1}(t)))}{-\phi'(\phi^{-1}(t))}\cos^{2}t {\rm d}t.
\end{equation*}
o\`{u} on contr\^{o}le tr\`{e}s bien $c(g^{-1}(\phi^{-1}(t)))$. Combinant $-\phi' = Q^{1/2} +\frac{1}{4}\frac{Q'}{Q}\sin(2\phi)$ et \eqref{equation-con-conf-fonctionnonguidee-C1:6} on a 
%Comme $\lambda >(c_{H} + \varepsilon)\frac{k^{2}\pi^{2}}{L^{2}}$, on a 
\begin{equation}
\label{equation-con-conf-fonctionnonguidee-C1:12}
(1-\epsilon_{1})Q^{1/2}\leq -\phi' \leq(1+\epsilon_{1})Q^{1/2}
\end{equation}
 %\`{a} l'exception \'eventuelle d'
 sauf, peut-\^{e}tre, un nombre fini de valeurs propres $\lambda \in \Lambda$ (le cardinal de $\Lambda$ d\'epend de $\varepsilon)$ et, par suite, 
\begin{multline*}
\label{equation-con-conf-fonctionnonguidee-C1:13}
\frac{M}{1+\epsilon_{1}} \int_{\phi(g(b))}^{\phi(g(a))}\frac{\cos^{2}t}{Q^{1/2}(t)}{\rm d}t  \leq \int_{a}^{b}\cos^{2}(\phi(g(x))){\rm d}x \leq  \frac{M}{1-\epsilon_{1}} \int_{\phi(g(b))}^{\phi(g(a))}\frac{\cos^{2}t}{Q^{1/2}(t)}{\rm d}t\\
\frac{M(1-\epsilon_{2})}{c_{H}(\frac{\lambda}{c_{0}}- \frac{k^{2}\pi^{2}}{L^{2}})^{\frac{1}{2}}} \int_{\phi(g(b))}^{\phi(g(a))}\cos^{2}t{\rm d}t  \leq \int_{a}^{b}\cos^{2}(\phi(g(x))){\rm d}x  \leq 
 \frac{M(1+\epsilon_{2})}{c_{0}(\frac{\lambda}{c_{H}}- \frac{k^{2}\pi^{2}}{L^{2}})^{\frac{1}{2}}} \int_{\phi(g(b))}^{\phi(g(a))}\cos^{2}t{\rm d}t.
\end{multline*}
Dans l'int\'egrale $\int_{\phi(g(b))}^{\phi(g(a))}\cos^{2}t{\rm d}t = \frac{1}{2}\lbrack \phi(g(a)) -\phi(g(b)) +\frac{\cos(2(\phi(g(a))))  -  \cos(2(\phi(g(b))))}{2}\rbrack$ le dernier terme du crochet est major\'e par 1. Pour $\phi(g(a)) -\phi(g(b))  = (g(b)-g(a))(-\phi'(d_{1}))$ on a vu que l'on contr\^{o}le $\phi'(d_{1})$ ind\'ependamment de la position de $a$ et $b$ (utiliser \eqref{equation-con-conf-fonctionnonguidee-C1:6} et \eqref{equation-con-conf-fonctionnonguidee-C1:12}). Il reste \`{a} constater que $(g(b)-g(a))= (b-a)\frac{1}{d_{2}}$ pour un r\'eel $\frac{1}{c_{H}}<d_{2}<\frac{1}{c_{0}}$ ce qui montre que c'est la diff\'erence $(b-a)$ qui importe. On fait de m\^{e}me pour $\int_{\phi(g(H))}^{\phi(g(0))}\cos^{2}t{\rm d}t$ ce qui permet de conclure pour un parall\'el\'epip\`{e}de puisque $\int_{l_{1}}^{l_{2}}\sin^{2}(k\pi x_{1}){\rm d}x_{1} = \frac{l_{1} -l_{2}}{2} -\lbrack \frac{\sin(2 k\pi x_{1})}{4k\pi}\rbrack_{l_{1}}^{l_{2}}.$ Il n'y a pas de difficult\'e \`{a} g\'en\'eraliser ceci \`{a} n'importe quel ouvert $\omega$ puisque ce dernier contient toujours un parall\'el\'epip\`{e}de.
\end{proof}
%%%%%%%%%%%%%%%%%%%%%%%%%%%%%%%%%%%%%%%%%%%%%%%%%%%%
%%%%%%%%%%%%%%%%%%%%%%%%%%%%%%%%%%%%%%%%%%%%%%%%%%%%
\section{Lemme \ref{lemme-conc-conf-fonctionguidee-1saut:2} : indications}
\setcounter{equation}{0}
\label{annexe-Lemme preparatoire:1}
\begin{maremarque}
\label{remarque:3}
Comme les valeurs propres guid\'ees $\lambda_{k,\ell}$ v\'erifient $k^{2}\pi^{2}< \lambda_{k,\ell} < c_{1}k^{2}\pi^{2}$, elles sont comparables \`{a} $k^{2}\pi^{2}.$
\end{maremarque}
\vskip-.3cm
\noindent
$\bullet$ Les notations d\'efinies en \eqref{notation:1} seront parfois \'etendues au cours du texte en y rempla\c{c}ant la valeur propre $\lambda_{k,\ell}$ par la variable $\lambda$ parcourant un intervalle de $\mathbb{R}$ mais aucune confusion ne sera possible.\\
$\bullet$ L'entier $k$ est provisoirement fix\'e, la constante $c_{1}$ est suppos\'ee strictement sup\'erieure \`{a} $c_{0}=1$. On a $\mathcal{L}_{k}$ valeurs propres $\lambda_{k,\ell}$ class\'ees dans l'ordre croissant $k^{2}\pi^{2}< \lambda_{k,1}< \lambda_{k,2}<\cdots<\lambda_{k,\mathcal{L}_{k}}<c_{1}k^{2}\pi^{2}.$ Elles correspondent aux abscisses des points d'intersection de la courbe repr\'esentant la fonction croissante $\lambda\mapsto\frac{1}{\xi'_{1}}\tanh(\xi'_{1}(H-h_{0}))$ avec la courbe $\lambda \mapsto -\frac{1}{\xi_{0}}\tan(\frac{\xi_{0}}{2})$ o\`{u}, ici, $\xi_{0}= (\lambda-k^{2}\pi^{2})^{1/2},  \xi'_{1}= (k^{2}\pi^{2}-\frac{\lambda}{c_{1}})^{1/2}.$ La figure 3 les repr\'esente o\`{u} on a pos\'e
\begin{equation}
\label{notation:2}
\lambda_{\ell}= \lambda_{\ell}(k) = k^{2}\pi^{2} + (2\ell-1)^{2}\frac{\pi^{2}}{4h_{0}^{2}}, \lambda^{\ell}=\lambda^{\ell}(k):= k^{2}\pi^{2} + \ell^{2}\frac{\pi^{2}}{h_{0}^{2}}, 1\leq \ell\leq \mathcal{L}_{k}+1.
\end{equation}
Maintenant, nous nous int\'eressons \`{a} l'intervalle $I_{k,\ell}:=(\lambda_{\ell}, \lambda_{\ell+1})$ et on voit que les deux courbes sont convexes sur $(\lambda_{\ell},\lambda^{\ell}).$ Avec le changement de variable $I_{k,\ell}\ni \lambda \to \mu = \lambda -k^{2}\pi^{2} - (2\ell -1)^{2}\frac{\pi^{2}}{4h_{0}^{2}},$ la nouvelle variable $\mu$ parcourt l'intervalle $ (0, 2\ell\frac{\pi^{2}}{h_{0}^{2}}),$ et $\lambda_{\ell}, \lambda^{\ell}, \lambda_{k,\ell}, (\lambda_{\ell} + \lambda_{\ell+1})/2$ correspondent respectivement \`{a} $ \mu_{\ell} =0, \mu^{\ell} =(4\ell-1)\frac{\pi^{2}}{4h_{0}^{2}},\mu_{k,\ell}, (\mu_{\ell} + \mu_{\ell+1})/2 =  \ell\frac{\pi^{2}}{h_{0}^{2}}$. La convexit\'e des deux courbes sur $(\mu_{\ell},\mu^{\ell})$ permet d'introduire la valeur $\tilde{\mu}$ point d'intersection entre la courbe en $\tanh$ et la droite qui est tangente en $\mu^{\ell}$ \`{a} l'autre courbe. Cette droite d\'efinie par $\mu\to -\frac{h_{0}^{3}}{2\ell^{2}\pi^{2}}(\mu-\mu^{\ell})$ permet de voir que $0<\mu^{\ell}-\mu_{k,\ell}<\mu^{\ell} -\tilde{\mu}.$\\
 \begin{minipage}[t]{155mm}
\begin{wrapfigure}{r}{4.7cm}
\includegraphics[scale=.7]{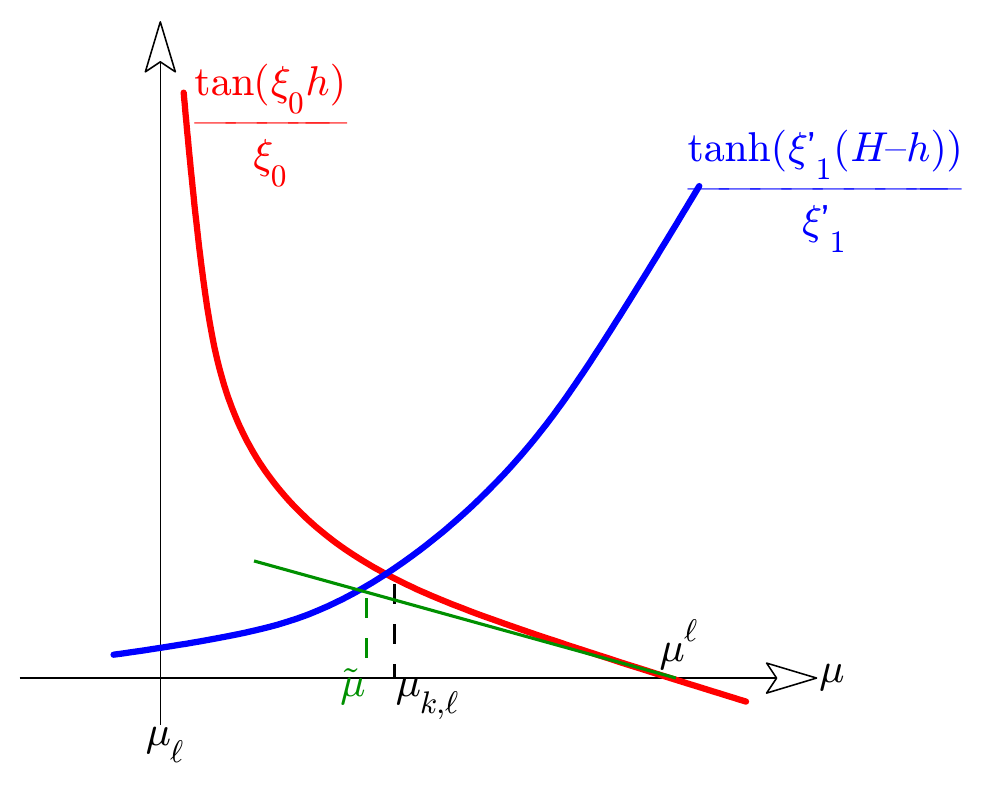}\hfill
\end{wrapfigure}
{\bf 1/}  La quantit\'e $\xi_{0}$ devient $\sqrt{\mu + (2\ell -1)^{2}\frac{\pi^{2}}{4h_{0}^{2}}}.$\\
{\bf 2/} La restriction de la courbe $\lambda \mapsto -\frac{1}{\xi_{0}}\tan(\xi_{0}h_{0})$ \`{a} l'intervalle $I_{k,\ell}$\\ est ind\'ependante de l'indice $k$ et correspond maintenant \`{a}
$$\mu\to -\frac{\tan(\sqrt{\mu+(2\ell-1)^{2}\frac{\pi^{2}}{4h_{0}^{2}}}h_{0})}{\sqrt{\mu+(2\ell-1)^{2}\frac{\pi^{2}}{4h_{0}^{2}}}}.$$ 
D'autre part,  la courbe en tangente hyperbolique d\'epend de $k$\\
 et s'\'ecrase de plus en plus vers l'axe des abscisses lorsque $k\to \infty,$\\
 ce qui prouve que la suite $(\mu_{k,\ell})_{k}$ tend vers $\mu^{\ell}$ quand $k\to\infty,$\\ l'indice $\ell$ restant fixe.
%Si $\mu^{\ell}$ correspond \`{a} $\lambda^{\ell}$ et $\mu_{k,\ell}$ \`{a} $\lambda_{k,\ell},$ on obtient
\end{minipage}\\

Lorsque $k\to\infty,$ on v\'erifie que
%\vskip2cm
\begin{eqnarray}
\label{equation-intersection-1saut:1}
\frac{\tanh(\xi'_{1}(H-h_{0}))}{\xi'_{1}} &=&\sqrt{\frac{c_{1}}{c_{1}-1}}\frac{1}{k\pi}\left( 1 + \frac{1}{2(c_{1}-1)k^{2}}(\frac{\mu-\mu^{\ell}}{\pi^{2}}+\frac{\ell^{2}}{h_{0}^{2}}) +o(\frac{1}{k^{2}})\right),
%-\frac{\tan(\xi_{0}h)}{\xi_{0}} &=& \frac{h_{0}^{3}}{2\ell^{2}\pi^{2}}(\mu^{\ell}-\mu) +o(\mu^{\ell}-\mu),\nonumber\\
%\mu^{\ell} -\mu_{k,\ell} &=&  2\frac{\ell^{3}\pi^{2}}{h_{0}^{3}}\sqrt{\frac{c_{1}}{c_{1}-1}}\frac{1}{k} + o(\frac{1}{k}).\nonumber
\end{eqnarray}
et ainsi la valeur $\tilde{\mu}$ est solution de
\begin{equation*}
\label{equation-intersection-1saut:2}
\sqrt{\frac{c_{1}}{c_{1}-1}}\frac{1}{k\pi}\left( 1 + \frac{1}{2(c_{1}-1)k^{2}}(\frac{\mu-\mu^{\ell}}{\pi^{2}}+\frac{\ell^{2}}{h_{0}^{2}}) +o(\frac{1}{k^{2}})\right) = -\frac{h_{0}^{3}}{2\ell^{2}\pi^{2}}(\mu-\mu^{\ell}),
\end{equation*}
ce qui donne
\begin{equation}
\label{equation-intersection-1saut:3}
\mu^{\ell}-\tilde{\mu} = 2\sqrt{\frac{c_{1}}{c_{1}-1}}\frac{\ell^{2}}{h_{0}^{3}}\frac{\pi}{k}\left(1 + \frac{\ell^{2}}{2(c_{1}-1)k^{2}h_{0}^{2}} +
o(\frac{1}{k^{2}})\right) = O(\frac{1}{k}).
\end{equation}
Toujours lorsque $k\to \infty,$ on v\'erifie que $-\frac{\tan(\sqrt{\mu+(2\ell-1)^{2}\frac{\pi^{2}}{4h_{0}^{2}}}h_{0})}{\sqrt{\mu+(2\ell-1)^{2}\frac{\pi^{2}}{4h_{0}^{2}}}} = \frac{h_{0}^{3}}{2\ell^{2}\pi^{2}}(\mu^{\ell}-\mu) +o(\mu^{\ell}-\mu)$. On \'ecrit que cette expression est \'egale au second membre de \eqref{equation-intersection-1saut:1}, ce qui est exactement la relation de dispersion \eqref{dispersionguidee-1saut:1} et donne
\begin{equation}
\label{equation-intersection-1saut:4}
\mu^{\ell} -\mu_{k,\ell} = 2\frac{\ell^{2}\pi}{h_{0}^{3}}\sqrt{\frac{c_{1}}{c_{1}-1}}\frac{1}{k} + o(\frac{1}{k^{2}}) + o(\mu^{\ell}-\mu_{k,\ell})
\end{equation}
Utilisant \eqref{equation-intersection-1saut:3}, on peut r\'esoudre la difficult\'e de la pr\'esence simultan\'ee de $o(\frac{1}{k^{2}})$ et $o(\mu^{\ell}-\mu_{k,\ell})$ dans \eqref{equation-intersection-1saut:4}, ce qui donne
\begin{equation*}
\label{equation-intersection-1saut:5}
\mu^{\ell}-\mu_{k,\ell} = 2\frac{\ell^{2}\pi}{h_{0}^{3}}\sqrt{\frac{c_{1}}{c_{1}-1}}\frac{1}{k} + o(\frac{1}{k}).
\end{equation*}
Il n'y a plus qu'\`{a} revenir aux notations en $\lambda$ d'o\`{u} $ \lambda_{k,\ell}^{\frac{1}{2}} = k\pi (1 + \frac{\ell^{2}}{2h_{0}^{2}}\frac{1}{k^{2}} + o(\frac{1}{k^{2}}))$ tandis que la valeur de $\xi'_{{1}_{\vert\lambda_{k,\ell}}}$ est donn\'ee par $\xi'_{1} = \sqrt{\frac{c_{1}}{c_{1}-1}}k\pi\left(1- \frac{\ell^{2}}{2h_{0}^{2}(c_{1}-1)k^{2}} + o(\frac{1}{k^{2}})\right).$
%%%%%%%%%%%%%%%%%%%%%%%%%%%%%%%%%%%%%%%%%%%%%%%%%%%%
%%%%%%%%%%%%%%%%%%%%%%%%%%%%%%%%%%%%%%%%%%%%%%%%%%%%
\section{Proposition \ref{proposition-conc-conf-fonctionguidee-1saut:1}
 : indications}
\label{annexe-comportement-asymptotique-fonctionpropreguidee1indications}
\setcounter{equation}{0}
\begin{maremarque}
\label{remarque:4}
Choisir $h_{0}=\frac{1}{2}$ n'est pas restrictif. En effet, partons des trois \'el\'ements : l'ouvert $\Omega = (0,1)\times(0,H)$ avec une interface en $x_{2} =h$, le coefficient de diffusion $c$ et le couple $(A,D(A))$ introduits dans la Section \ref{section-introduction} et posons
\begin{equation}
\label{applicationunitaire:1}
\begin{array}{c}
\tilde{\Omega}= (0,1)\times(0,\tilde{H}), \Omega = (0,1)\times(0,H), 2h\tilde{H}=H,\\
\varphi:\tilde{\Omega}\mapsto\Omega, \varphi(\tilde{x}) = \left\lbrace\begin {array}{l}
(\tilde{x}_{1},2h\tilde{x}_{2})\mbox{ si }0<\tilde{x}_{2}<\frac{1}{2},\\
(\tilde{x}_{1},2h(\tilde{x}_{2}-\frac{1}{2}) +h\mbox{ si }\frac{1}{2}<\tilde{x}_{2}<\tilde{H},
\end{array}
\right.\\
c(x)=(2h)^{2}\tilde{c}(\varphi^{-1}(x)),\\
U: L^{2}(\tilde{\Omega}, \tilde{c}^{-1}{\rm d}\tilde{x})\mapsto L^{2}(\Omega, \bar{c}^{-1}{\rm d}x), (U\tilde{f})(x) = \sqrt{2h}\tilde{f}(\varphi^{-1}(x)),
\end{array}
\end{equation}
on v\'erifie que 
\begin{enumerate}
\item $U$ est une application unitaire entre ces deux espaces d'Hilbert;
\item l'interface de $\Omega$ est ramen\'ee de $x_{2} =h$ \`{a} $\tilde{x}_{2} = \frac{1}{2};$
\item $U^{-1}AU = \tilde{A}$ o\`{u} le couple $(\tilde{A},D(\tilde{A})$ est d\'efini par $\tilde{A}=-\tilde{c}\Delta$ et $D(\tilde{A}) := \{\tilde{u}\in H^{1}_{0}(\tilde{\Omega}); \tilde{A}\tilde{u}\in L^{2}(\tilde{\Omega})\}$.
\end{enumerate}
 Comme $c_{\vert \tilde{\Omega}_{0}}$ n'est pas \'egal \`{a} 1 mais \`{a} $(2h)^{2}$ il faut corriger les valeurs propres par ce coefficient multiplicatif mais les fonctions propres sont inchang\'ees.
 \end{maremarque}
 \label{remarque:6}
Les quantit\'es $\xi'_{1}(k,\ell), 1\leq \ell\leq \mathcal{L}_{k},$ sont les indicateurs du taux de d\'ecroissance des fonctions propres guid\'ees dans la partie $\Omega_{1}$ et elles v\'erifient
$
%\begin{equation}
%\label{xi-1:1}
\xi'_{1}(k,1)>\xi'_{1}(k,2)>\cdots >\xi'_{1}(k,\mathcal{L}_{k})>0.
%\end{equation}
$
\begin{maremarque}
Pour $h=\frac{1}{2}, c_{1}=2,$ des calculs assez simples aboutissent aux in\'egalit\'es plus pr\'ecises suivantes 
\begin{equation*}
%\label{estimationguidee:0}
\begin{array}{c}
\pi<\xi_{0}(k,1)<\xi_{0}(k,2)<\cdots <\pi(k-2)<\xi_{0}(k,\mathcal{L}_{k})<\pi k\\
\frac{\pi}{\sqrt{2}}\sqrt{k^{2}-1}>\xi'_{1}(k,1)>\xi'_{1}(k,2)>\cdots >\sqrt{2}\pi\sqrt{k-1}>\xi'_{1}(k,\mathcal{L}_{k})>\frac{\pi}{\sqrt{2}}\sqrt{k-\frac{1}{2}}.
\end{array}
\end{equation*}
Dans ce cas particulier, la quantit\'e $\xi'_{1}(k,\mathcal{L}_{k})$ est de l'ordre de $\sqrt{k}$ donc du m\^{e}me ordre que $\lambda_{k,\mathcal{L}_{k}}^{\frac{1}{4}}$ mais ce r\'esultat ne semble pas exact dans le cas g\'en\'eral, sans l'exclure cependant.
\end{maremarque}
%%%%%%%%%%%%%%%%%%%%%%%%%%%%%%%%%%%%%%%%%%%%%%%%%%%%

%%%%%%%%%%%%%%%%%%%%%%%%%%%%%%%%%%%%%%%%%%%%%%%%%%%%
Dans la suite on n'impose plus $c_{1}=2$ mais, comme on suppose $h_{0}=\frac{1}{2},$ on a
\begin{equation}
\label{notation:2bis}
\lambda_{\ell}= \lambda_{\ell}(k) = k^{2}\pi^{2} + (2\ell-1)^{2}\pi^{2}, \lambda^{\ell}=\lambda^{\ell}(k):= k^{2}\pi^{2} + 4\ell^{2}\pi^{2}, 1\leq \ell\leq \mathcal{L}_{k}+1.
\end{equation}
\begin{figure}
\hspace{-1cm}\includegraphics[scale=.45]{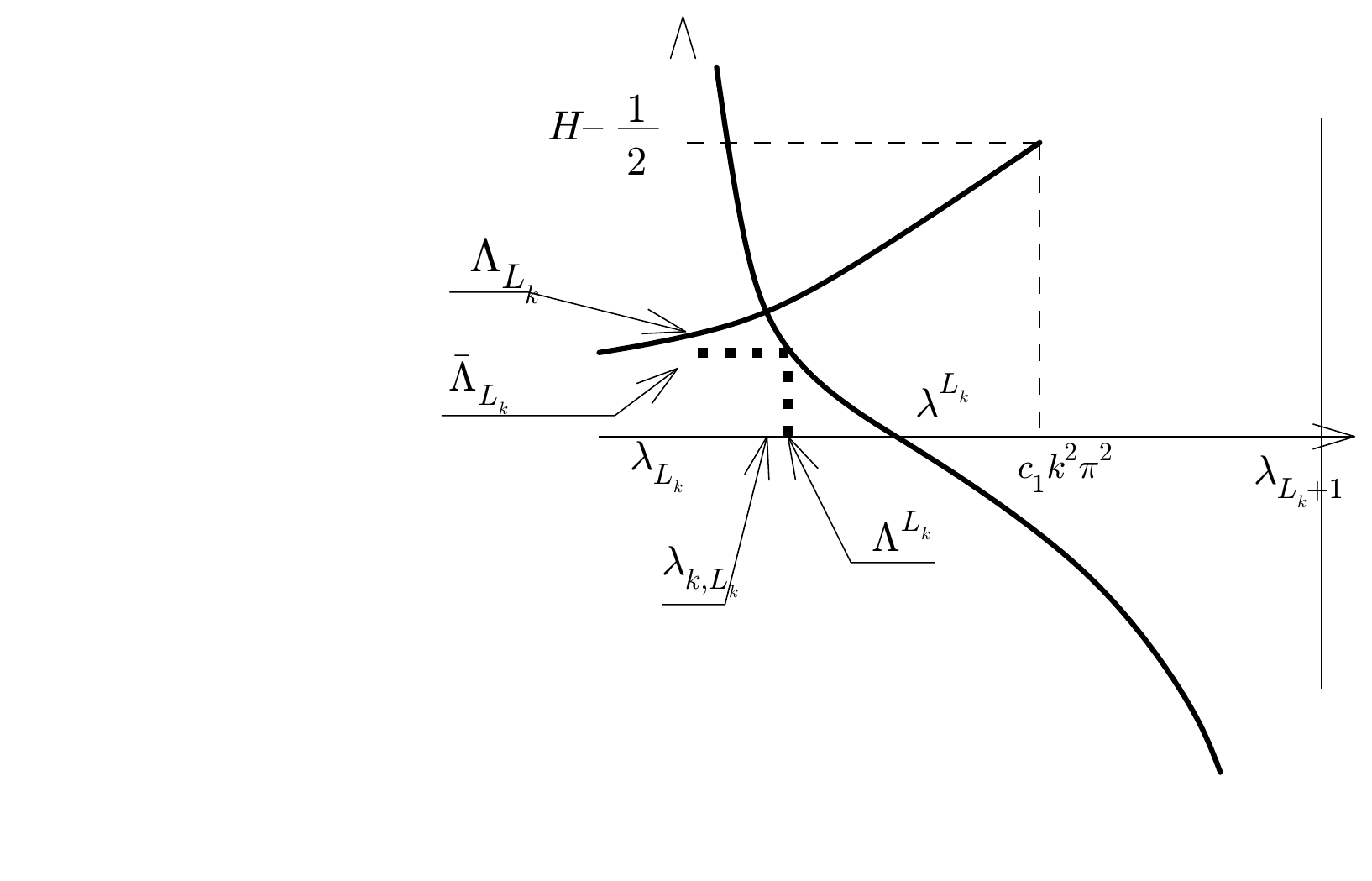}\hspace{1cm}\includegraphics[scale=.45]{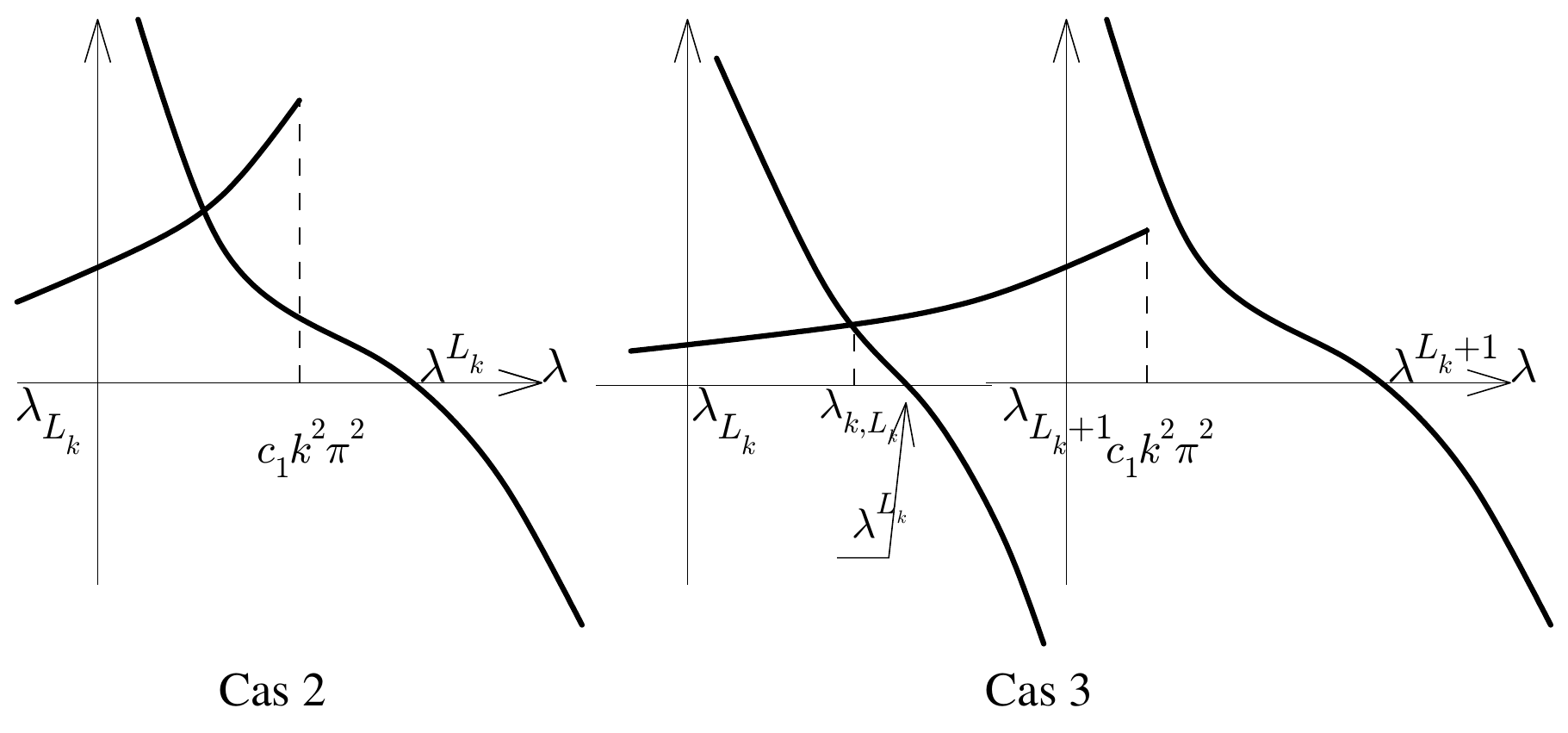}\hfill
\label{fig:figure4}
\caption{{\bf Cas 1} \`{a} gauche, {\bf Cas 2} au milieu, {\bf Cas 3} \`{a} droite}
\end{figure}
%Appelant $\bar{\mathcal{L}}_{k}$ le plus petit entier v\'erifiant $c_{1}k^{2}\pi^{2}\leq \lambda_{\bar{\mathcal{L}}_{k}+ 1}, $ 
On est confront\'e \`{a} trois possibilit\'es pour situer $\mathcal{L}_{k}$ (le plus grand entier tel que $\lambda_{k,\mathcal{L}_{k}}<c_{1}^{2}k^{2}\pi^{2}$), dans les intervalle $I_{k,\ell} =(\lambda_{\ell}, \lambda_{\ell+1})$ :
\begin{enumerate}
\item {\bf Cas 1},  $\lambda^{\mathcal{L}_{k}}\leq c_{1}k^{2}\pi^{2}\leq \lambda_{\mathcal{L}_{k}+1}$  (repr\'esent\'e \`{a} gauche dans la Figure 4) d'o\`{u} $\mathcal{L}_{k}\leq \frac{\sqrt{c_{1}-1}}{2}k\leq \mathcal{L}_{k}+\frac{1}{2}$ ;
 \item {\bf Cas 2}, i.e. $\lambda_{\mathcal{L}_{k}}<c_{1}k^{2}\pi^{2}<\lambda^{\mathcal{L}_{k}}$ et comme les deux courbes ont un point commun dans cet intervalle (centre de la Figure 4) on a $\mathcal{L}_{k}-\frac{1}{2}< \frac{\sqrt{c_{1}-1}}{2}k< \mathcal{L}_{k}$  et $0<-\frac{\tan(\frac{1}{2}\sqrt{c_{1}-1}k\pi)}{\sqrt{c_{1}-1}k\pi}<H-\frac{1}{2}$ ;
\item {\bf Cas 3}, i.e. $\lambda_{\mathcal{L}_{k}+1}<c_{1}k^{2}\pi^{2}<\lambda^{\mathcal{L}_{k}+1}$ et  comme ces deux courbes n'ont pas de point commun dans cet intervalle (\`{a} droite de la Figure 4) on a $\mathcal{L}_{k}+\frac{1}{2}< \frac{\sqrt{c_{1}-1}}{2}k< \mathcal{L}_{k}+1$ et $H-\frac{1}{2}<-\frac{\tan(\frac{1}{2}\sqrt{c_{1}-1}k\pi)}{\sqrt{c_{1}-1}k\pi}.$
%Ce dernier cas ressemble alors au premier, moyennant le d\'ecalage d'une unit\'e.
\end{enumerate}
\begin{maremarque}
\label{remarque:5} On voit ainsi que les entiers $k$ et $\mathcal{L}_{k}$ sont comparables ($k \approxeq \mathcal{L}_{k})$. De plus, pour le Cas 1 le nombre r\'eel $\frac{\sqrt{c_{1}-1}}{2}k$ appartient \`{a}  la premi\`{e}re moiti\'e d'un intervalle entre deux entiers, i.e.
 $\frac{\sqrt{c_{1}-1}}{2}k\in \lbrack n, n+\frac{1}{2}\rbrack$ (intervalle ferm\'e) alors que dans les deux autres cas il est dans la seconde moiti\'e, i.e. $\frac{\sqrt{c_{1}-1}}{2}k\in (n+\frac{1}{2}, n)$ (intervalle ouvert), les positions relatives de $H-\frac{1}{2}$ et $\frac{\sqrt{c_{1}-1}}{2}k$ distinguant ces deux cas entre eux.
 \end{maremarque}
%%%%%%%%%%%%%%%%%%%%%%%%%%%%%%%%%%%%%%%%%%%%%%%%%%%%
\begin{monlem}
\label{lemme-comportement-asymptotique-fonctionpropreguidee:3}
On consid\`{e}re la suite des valeurs propres $(\lambda_{k,\mathcal{L}_{k}})_{k\geq1}.$ Chacun des entiers $k$ et $\mathcal{L}_{k}$ est comparable \`{a} $\lambda_{k,\mathcal{L}_{k}}^{\frac{1}{2}}$. De plus,
\begin{enumerate}
\item {\bf Cas} 1 et 3 : les quantit\'es $k,\mathcal{L}_{k},\lambda^{\frac{1}{2}}_{k,\mathcal{L}_{k}}$ et $(\xi'_{1})^{2}$ sont comparables entre elles;
\item {\bf Cas} 2 : les quantit\'es $k,\mathcal{L}_{k},\lambda^{\frac{1}{2}}_{k,\mathcal{L}_{k}}$ sont comparables entre elles et $\frac{(\xi'_{1})^{2}}{\lambda^{\frac{1}{2}}_{k,\mathcal{L}_{k}}} = O(1).$
\end{enumerate}
\end{monlem}
\begin{proof} L'encadrement le plus large est $\lambda_{\mathcal{L}_{k}}< \lambda_{k,\mathcal{L}_{k}}<\lambda_{\mathcal{L}_{k}+1}$ et comme $k$ et $\mathcal{L}_{k}$ sont comparables on en d\'eduit la premi\`{e}re phrase de l'\'enonc\'e. Nous \'etudions les trois cas s\'epar\'ement.\\
{\bf\large Cas 1} $\lambda^{\mathcal{L}_{k}}\leq c_{1}k^{2}\pi^{2}\leq \lambda_{\mathcal{L}_{k}+1}.$\\
Dans cette situation, $\mathcal{L}_{k}$ est le plus grand entier tel que $\lambda^{\mathcal{L}_{k}}\leq c_{1}k^{2}\pi^{2}$, i.e. $k^{2}\pi^{2}+ 4\mathcal{L}_{k}^{2}\pi^{2}\leq c_{1}k^{2}\pi^{2}$ d'o\`{u}  $\mathcal{L}_{k}=E\left(k\frac{\sqrt{c_{1}-1}}{2}\right)$ ($E(x)$ d\'esigne la partie enti\`{e}re du nombre r\'eel $x$), i.e. $k\frac{\sqrt{c_{1}-1}}{2} =\mathcal{L}_{k} + \alpha$ avec $0\leq \alpha <1.$
%Par suite, $\lambda_{k,\mathcal{L}_{k}}$ est ainsi comparable \`{a} $\mathcal{L}_{k}^{2}$ puisque $\lambda_{\mathcal{L}_{k}}<\lambda_{k,\mathcal{L}_{k}}<\lambda^{\mathcal{L}_{k}}.$ 
Pour pr\'eciser $\lambda_{k,\mathcal{L}_{k}}$ nous proc\'edons en deux \'etapes :
\begin{enumerate}
\item[(a)] Commen\c{c}ons par \'evaluer la valeur $\Lambda_{\mathcal{L}_{k}}$ de $\lambda\mapsto\frac{1}{\xi'_{1}}\tanh(\xi'_{1}(H-\frac{1}{2}))$ en $\lambda = \lambda_{\mathcal{L}_{k}}.$ En ce point on a $(\xi'_{1})^{2} = \frac{\pi^{2}}{c_{1}}((c_{1}-1)k^{2} -(2\mathcal{L}_{k}-1)^{2})$. Comme $(c_{1}-1)k^{2} = 4(\mathcal{L}_{k}+\alpha)^{2}$, on obtient $\xi'_{1} = 2\pi\sqrt{\frac{\mathcal{L}_{k}}{c_{1}}(1+2\alpha + \frac{(2\alpha -1)(1+2\alpha)}{4L})}$, d'o\`{u}
\begin{equation}
\label{xi-1:2}
\Lambda_{\mathcal{L}_{k}} = \frac{\tanh\left( 2\pi\sqrt{\frac{\mathcal{L}_{k}}{c_{1}}(1+2\alpha + \frac{(2\alpha -1)(1+2\alpha)}{4\mathcal{L}_{k}})}(H-\frac{1}{2})\right)}{2\pi\sqrt{\frac{\mathcal{L}_{k}}{c_{1}}\left(1+2\alpha + \frac{(2\alpha -1)(1+2\alpha)}{4\mathcal{L}_{k}}\right)}}.
\end{equation}
Pour $\mathcal{L}_{k}$ grand, donc $k$ grand, la majoration $\Lambda_{\mathcal{L}_{k}}>\frac{1}{4\pi}\sqrt{\frac{c_{1}}{L}}$ est vraie.
\item[(b)] Nous \'evaluons la valeur $\bar{\Lambda}$ de la fonction $\lambda \mapsto -\frac{1}{\xi_{0}}\tan(\frac{\xi_{0}}{2})$ en $\Lambda^{\mathcal{L}_{k}} = \frac{\lambda_{\mathcal{L}_{k}} + \lambda^{\mathcal{L}_{k}}}{2},$ milieu des abscisses $\lambda_{\mathcal{L}_{k}}$ et $\Lambda^{\mathcal{L}_{k}}.$ On a $\Lambda^{\mathcal{L}_{k}} = k^{2}\pi^{2} + (4\mathcal{L}_{k}^{2} -\frac{4\mathcal{L}_{k}-1}{2})\pi^{2}$ et la valeur de $\xi_{0}$ en ce point est \'egale \`{a} $2\pi \mathcal{L}_{k}(1-\frac{4\mathcal{L}_{k}-1}{8\mathcal{L}_{k}^{2}})^{1/2}.$ Quand $\mathcal{L}_{k}\to\infty, \frac{\xi_{0}}{2}\sim \pi \mathcal{L}_{k}- \frac{\pi}{4}$ et
\begin{equation}
\label{xi-1:3}
\bar{\Lambda}_{\mathcal{L}_{k}} =-\frac{\tan\left(\frac{ \xi_{0}}{2}\right)}{\xi_{0}}\sim \frac{1}{2\pi \mathcal{L}_{k}} < \Lambda_{\mathcal{L}_{k}}.
\end{equation}
\end{enumerate}
Cela signifie que la valeur propre $\lambda_{k,\mathcal{L}_{k}}$ satisfait 
$\lambda_{\mathcal{L}_{k}} < \lambda_{k,\mathcal{L}_{k}}< \Lambda^{\mathcal{L}_{k}}$ pour $k$ grand, ce qui implique que
\begin{equation}
\label{xi-1:4}
c_{1}(\xi'_{1}(k,\mathcal{L}_{k}))^{2}> c_{1}k^{2}\pi^{2}- \frac{\lambda_{\mathcal{L}_{k}} +\lambda^{\mathcal{L}_{k}}}{2} = c_{1}k^{2}\pi^{2} -\lambda^{L} + \frac{\lambda^{\mathcal{L}_{k}} -\lambda_{\mathcal{L}_{k}}}{2}> \frac{\lambda^{\mathcal{L}_{k}} -\lambda_{\mathcal{L}_{k}}}{2} = \frac{4\mathcal{L}_{k}-1}{2}\pi^{2}\to\infty.
\end{equation}
Comme $\xi'_{1} = 2\pi\sqrt{\frac{\mathcal{L}_{k}}{c_{1}}(1+2\alpha + \frac{(2\alpha -1)(1+2\alpha)}{4\mathcal{L}_{k}})},$ valeur de $\xi'_{1}$ en $\lambda_{\mathcal{L}_{k}},$ est sup\'erieur \`{a} la valeur de $\xi'_{1}$ en $\lambda_{k,\mathcal{L}_{k}}$,  on d\'eduit de \eqref{xi-1:4} que $\xi'_{1}(k,\mathcal{L}_{k})$ est comparable \`{a} $\sqrt{\mathcal{L}_{k}}$ d'o\`{u} \`{a} $\lambda_{k,\mathcal{L}_{k}}^{\frac{1}{4}}$.\\
%%%%%%%%%%%%%%%%%%%%%%%%%%%%%%%%%%%%%%%%%%%%%%%%%%%%
{\bf\large Cas 3} $\lambda_{L_{k} +1}<c_{1}k^{2}\pi^{2}<\lambda^{L_{k} +1},$ sans intersection\\
Pour ce cas (Figure 4, \`{a} droite) il vient 
\begin{equation}
\label{xi-1:5}
(4\mathcal{L}_{k}+1)\pi^{2}=\lambda_{\mathcal{L}_{k}+1}-\lambda^{\mathcal{L}_{k}} <c_{1}(\xi'_{1}(k,\mathcal{L}_{k}))^{2}= c_{1}k^{2}\pi^{2}-\lambda_{k,\mathcal{L}_{k}}< \lambda^{\mathcal{L}_{k}+1} -\lambda_{\mathcal{L}_{k}} =3(4\mathcal{L}_{k}+1)\pi^{2}.
\end{equation}
ce qui montre que $\xi'_{1}(k,\mathcal{L}_{k})$ est comparable \`{a} $\sqrt{\mathcal{L}_{k}}$ d'o\`{u} \`{a} $\lambda_{k,\mathcal{L}_{k}}^{\frac{1}{4}}$.\\
%%%%%%%%%%%%%%%%%%%%%%%%%%%%%%%%%%%%%%%%%%%%%%%%%%%%
{\bf\large Cas 2} $\lambda_{\mathcal{L}_{k}}<c_{1}k^{2}\pi^{2}<\lambda^{\mathcal{L}_{k}},$ avec intersection\\
On peut d\'ej\`{a} majorer $\xi'_{1}(k,\mathcal{L}_{k})$. En effet, de $c(\xi'_{1})^{2} = c_{1}k^{2}\pi^{2} -\lambda_{k,\mathcal{L}_{k}}<c_{1}k^{2}\pi^{2}-\lambda_{\mathcal{L}_{k}}= c_{1}k^{2}\pi^{2}-\lambda^{\mathcal{L}_{k}} + \lambda^{\mathcal{L}_{k}} -\lambda_{\mathcal{L}_{k}} <  \lambda^{\mathcal{L}_{k}} -\lambda_{\mathcal{L}_{k}} = (4\mathcal{L}_{k}-1)\pi^{2},$ on d\'eduit que $\xi'_{1}\leq \sqrt{\frac{4\mathcal{L}_{k}-1}{c_{1}}}\pi$. Cette derni\`{e}re quantit\'e est de l'ordre de $\lambda^{\frac{1}{4}}_{k,\mathcal{L}_{k}}$ d'o\`{u} l'existence d'une constante $C_{1}>0$ telle $ 0< \xi'_{1}(k,\mathcal{L}_{k})\leq C_{1}\lambda^{\frac{1}{4}}_{k,\mathcal{L}_{k}}.$
\end{proof}
Il faudrait am\'eliorer la minoration de $\xi'_{1}$ dans le Cas 2. Pour l'instant et pour ce cas, le taux de d\'ecroissance dans l'exponentielle de \eqref{equation-conc-conf-fonctionpropreguidee-1saut:3} est de l'ordre de $\lambda_{k,\mathcal{L}_{k}-1}^{\frac{1}{4}}$ pour une suite $(k_{n})$ convenable. Les trois situations existent vraiment comme le montre le
\begin{monlem}
\label{lemme-comportement-asymptotique-fonctionpropreguidee:4}
 Si $\frac{\sqrt{c_{1}-1}}{2} = \frac{p}{q}\in \mathbb{Q}, p,q>0,$ les cas 1 et 2 arrivent
\begin{enumerate}
\item[a-]
\label{item-cas1:1} Si $\frac{p}{q}$ est un entier alors on est toujours dans le Cas 1.
\item[b-] Si $ \frac{p}{q}\in \mathbb{Q}\setminus\mathbb{N}$, la suite $(k_{n})_{n}, k_{n} = nq,$ satisfait le Cas 1.
\item[c-] Si $\frac{p}{q}\leq\frac{1}{2},$ la suite $(k_{n})_{n}$ avec $k_{n} = nq +1$ satisfait le Cas 1. De plus, si $\frac{p}{q}$ est assez proche de 0 pour qu'il existe un entier $m$ tel que $\frac{1}{2}<m\frac{p}{q}<1$, la suite $(k_{n})_{n}$ avec $k_{n} = nq +m$ satisfait le Cas 2.
\item[d-] Si $\frac{1}{2}<\frac{p}{q}<1$ la suite $(k_{n})_{n}$ avec $ k_{n} = nq +1$ satisfait le Cas 2.
%si $\frac{p}{q}<1$ est proche de 1.
%\item[e-] Si $\frac{1}{2}<\frac{p}{q}<1$ la suite $(k_{n})_{n}, k_{n} = 2nq +1$ satisfait le Cas 3 si $\frac{p}{q}<1$ est proche de $\frac{1}{2}.$
\end{enumerate}
\end{monlem}
\begin{proof} La Remarque \ref{remarque:5} rend \'evidentes les deux premi\`{e}res affirmations. Pour la derni\`{e}re, il faut en plus remarquer que $-\frac{\tan(\frac{1}{2}\sqrt{c_{1}-1}k\pi)}{\sqrt{c_{1}-1}k\pi}= -\frac{q\tan(\frac{p}{q}\pi)}{2pk\pi}\downarrow 0$ quand $n\to\infty$ et donc deviendra inf\'erieure \`{a} $H-\frac{1}{2}.$ On conclue de m\^{e}me pour la troisi\`{e}me.
%Pour les deux derni\`{e}res, prendre la sous-suite $k_{n} = 2nq+1.$
\end{proof}
Prouvons que si $0< \alpha = \frac{\sqrt{c_{1}-1}}{2}\in \mathbb{R}\setminus \mathbb{Q},$ il existe une sous-suite d'entiers $(k_{n})_{n}$ pour lesquels on est dans les Cas 2 ou  3. On veut donc montrer que pour tout $n$ il existe un entier $k_{n}>n$ et un entier $L_{n}$ tel que $L_{n}+\frac{1}{2}<k_{n} \alpha<L_{n}+1.$ Proc\'edant par contradiction, nous devons montrer que la proposition
\begin{equation}
\label{doubleinegaliteanier:2}
\mbox{ Il existe } n\mbox{ tel que, pour tout } k>n, \mbox{ on n'a pas d'entier } L \mbox{ v\'erifiant } L\leq k\alpha\leq L+\frac{1}{2},
\end{equation}
est fausse. Nous la supposons exacte pour $0<\alpha<\frac{1}{2}$, et prenons $k>n.$ Il existe alors un entier $p\geq 1$ pour lequel $L + \frac{1}{2}<(k+p)\alpha<L+1$, ce qui contredit \eqref{doubleinegaliteanier:2}. Supposons maintenant exact \eqref{doubleinegaliteanier:2}
 pour un $\alpha = \frac{1}{2}+\beta, 0<\beta<\frac{1}{2}$ et prenons
 $k$ de la forme $k= 2 p, p\in \mathbb{N},$ d'o\`{u}, pour un certain $L,$ l'in\'egalit\'e $L\leq p + 2p\beta\leq L+\frac{1}{2}$ qui est impossible si $p\geq \frac{1}{2\beta}.$ La g\'en\'eralisation aux autres irrationnels est ais\'ee.
 %%%%%%%%%%%%%%%%%%%%%%%%%%%%%%%%%%%%%%%%%%%%%%%%%%%%
%%%%%%%%%%%%%%%%%%%%%%%%%%%%%%%%%%%%%%%%%%%%%%%%%%%%
%{\color{red}
\section{D\'etail des calculs pour la zone (I).}
\setcounter{equation}{0}
\label{annexe-detailcalculs-zone(1):1}
\setcounter{equation}{0}
Nous consid\'erons une fonction propre associ\'ee \`{a} une valeur propre $\lambda$ de la zone (I) : $c_{1}\frac{k^{2}\pi^{2}}{L^{2}}<\lambda<c_{2}\frac{k^{2}\pi^{2}}{L^{2}}.$ Les conditions de transmission aux interfaces $S_{0}, S_{1}$ s'\'ecrivent
\begin{eqnarray*}
\label{equation-annexe-detailcalculs-zone(1):1}
a_{1}\sin(\xi_{1}h_{0}) + b_{1}\cos(\xi_{1}h_{0}) & = & a_{0}\sin(\xi_{0}h_{0})\nonumber\\
c_{1}a_{1}\xi_{1}\cos(\xi_{1}h_{0}) - c_{1}\xi_{1}b_{1}\sin(\xi_{1}h_{0}) & = & c_{0}a_{0}\xi_{0}\cos(\xi_{0}h_{0})\nonumber\\
a_{1}\sin(\xi_{1}h_{1}) + b_{1}\cos(\xi_{1}h_{1}) & = & a_{2}\sinh(\xi'_{2}(H-h_{1}))\nonumber\\
c_{1}a_{1}\xi_{1}\cos(\xi_{1}h_{1}) - c_{1}\xi_{1}b_{1}\sin(\xi_{1}h_{1}) & = &- c_{2}a_{2}\xi'_{2}\cosh(\xi'_{2}(H-h_{1})),
\end{eqnarray*}
et la r\'esolution donne pour $(a_{1},b_{1})$
\begin{eqnarray}
a_{1} & = & a_{0}\frac{c_{1}\xi_{1}\sin(\xi_{0}h_{0})\sin(\xi_{1}h_{0}) + c_{0}\xi_{0}\cos(\xi_{0}h_{0})\cos(\xi_{1}h_{0}) }{c_{1}\xi_{1}},\nonumber\\
b_{1} & = & - a_{0}\frac{c_{0}\xi_{0}\cos(\xi_{0}h_{0})\sin(\xi_{1}h_{0}) - c_{1}\xi_{1}\sin(\xi_{0}h_{0})\cos(\xi_{1}h_{0}) }{c_{1}\xi_{1}},\nonumber\\
\label{equation-annexe-detailcalculs-zone(1):1'}
a_{1}^{2} + b_{1}^{2} & = & a_{0}^{2}\left( \sin^{2}(\xi_{0}h_{0}) + (\frac{c_{0}\xi_{0}}{c_{1}\xi_{1}})^{2}\cos^{2}(\xi_{0}h_{0})\right),
\end{eqnarray}
tandis que pour $a_{2}$
\begin{eqnarray*}
a_{1} & = & - a_{2}\frac{- c_{1}\xi_{1}\sinh(\xi'_{2}(H-h_{1}))\sin(\xi_{1}h_{1}) + c_{2}\xi'_{2}\cosh(\xi'_{2}(H-h_{1}))\cos(\xi_{1}h_{1}) }{c_{1}\xi_{1}},\nonumber\\
b_{1} & = & a_{2}\frac{c_{2}\xi'_{2}\sin(\xi_{1}h_{1})\cosh(\xi'_{2}(H-h_{1})) + c_{1}\xi_{1}\sinh(\xi'_{2}(H-h_{1}))\cos(\xi_{1}h_{1}) }{c_{1}\xi_{1}},\nonumber\\
a_{1}^{2} + b_{1}^{2} & = & a_{2}^{2}\left( \sinh^{2}(\xi'_{2}(H-h_{1})) + (\frac{c_{2}\xi'_{2}}{c_{1}\xi_{1}})^{2}\cosh^{2}(\xi'_{2}(H-h_{1}))\right),\nonumber\\
\label{equation-annexe-detailcalculs-zone(1):2}
a_{2}^{2} & = & a_{0}^{2}\frac{ \sin^{2}(\xi_{0}h_{0}) + (\frac{c_{0}\xi_{0}}{c_{1}\xi_{1}})^{2}\cos^{2}(\xi_{0}h_{0})}{\sinh^{2}(\xi'_{2}(H-h_{1})) + (\frac{c_{2}\xi'_{2}}{c_{1}\xi_{1}})^{2}\cosh^{2}(\xi'_{2}(H-h_{1}))}.
\end{eqnarray*}
{\bf Cas particulier 1 : $(c_{1} + \varepsilon)\frac{k^{2}\pi^{2}}{L^{2}}<\lambda<(c_{2}-\varepsilon)\frac{k^{2}\pi^{2}}{L^{2}}$, pour un $\varepsilon >0.$}\\
On ne consid\`{e}re ainsi que les valeurs propres $\lambda_{k,\ell}$ qui ne s'approchent pas trop pr\`{e}s des bords de la zone (I), au sens qui vient d'\^{e}tre pr\'ecis\'e, ce qui implique d'ailleurs $\xi'_{2}\to\infty.$ Avec ce cadre, on a 
\begin{eqnarray*}
\label{equation-annexe-detailcalculs-zone(1):3}
\frac{c_{2}\varepsilon}{c_{1}(c_{2}-c_{1}-\varepsilon)} &\leq \left( \frac{c_{2}\xi'_{2}}{c_{1}\xi_{1}}\right)^{2} \leq \frac{c_{2}(c_{2}-c_{1}-\varepsilon)}{c_{1}\varepsilon}\mbox{ et } 
\frac{c_{0}(c_{2}-c_{0}-\varepsilon)}{c_{1}(c_{2}-c_{1}-\varepsilon)} &\leq \left( \frac{c_{0}\xi_{0}}{c_{1}\xi_{1}}\right)^{2} \leq \frac{c_{0}(c_{1}-c_{0}+\varepsilon)}{c_{1}\varepsilon}.
\end{eqnarray*}
Ainsi, pour toute suite infinie de valeurs propres distinctes, les quantit\'es $a_{1}^{2} + b_{1}^{2}$ et $a_{0}^{2}$ sont comparables d'apr\`{e}s \eqref{equation-annexe-detailcalculs-zone(1):1'}. De plus, il existe deux constantes $M_{1},M_{2}>0$, d\'ependant de $\varepsilon,$ telles que 
\begin{equation*}
\label{equation-annexe-detailcalculs-zone(1):4}
M_{1} e^{-2\xi'_{2}(H-h_{1})}\leq( \frac{a_{2}}{a_{0}})^{2}\leq M_{2}e^{-2\xi'_{2}(H-h_{1})},
\end{equation*}
i.e. $(a_{2}/a_{0})^{2}\approxeq  e^{-2\xi'_{2}(H-h_{1})}$ avec la notation introduite \`{a} la fin de la section \ref{section-introduction}, d'o\`{u} \eqref{equation-2sauts-zone(I):2} et \eqref{equation-2sauts-zone(I):3}.\\
Nous consid\'erons maintenant deux autres cas particuliers dont on ne peut pas affirmer pour l'instant qu'ils existent. Pour les \'etudier, il est utile d'avoir en m\'emoire les relations suivantes
\begin{eqnarray}
\label{equation-annexe-detailcalculs-zone(1):5}
\!\!\!\!\!c_{2}(\xi'_{2})^{2} + c_{1}\xi_{1}^{2}  =  (c_{2}-c_{1})\frac{k^{2}\pi^{2}}{L^{2}},\,
c_{2}(\xi'_{2})^{2} + c_{0}\xi_{0}^{2}  =  (c_{2}-c_{0})\frac{k^{2}\pi^{2}}{L^{2}},\,
c_{0}\xi_{0}^{2} - c_{1}\xi_{1}^{2}  =  (c_{1}-c_{0})\frac{k^{2}\pi^{2}}{L^{2}},
\end{eqnarray}
et de poser ${\rm Num} =  \sin^{2}(\xi_{0}h_{0}) + (\frac{c_{0}\xi_{0}}{c_{1}\xi_{1}})^{2}\cos^{2}(\xi_{0}h_{0})$ et ${\rm Den}= \sinh^{2}(\xi'_{2}(H-h_{1})) + (\frac{c_{2}\xi'_{2}}{c_{1}\xi_{1}})^{2}\cosh^{2}(\xi'_{2}(H-h_{1})).$\\ 
{\bf Cas particulier 2 : $\xi'_{2}\to 0$.}\\
Alors, \eqref{equation-annexe-detailcalculs-zone(1):5} implique $\xi_{1}\to \infty$ et $\left( \frac{c_{0}\xi_{0}}{c_{1}\xi_{1}}\right)^{2}\sim \frac{c_{0}(c_{2}-c_{0})}{c_{1}(c_{2}-c_{1})}$ d'o\`{u} ${\rm Num}\approxeq 1$
%Il existe deux constantes $M_{1}, M_{2}>0$ telles que $M_{1}<N<M_{2}$ 
et ${\rm Den}= \sinh^{2}(\xi'_{2}(H-h_{1}))\left( 1 + \big(\frac{c_{2}\xi'_{2}}{c_{1}\xi_{1}}\big)^{2} \,\frac{\cosh^{2}(\xi'_{2}(H-h_{1}))}{\sinh^{2}(\xi'_{2}(H-h_{1}))}\right)\sim \sinh^{2}(\xi'_{2}(H-h_{1})).$ Par suite, on a $(\frac{a_{2}}{a_{0}})^{2}\approxeq (\xi'_{2}(H-h_{1}))^{-2}$ et $a_{0}^{2}\approxeq a_{1}^{2} + b_{1}^{2}$ d'o\`{u} $\int_{\Omega_{2}}u^{2}(x){\rm d}x\approxeq a_{0}^{2}(H-h_{1}), \int_{\Omega_{0}}u^{2}(x){\rm d}x\approxeq \frac{a_{0}^{2}}{2}h_{0}$ et   $\int_{\Omega_{1}}u^{2}(x){\rm d}x\approxeq \frac{a_{1}^{2}+b_{1}^{2}}{2}h_{0}.$ Sachant que $\frac{c_{0}\xi_{0}}{c_{1}\xi_{1}}\approxeq 1$, \eqref{equation-annexe-detailcalculs-zone(1):1'} implique $a_{1}^{2}+b_{1}^{2}\approxeq a_{0},$  ce qui prouve \eqref{equation-2sauts-zone(I):4}.
%$M_{1} a_{0}^{2} \leq a_{1}^{2} + b_{1}^{2} \leq M_{2} a_{0}^{2} .$
\\
{\bf Cas particulier 3 : $\xi_{1}\to 0.$}\\
\eqref{equation-annexe-detailcalculs-zone(1):5} implique d'une part $\xi'_{2}\to \infty$ et $\left( \frac{c_{2}\xi'_{2}}{c_{1}\xi_{1}}\right)^{2} \to \infty$ et, d'autre part, $ \xi_{0}\to\infty$ et $\left(\frac{c_{0}\xi_{0}}{c_{2}\xi'_{2}}\right)^{2} \sim \frac{c_{0}(c_{1}-c_{0})}{c_{1}(c_{2}-c_{1})}.$ Ainsi ${\rm Den}\sim \left( \frac{c_{2}\xi'_{2}}{c_{1}\xi_{1}}\right)^{2}\cosh^{2}(\xi'_{2}(H-h_{1}))$ ce qui donne
\begin{equation*}
\label{equation-annexe-detailcalculs-zone(1):6}
a_{2}^{2}\sim \frac{a_{0}^{2}}{\cosh^{2}(\xi'_{2}(H-h_{1}))}\left( \big(\frac{c_{1}\xi_{1}}{c_{2}\xi'_{2}}\big)^{2}\sin^{2}(\xi_{0}h_{0}) + \big(\frac{c_{0}\xi_{0}}{c_{2}\xi'_{2}}\big)^{2}\cos^{2}(\xi_{0}h_{0})\right)
\end{equation*}
d'o\`{u} $a_{2}^{2}\leq a_{0}^{2}\frac{M}{\cosh^{2}(\xi'_{2}(H-h_{1}))}$ pour une constante $M.$ 
\begin{enumerate}
\item Si $\cos(\xi_{0}h_{0})$ ne tend pas vers 0 alors $\tan(\xi_{0}h_{0})$ est born\'ee. Ainsi, le c\^{o}t\'e gauche de \eqref{suite(stratification3valeurs)equation:3} tendrait vers 0 tandis, que du c\^{o}t\'e droit, le d\'enominateur tendrait vers -1 et  le num\'erateur vers $(h_{1}-h_{0})/c_{1}$, d'o\`{u} une contradiction montrant que cette situation ne peut arriver.
%$a_{2}^{2}\approxeq \frac{a_{0}^{2}}{\cosh^{2}(\xi'_{2}(H-h_{1}))}$ mais $a_{1}^{2} + b_{1}^{2}\to \infty.$
\item Si $\cos(\xi_{0}h_{0})\to 0$ et $\vert \frac{\tan(\xi_{0}h_{0})}{c_{0}\xi_{0}}\vert \to\infty,$ on utilise encore \eqref{suite(stratification3valeurs)equation:3} : le membre de gauche tend vers 0 mais le c\^{o}t\'e droit \'equivaut \`{a} $\frac{1}{c_{1}\xi_{1}^{2}(h_{1}-h_{0})-\frac{c_{0}\xi_{0}}{\tan(\xi_{0}h_{0})}}\to \infty$ d'o\`{u} encore une contradiction.
\item Si $\cos(\xi_{0}h_{0})\to 0$ et $\frac{\tan(\xi_{0}h_{0})}{c_{0}\xi_{0}}=0(1)$ on a $\xi_{0} \cos(\xi_{0}h_{0})\nrightarrow 0$ et le d\'enominateur du c\^{o}t\'e droit de \eqref{suite(stratification3valeurs)equation:3} \'equivaut \`{a} -1.  Il faut donc que $\frac{\tan(\xi_{0}h_{0})}{c_{0}\xi_{0}} \sim -\frac{h_{1}-h_{0}}{c_{1}}.$ Cette \'eventualit\'e implique  $a_{1}^{2} + b_{1}^{2}\sim a_{0}^{2}\big(\frac{c_{0}\xi_{0}}{c_{1}\xi_{1}}\big)^{2}\cos^{2}(\xi_{0}h_{0})\to \infty$ d'o\`{u} concentration de la masse dans la zone interm\'ediaire $\Omega_{1}.$
\end{enumerate}
%%%%%%%%%%%%%%%%%%%%%%%%%%%%%%%%%%%%%%%%%%%%%%%%%%%%
%%%%%%%%%%%%%%%%%%%%%%%%%%%%%%%%%%%%%%%%%%%%%%%%%%%%
%%%%%%%%%%%%%%%%%%%%%%%%%%%%%%%%
%%%%%%%%%%%%%%%%%%%%%%%%%%%%%%%%%%%%%%%%%%%%%%%%%%%%
%\bibliography{Concentration-v1.bib}

\begin{thebibliography}{[13]}
%%%%%%%%%%%%%%%%%%%%%%%%%%%%%%%%%%%%%%%%%%%%%%%%%%%%
%%%%%%%%%%%%%%%%%%%%%%%%%%%%%%%%%%%%%%%%%%%%%%%%%%%%
\bibitem{AG:1} Agmon S. {\it Lectures on exponential decay of solutions of second order elliptic equations}, Mathematical Notes, Princeton University Press \& University of Tokyo Press, {\bf 29} (1982).
%%%%%%%%%%%%%%%%%%%%%%%%%%%%%%%%%%%%%%%%%%%%%%%%%%%%
\bibitem{AFM:1} Arnold D. N., David G., Filoche M., Jerison D. \& Mayboroda S. {\it Localization of eigenfunctions via an effective potential}, CPDE, {\bf 44} (2019), n$^0$11 : 1186-1216. arXiv:1712.02419v4 [math.AP]  3 Oct 2018.
%%%%%%%%%%%%%%%%%%%%%%%%%%%%%%%%%%%%%%%%%%%%%%%%%%%%
%%%%%%%%%%%%%%%%%%%%%%%%%%%%%%%%%%%%%%%%%%%%%%%%%%%%
\bibitem{Cristofol:1} Cristofol M. {\it Guided waves in a stratified elastic and locally perturbed space}, Math. Meth. Applied Sciences, {\bf 23} (2000), 1257-1286.
%%%%%%%%%%%%%%%%%%%%%%%%%%%%%%%%%%%%%%%%%%%%%%%%%%%%
\bibitem{DerGui:1} Dermenjian Y.\& Guillot J.C. {\it Scattering of elastic waves in a perturbed isotropic half space with a free boundary. The limiting absorption principle}, Math. Meth. Applied Sciences, {\bf 10} (1988), 87-124.
%%%%%%%%%%%%%%%%%%%%%%%%%%%%%%%%%%%%%%%%%%%%%%%%%%%%%
%\bibitem{DocManuscrit:1} Document manuscrit du 08/06/2018, corrig\'e et compl\'et\'e le 26/08/2018, 6p.
%%%%%%%%%%%%%%%%%%%%%%%%%%%%%%%%%%%%%%%%%%%%%%%%%%%%%
%%%%%%%%%%%%%%%%%%%%%%%%%%%%%%%%%%%%%%%%%%%%%%%%%%%%%
%\bibitem{DocManuscrit:3} Document manuscrit du 18/05/2018, corrig\'e et compl\'et\'e les 11/06/18 et 12/09/18, 7p.
%%%%%%%%%%%%%%%%%%%%%%%%%%%%%%%%%%%%%%%%%%%%%%%%%%%%
\bibitem{DoFe:1} Donnelly H. \& Fefferman C. {\it Nodal sets of eigenfunctions on Riemannian manifolds}, Inv. Math. {\bf 93} (1988), 161-183.
%%%%%%%%%%%%%%%%%%%%%%%%%%%%%%%%%%%%%%%%%%%%%%%%%%%%
\bibitem{EP:1} Epstein P.S. {\it Reflection of waves in an inhomogeneous absorbing medium}, Proc. Nat. Acad. Sci. U.S. {\bf 16} (1930), 627-637.
%%%%%%%%%%%%%%%%%%%%%%%%%%%%%%%%%%%%%%%%%%%%%%%%%%%%
\bibitem{JL:1} Jerison  D. \& Lebeau G.  {\it Nodal sets of sums of eigenfunctions} in {\it Harmonic analysis and partial differential equations (Chicago, Il. 1996)}, Chicago Lectures in Math., p. 223-239., Univ. Chicago Press, Chicago, Il., 1990.
%%%%%%%%%%%%%%%%%%%%%%%%%%%%%%%%%%%%%%%%%%%%%%%%%%%%
%%%%%%%%%%%%%%%%%%%%%%%%%%%%%%%%%%%%%%%%%%%%%%%%%%%%
\bibitem{LaLe:1} Laurent C. \& L\'eautaud M. {\it Tunneling estimates and approximate controllability for hypoelliptic equations}, accept\'e par Mem. Amer. Math. Soc., arXiv:1703.10797v1 [math.AP] du 31 mars 2017.
%%%%%%%%%%%%%%%%%%%%%%%%%%%%%%%%%%%%%%%%%%%%%%%%%%%%
\bibitem{LR:1} Lebeau G. \& Robbiano L.  {\it Contr\^{o}le exact de l'\'equation de la chaleur}, CPDE, {\bf 20} (1995) :335-356.
%%%%%%%%%%%%%%%%%%%%%%%%%%%%%%%%%%%%%%%%%%%%%%%%%%%%
%\bibitem{N-U:1} Nikiforov A. F. \& Uvarov V. B. {\it Special Functions of Mathematical Physics, A unified introduction with applications}, Birkh\"{a}user, Basel-Boston (1988).
%%%%%%%%%%%%%%%%%%%%%%%%%%%%%%%%%%%%%%%%%%%%%%%%%%%%
\bibitem{Ped-Wh:1} Pedersen M.A. \& White De Wayne, {\it Ray theory of the general Epstein profile}, J. Acoustic Soc. America, {\bf 44} (1968), n$^0$3, 765-786.
%%%%%%%%%%%%%%%%%%%%%%%%%%%%%%%%%%%%%%%%%%%%%%%%%%%%
\bibitem{Schu:1} Schulenberger J.R. {\it Elastic waves in the half space $\mathbb{R}^{2}_{+}$}, J. Differential Equations, {\bf 20} (1978), 405-438.
%%%%%%%%%%%%%%%%%%%%%%%%%%%%%%%%%%%%%%%%%%%%%%%%%%%%%
\bibitem{Wil:1} Wilcox C.H. {\it Sound propagation in stratified fluids}, Applied Mathematical Sciences, vol. 50, Springer-Verlag, Berlin, 1984.
%%%%%%%%%%%%%%%%%%%%%%%%%%%%%%%%%%%%%%%%%%%%%%%%%%%%%
\bibitem{Z:1} Zwillinger D. {\it Handbook of differential equations}, Academic Press, $3^{\rm rd}$ edition, 1998.
\end{thebibliography}
%\bibliographystyle{amsplain}

\end{document}